\documentclass[11pt,a4paper]{article}

\usepackage{cite}                                       
\usepackage{graphicx}        
\usepackage[latin1]{inputenc}
\usepackage[english]{babel}
\usepackage{amssymb,amsmath}
\usepackage{a4wide}
\usepackage{fancyhdr}
\usepackage{enumitem}
\usepackage{color}

\def\ph{\varphi}

\newcommand{\Dscr}{\mathcal{D}}

\newcommand{\Nscr}{\mathcal{N}}
\newcommand{\Oscr}{\mathcal{O}}
\newcommand{\Pscr}{\mathcal{P}}
\newcommand{\Rscr}{\mathcal{R}}
\newcommand{\Xscr}{\mathcal{X}}

\def\RR{\mathbb{R}}
\def\ZZ{\mathbb{Z}}
\def\NN{\mathbb{N}}
\def\Id{\mathbb I}

\def\lie#1{L_{#1}}
\def\chiph{\chi^{\phantom s}}
\def\gt{>}
\def\lt{<}
\def\le{\leq}
\def\ge{\geq}
\def\Chi{\Xscr}
\def\scal{\diamond}

\def\reali{\mathbb{R}}
\def\complessi{\mathbb{C}}

\def\nucleo{\Nscr^{(s)}}
\def\range{\Rscr^{(s)}}
\def\pspazio{\Pscr^{(s)}}

\newcommand{\modf}[1]{{\left\lfloor #1 \right\rfloor}}

\newcommand{\tond}[1]{{\left(#1\right)}}
\newcommand{\bigtond}[1]{{\bigl(#1\bigr)}}

\newcommand{\quadr}[1]{{\left[#1\right]}}

\newcommand{\inter}[1]{{\left\langle#1\right\rangle}}
\newcommand{\biginter}[1]{{\bigl\langle#1\bigr\rangle}}

\newcommand{\dinter}[1]
{{\left\langle\hskip-2pt\left\langle#1\right\rangle\hskip-2pt\right\rangle}}
\newcommand{\qed}{{\vskip-18pt\null\hfill$\square$\vskip6pt}}

\newcommand{\Poi}[2]{\{#1,#2\}}

\newcommand{\ncamp}[1]{\|#1\|^\oplus}
\newcommand{\norm}[1]{\left\|#1\right\|}

\newtheorem{theorem}{Theorem}[section]
\newtheorem{lemma}{Lemma}[section]
\newtheorem{corollary}{Corollary}[section]
\newtheorem{proposition}{Proposition}[section]
\newtheorem{definition}{Definition}[section]
\newtheorem{remark}{Remark}[section]

\newenvironment{proof}[0]{\noindent{\bf proof:\enspace}}{\hfill$\square$\par\medskip}

\title{An extensive adiabatic invariant\\ for the Klein-Gordon
  model\\ in the thermodynamic limit}

\author{Antonio Giorgilli, Simone Paleari and Tiziano Penati}

\begin{document}
\allowdisplaybreaks

\maketitle

\begin{abstract}
We construct an extensive adiabatic invariant for a Klein-Gordon chain
in the thermodynamic limit.  In particular, given a fixed and
sufficiently small value of the coupling constant $a$, the evolution
of the adiabatic invariant is controlled up to times scaling as
$\beta^{1/\sqrt{a}}$ for any large enough value of the inverse
temperature $\beta$. The time scale becomes a stretched exponential if
the coupling constant is allowed to vanish jointly with the specific
energy.  The adiabatic invariance is exhibited by showing that the
variance along the dynamics, i.e. calculated with respect to time
averages, is much smaller than the corresponding variance over the
whole phase space, i.e. calculated with the Gibbs measure, for a set
of initial data of large measure.  All the perturbative constructions
and the subsequent estimates are consistent with the extensive nature
of the system.
\end{abstract}


\section{Introduction and statement of results}
\label{s:1}

In the quest for a mathematically rigorous foundation of Statistical
Physics in general, and Statistical Mechanics in particular, despite
many efforts and recent successes, a lot of work is still to be done.
More specifically, if one considers an Hamiltonian system, instead of
some ad hoc model, for the microscopic description of large systems, the
behaviour over different long time scales is often still a challenge.
One of the possible, and natural strategies, is to apply the
techniques and results of Hamiltonian perturbation theory to large
systems, with particular attention to the thermodynamic limit, i.e. when
the number of degrees of freedom grows very large, at fixed, non
vanishing, specific energy.  The present paper is concerned with the
existence of an adiabatic invariant for an arbitrarily large one
dimensional Klein-Gordon chain, with estimates uniform in the size of
the system.

It is well known that results like the KAM and the Nekhoroshev
theorems stated for finite dimensional systems (see
e.g.~\cite{Kol54i,Mos61,Mos62,Arn63,Nek77,Nek79}) appear to be
somewhat useless as the number $N$ of degrees of freedom of the system
system grows, for the estimated dependence on $N$ of the constants
involved is usually very bad, and in particular the (specific) energy
thresholds do vanish in the limit $N\to\infty$. It is worth to mention
however that a first theoretical result at finite specific energy,
hence with estimates uniform in $N$, and on an average time scale can
be found in \cite{Car07}.  Extensions in the infinite dimensional case
have been made (see,
e.g.,~\cite{Bou97b,Bou98,Cra96,Kuk92,Kuk00,BN,PBC,Poe99b} for the case
of partial differential equations, or~\cite{BamG93,BFG,FSW} for the
case of lattices), but always for finite energy, i.e., for zero
specific energy.  Our aim is precisely to remove such a drawback,
producing a long time estimate for finite {\it specific} energy.

We consider the Hamiltonian
\begin{equation}
\label{e.H}
H(x,y) = \frac12\sum_{j=1}^N \quadr{ \tond{y^2_j + x^2_j} +
a(x_j-x_{j-1})^2 + \frac12x_j^4}  \ ,
\end{equation}
of a Klein-Gordon chain with $N$ degrees of freedom, periodic boundary
conditions $x_0=x_N$, and coupling constant $a$.

A previous investigation of a similar model has been made
in~\cite{GioPP12}.  In that paper a first order (in the sense of
perturbation theory) adiabatic invariant has been analytically
constructed.  Moreover, by numerical investigation it has been shown
that the adiabatic invariance persists for times much longer than
those predicted by the first order theoretical analysis.  Thus, the
model appeared to be worth of further theoretical investigation.  A
very recent breakthrough in this direction is represented by the
paper~\cite{CarM12}, which exploits the idea of complementing the
perturbation estimates with probabilistic techniques, thus producing a
control of the long time evolution in the thermodynamic limit.  We
will come back later to the relation between that paper and the
present work.

Let us give a brief sketch of our procedure.  The basic idea of both the
quoted works is to avoid the usual procedure of introducing normal modes
for the quadratic part of the Hamiltonian~(\ref{e.H}), thus considering
the model as a set of identical harmonic oscillators with a coupling
which includes a small quadratic term describing a nearest neighbours
interaction controlled by the small parameter $a$.  We construct an
extensive adiabatic invariant as follows.  First, as in~\cite{GioPP12},
we exploit a transformation of the quadratic part of the Hamiltonian
into the sum of two terms in involution, one of them including all
resonant coupling terms.  The relevant fact is that the transformation
preserves the extensive nature of the system and produces new
coordinates which are each exponentially localized around the
corresponding original ones.  As a subsequent step, the perturbation
process is performed here at higher order.  Thus we produce an adiabatic
invariant which still preserves both the extensive nature of the system
and the exponential decay of the interaction with the distance.
Furthermore we produce estimates which are uniform in the number $N$ of
degrees of freedom.

We stress that our model contains two independent perturbation
parameters, namely: (i)~the coupling parameter $a$, and (ii)~the
specific energy $\epsilon$.  This is a point that deserves particular
consideration.  We pay special attention in keeping these two
parameters separated, so that we can deal with the physically sound
hypothesis that the coupling parameter $a$ is fixed, and the inverse
temperature $\beta$ grows arbitrarily large.  Actually, the main
theorem~\ref{t.Adiab0} is formulated so that one is allowed to play
independently with both parameters in suitable ranges.

A second relevant point is concerned with the question how to assess the
adiabatic invariance of our quantity. The delicate point is again
related to the thermodynamic limit, which was indeed a major obstacle in
tackling the problem with perturbation methods, but can be dealt with
using a statistical approach. In a simplified description (see,
e.g.,~\cite{Lan73}) we can say that, as the number of degrees of freedom
grows, all the extensive functions appear as essentially constant over
the energy surface, in the sense that for increasing $N$ their densities
approach a delta function centered around their average value. Clearly
an almost constant function is also approximately constant along an
orbit, which seems not to give a meaningful information.  The idea is
thus to compare the dynamical fluctuation with the statistical deviation
of the function over the phase space, using the Gibbs measure. A
function defined on the phase space will be considered reasonably
conserved if its fluctuation along the orbit is significantly smaller
than its Gibbs variance, for a large set of initial data.

In the present paper we are able to show that, in the physically sound
assumption of fixed coupling constant $a$, as the specific energy
$\epsilon$ goes to zero, for a large (asymptotically full) Gibbs measure
of initial data, and for time scaling as inverse powers of $\epsilon$,
the time variance of our quantity is smaller than the corresponding
Gibbs variance, their ratio vanishing as a power of $\epsilon$.  The
estimates are uniform in the number $N$ of degrees of freedom.

We come now to a formal presentation of the results in a somehow
simplified form.  A general formulation is given in the main
theorem~\ref{t.Adiab0}, where some parameters appear that may be
subjected to a fine tuning ($a$ and $\beta$ among them).  In the
statement below we reduce the number of free parameters by making
appropriate choices, so as to give more readable, but still physically
meaningful results.  Whenever it will be useful we shall denote by $z$
all the coordinates and momenta $(x,y)$, and by $H(z,a)$ the Hamiltonian
so as to bring into evidence the dependence on the coupling constant
$a$.

We denote here by $dz$ the $2N$ dimensional Lebesgue measure in the
phase space ${\mathcal M}:=\RR^{2N}$, by $dm$ the Gibbs measure and by $Z$
the corresponding partition function, namely
\begin{equation*}
 dm(\beta,a) := \frac{e^{-\beta H(z,a)}}{Z(\beta,a)}dz \ ,
\qquad
 Z(\beta,a) := \int_{\mathcal M} e^{-\beta H(z,a)}dz \ ;
\end{equation*}
for every function $X:{\mathcal M}\to\RR$ we denote its phase average and
its variance\footnote{We do not use the standard notation $\sigma^2_X$
because we reserve the subscript, in particular with a $t$, for the
variance along the dynamics.} respectively by
\begin{equation*}
 \inter{X} := \int_{\mathcal M} X dm(\beta,a) \ ,
\qquad
 \sigma^2\quadr{X} := \inter{X^2} - \inter{X}^2 \ . 
\end{equation*}
For every measurable set $A\in{\mathcal M}$, we will denote $m(A):=\int_A
dm(\beta,a)$.

We recall that, for $\beta$ large and $a$ small, $\beta$ is roughly the
inverse of the average specific energy
\begin{equation*}
\frac1\beta \sim \frac{\inter{H}}{N} \ .
\end{equation*}

We also need to define the time average and the time variance, evaluated
along the time evolution.  Denoting by $\phi^t$ the Hamiltonian flow,
these quantities are naturally defined as
\begin{equation*}
 \overline{X}(z,t) := \frac1{t} \int_0^t \tond{X\circ\phi^s}(z) ds \ ,
\qquad
 \sigma^2_t\quadr{X} := \overline{X^2} - \overline{X}^2 \ . 
\end{equation*}
We remark that $\sigma^2_t\quadr{X}$ is a function of $(z,t)$, and that
all the previously defined averages and variances are clearly functions
of $\beta$ and $a$, even though we do not write these dependences
explicitly.

We state here a particular version of the main result of the paper
giving, for fixed coupling constant, and small specific energies, a
control for time scales growing as a power of $\beta$.  In the
statements of the present Section the symbols $C_1,\, C_2,\ldots$
denote constants that may have different values in different contexts.

\begin{theorem}
\label{t.main.power}
There exist positive constants $a^*$, $\beta_0$, $\beta_1$, $C_1$ and
$C_2$ such that, for all $0<a<a^*$, given the integer
$r:=\left\lfloor{C_1\sqrt\frac{1+2a}{a}} \right\rfloor$, there exists
an extensive polynomial $\Phi:{\mathcal M}\to\RR$ of degree $2r+2$,
such that, for all $\beta > \max\{\beta_0,\beta_1 r^6\}$ one has
\begin{equation*}
 m\left( z\in\RR^{2N}\ \colon\  \sigma^2_t[\Phi] \geq
  \frac{\sigma^2[\Phi]}{\sqrt\beta} \right)
\leq  \frac{C_2}{\beta}\tond{\frac{t}{\overline t}}^{2} \ ,
 \qquad\qquad  \overline t={\beta^{r/2}}
\end{equation*}
\end{theorem}

\begin{remark}
According to the result stated above, given a system with
Hamiltonian~\eqref{e.H} with a sufficiently small, and fixed, coupling
constant $a$, there exists a quantity whose time variance is smaller
than its phase variance for a set of initial data of large Gibbs
measure; this holds over long times scaling with
$\epsilon^{-C/\sqrt{a}}$, for small enough average specific energy
$\epsilon$. Actually, given the relation between $\beta$, $r$ and $a$,
the minimal time scale (corresponding to the maximal specific energy
allowed) is of order $r^r$, i.e. $(1/\sqrt{a})^{1/\sqrt{a}}$.
\end{remark}

We may state another result, where the time scale is a stretched
exponential in $\beta$.  The price to be paid, again in the hypothesis
of a fixed coupling constant, is that the specific energy must be
bounded both from below and from above; otherwise it is necessary to let
the coupling vanish as the specific energy goes to zero.

\begin{theorem}
\label{t.main.exp}
There exist positive constants $a^*$, $\beta_*$, $C^*$, $C_1$, $C_2$
and $C_3$ such that, for all $\beta \geq \beta_*$ and $0<a<a^*$
satisfying
\begin{equation*}
 \sqrt{a}\sqrt[3]\beta \leq C^* \ ,
\end{equation*}
given the integer $r:=\left\lfloor {C_1\sqrt[3]\beta} \right\rfloor$,
there exist an extensive polynomial $\Phi:{\mathcal M}\to\RR$ of degree
$2r+2$, such that one has
\begin{equation*}
 m\left( z\in\RR^{2N}\ \colon\  \sigma^2_t[\Phi] \geq
  \frac{\sigma^2[\Phi]}{\sqrt\beta} \right)
\leq  \frac{C_2}{\beta}\tond{\frac{t}{\overline t}}^{2} \ ,
 \qquad\qquad  \overline t={e^{C_3\sqrt[3]\beta}}
\end{equation*}
\end{theorem}

\noindent
The proofs of the theorems stated here are given in Section~\ref{s:5} as
corollaries of the main Theorem~\ref{t.Adiab0}.

The paper actually consists of two separate parts, namely: (i)~the
construction of an approximate conserved quantity with perturbation
methods, and (ii)~the control of the dynamical fluctuation using
statistical tools.

The first part makes use of the formal perturbation expansion method
introduced in~\cite{GioG78} and used in subsequent works, but implements
a quantitative scheme of estimates that exploits the characteristics of
the present system, namely the complete resonance, the extensivity of
the model and the exponential decay of interactions with the distance,
in order to produce estimates uniform in $N$.  The role of complete
resonance in removing the critical dependencies on $N$ goes back
to~\cite{BFG,BGG89} and has been used later on,
e.g. in~\cite{BamG93,GioPP12}.  The extensivity property has been dealt
with in our previous paper~\cite{GioPP12} exploiting the \emph{cyclic
symmetry}; some results are restated here in a more terse way, using the
formalism of circulant matrices.  The method of control of the
exponential decay introduced here is new, up to our knowledge.  The
quantitative perturbation scheme developed in the present work
significantly improves the one in~\cite{CarM12}.  In this respect we put
emphasis on the different method of solving the homological equation.
The problem is to invert a linear operator that depends on the coupling
parameter $a$.  We are able to formulate a direct inversion lemma, thus
replacing the truncation method used in~\cite{CarM12} with a more
effective one; the price we pay is that, at variance with their
approach, we actually have to control small divisors, thus introducing
an upper bound on the number of perturbation steps allowed.  The crucial
positive outcome of our choice, in this technical point, is the
possibility of preserving the independence of the parameters $a$ and
$\epsilon$, while they are collectively controlled in~\cite{CarM12} as
$a+\epsilon$.

In the second part, the statistical control of the fluctuation is
reminiscent of the scheme used in~\cite{CarM12}.  However, we are able
to produce improved estimates of the adiabatic invariance, and to prove
almost everything independently and in a different way. First of all we
do not rely on probabilistic techniques (like marginal probabilities)
but we use a more direct approach. In particular we exploit a mechanism
of cancellations of unwanted interaction terms, which allows us to bring
into evidence the decay properties of spatial correlations.  A second
point concerns the fundamental use, besides the decay properties of
spatial correlations, of the short range interaction properties of the
system which are preserved by our pertubative construction, the latter
being another outcome of the improvements of the first part.

Let us add a comment on the possible extensions of the present work.
Most of the ideas and techniques used here are not restricted to the one
dimensional case. E.g., properties like complete resonance, extensivity
and exponential decay of interaction range may be handled essentially in
the same way even for a multimensional lattice, perhaps at the price of
more complicated estimates. A trickier formalization may be required
in the part concerning the statistical estimates, in particular for the
cancellations.  

A further comment is devoted to the fact that in both the present result
and in~\cite{CarM12} the coupling parameter must be small enough: thus
the applicability to the Fermi-Pasta-Ulam model still remains open.
This is particularly relevant since the question of the relaxation
properties of the latter model at the thermodynamic limit still remains
not completely understood (see, e.g.,~\cite{FPU50}), despite some recent
advancement in the investigation of the integrability origin of the long
time stability exhibited both with long wave initial data and with
generic initial data (see, e.g.,
\cite{BPon05,palo,BenLP09,BenP11,BenCP13,BeGP,CGGP,GenGPP12}).

We close this review of the literature with a very recent\footnote{we
became aware of it actually during submission} and interesting result:
paper~\cite{DeRH13}. Although they consider different models and deal
with a different question, i.e. the problem of heat conduction, it
appears as a relevant work since they are able to perform a normal form
at the thermodynamic limit.

The paper is organized as follows. In Section~\ref{s:extensivity} the
general setting is introduced, with the formalization of the extensivity
of the system, the interaction range and its relation with the
perturbation tools. In Section~\ref{s:quadratic} we recall the normal
form transformation of the quadratic part of the Hamiltonian, which
produces the zeroth order approximation of the adiabatic invariant; the
formal construction is then carried on at higher perturbative orders in
Section~\ref{s:construction}. The control on the time evolution of the
adiabatic invariant, and the estimates on the measure of the set of
initial data for which they hold, are given in Section~\ref{s:5}, where
we actually give the complete and detailed version of the main result of
paper. An Appendix with several technical lemmas closes the paper.


\section{General setting}
\label{s:extensivity}

One of the guiding ideas of this work is to exploit some general
characteristics of a many particles mechanical system:
\begin{enumerate}[label=(\roman{*}), ref=(\roman{*})]
\item particles interacting with a two-body potential;

\item the potential is invariant with respect to rotations and
  translations;

\item the potential is assumed to be a smooth function; actually we
  consider the stronger condition of being analytic in the
  coordinates.
\end{enumerate}

\noindent
With these conditions the Hamiltonian may be given the generic form
\begin{displaymath}
H(q,p) = \sum_j \frac12 p_j^2 + \frac12\sum_{i\not=j} V(q_i,q_j)
\end{displaymath}
where the potential $V$ possesses the symmetry and short range
properties above.

These properties are quite general ones.  E.g., besides the realm of
Statistical Mechanics, they also apply to the Solar System, and have
actually been used by Lagrange in his theory of secular motions.

Here we restrict our attention to a system of identical particles on a
$d$--dimensional lattice, with a short or even finite range interaction.
In this case one needs just to know the local interaction of a particle
with its neighbors, or with the whole chain, and the complete
Hamiltonian is the sum of the contribution of every particle to both the
kinetic and the potential energy.  This is usually expressed by saying
that the Hamiltonian is {\sl extensive}.  Functions possessing the same
extensivity property of the Hamiltonian are particularly relevant.

\subsection{Formalization}
\label{ss:form}

We restrict our attention to the simplified model of a finite one
dimensional lattice with periodic boundary conditions and short range
interactions. We denote by $x_j,\,y_j$ the position and the momentum
of a particle, with $x_{j+N} = x_j$ and $y_{j+N}=y_j$ for any $j$.

\paragraph{Cyclic symmetry. }
\label{p:cyclic}

We give a formal implementation of extensivity by introducing the
concept of {\emph cyclic symmetry}. The \emph{cyclic permutation}
operator $\tau$ is defined as
\begin{equation}
\label{e.perm}
\tau(x_1,\ldots,x_N) = (x_2,\ldots,x_N,x_1)\ ,\quad
\tau(y_1,\ldots,y_N) = (y_2,\ldots,y_N,y_1)\ .
\end{equation}
We shall denote $\bigl(\tau f\bigr)(x,y)=f(\tau x,\tau y)$.

\begin{definition}
We say that a function $F$ is \emph{cyclically symmetric} if
$\tau F = F$.
\end{definition}

Cyclically symmetric functions may be constructed as follows.  Let $f$
be given.  A new function $F= f^{\oplus}$ is constructed as
\begin{equation}
F(x,y)= f^{\oplus}(x,y) = \sum_{l=1}^{N} \tau^l f(x,y)\ .
\label{e.cycl-fun}
\end{equation}
The upper index $\oplus$ should be considered as an operator defining
the new function.  We shall say that $f^{\oplus}(x,y)$ is generated by
the \emph{seed} $f(x,y)$.  Generally speaking the decomposition of a
cyclically symmetric function in the form~\eqref{e.cycl-fun} need not
be unique. We shall often use the convention of denoting extensive
functions with capital letters and their seeds with the corresponding
lower case letter.

The following properties will be useful:

\begin{enumerate}[label=(\roman{*}), ref=(\roman{*})]
\item if $f=f'+f''$ is a seed of a function $F$ then $\tau^{s'} f'
  +\tau^{s''}f''$ is also a seed of the same function, for any
  integers $s',\,s''$;
\item  the Poisson bracket $h^{\oplus} =
\Poi{f^\oplus}{g^\oplus}$ between two cyclically symmetric functions
is also cyclically symmetric.  A seed is easily constructed as
$h = \Poi{f}{g^\oplus}$, but other choices are allowed using the
property~(i) above.
\end{enumerate}

\paragraph{Norm of an extensive function. }
\label{p:norm-ext}

Assume now that we are equipped with a norm for our functions
$\norm{\cdot}$, e.g. the supremum norm over a suitable domain. We
introduce a norm $\ncamp{\cdot}$ for an extensive function
$F=f^\oplus$ by defining
\begin{equation*}
\bigl\|F\bigr\|^{\oplus} = \|f\|\ ,
\end{equation*}
i.e. we actually measure the norm of the seed. An obvious remark is
that the norm so defined depends on the choice of the seed, but this
will be harmless for the following reason. All the perturbation
procedure and the quantitative estimates on the norm, in the rest of
the paper, are based on the fact that all algebraic operations, in
particular Poisson brackets, induce a natural choice of the seed for
the resulting function. Thus the relevant estimates will be made
directly on the seed so that the initial choice is propagated through
the whole procedure. For these reasons, all the quantitative estimates
in the rest of the paper, could be restated as: the function we are
considering possesses a seed whose norm satisfies the stated
inequality. We do not explicitly mention this fact in every statement.
Moreover, we also have the following relevant facts:

\begin{enumerate}[label=(\roman{*}), ref=(\roman{*})]
\item for any $s$ one has $\norm{\tau^s f} = \norm{f}$; 
\item the inequality
$\|F\| \le N \|f\|$
holds true for any choice of the seed. 
\end{enumerate}

\noindent
This is particularly useful if we are able to produce norms of the
seed which are independent of $N$, since this fully exploits the
property of the system of being extensive. This is what we plan to do,
indeed.

\paragraph{Polynomial norms.}
\label{p:polinorms}

Let $f(x,y)=\sum_{jk}f_{j,k} x^jy^k$ be a homogeneous polynomial of
degree $s$ in $x,\,y$.  We define its polynomial norm as
\begin{equation*}
\|f\| := \sum_{j,k} |f_{j,k}|\ .
\end{equation*}

\paragraph{Short range interaction.}
\label{p:shortrange}
The short range interaction is characterized by writing the seed $f$
of a function as a sum $f=\sum_{m}f^{(m)}$, where the decomposition
$f^{(m)}$ is explained in Section~\ref{ss.int.range},
formula~\eqref{dcdm.5}.  We consider in particular the case of
exponential decay of interactions using two positive parameters: we
say that a function $f$ expanded as above is of class
$\Dscr(C_f,\sigma)$ in case one has $\norm{f^{(m)}}\le C_f
e^{-m\sigma}$.  Such a characterization of function is particularly
useful in statistical calculation.  The known quantitative
perturbation schemes will be adapted in order to deal with these
classes of functions.

\paragraph{Circulant matrices. }
\label{p:circulant}

Let us restrict our attention to the harmonic approximation around a
stable equilibrium. The Hamiltonian is a quadratic form represented by
a matrix $A$
\begin{displaymath}
H_0(x,y) = \frac12 y\cdot y + \frac12 Ax\cdot x.
\end{displaymath}
If the Hamiltonian $H$ is extensive, then the same holds also
for its quadratic part $H_0=h_0^\oplus$.  This implies that $A$
commutes with the matrix $\tau$ representing the cyclic
permutation~\eqref{e.perm}
\begin{equation*}
\tau_{ij}=
\begin{cases}
1\quad {\rm if}\ i=j+1\>({\rm mod}\,N)\>,\\
0\quad \rm{otherwise}.
\end{cases}
\end{equation*}
We remark that the matrix $\tau$ is orthogonal and generates a cyclic
group of order $N$ with respect to the matrix product.

We recall the following
\begin{definition}
\label{d.circulant}
A matrix $A\in {\rm Mat}_{\RR}(N,N)$ is said to be \emph{circulant} if
\begin{equation*}
A_{j,k} = a_{(k-j)\> ({\rm mod}\,N)}\ .
\end{equation*}
\end{definition}
Actually, the set of circulant matrices is a subset of Toepliz
matrices, i.e those which are constant on each diagonal. For a
comprehensive treatment of circulant matrices, see, e.g.,
\cite{Dav79}.
We just recall some properties that will be useful later.

\begin{enumerate}
\item The set of $N\times N$ circulant matrices is a real vector space
  of dimension $N$, and a basis is given by the cyclic group generated
  by $\tau$ (see 3.1 of \cite{Dav79}).
\item The set of matrices which commute with $\tau$, i.e. those $A$ such
      that $A\tau=\tau A$, coincides with the set of circulant matrices
      (see 3.1 of \cite{Dav79}).
\item The set of eigenvalues of a circulant matrix is the Discrete
  Fourier Transform of the first row of the matrix and vice-versa. This
  allows to construct the circulant matrix from its spectrum.
\item Let $M^2=A$, where $A$ is circulant; then $M$ is circulant,
  too. Moreover, from the definition of $M:=\sqrt{A}$, it follows that
  if $A$ is symmetric, then $M$ is also symmetric.
\end{enumerate}

In our problem the cyclic symmetry of the Hamiltonian implies that the
matrix $A$ of the quadratic form is circulant. Obviously it is also
symmetric, so that the space of matrices of interest to us has
dimension~$\modf{\frac{N}2}+1$. Indeed, a circulant and
symmetric matrix is completely determined by $\modf{\frac{N}2}+1$
elements of its first line.

\def\supp{\mathop{\rm supp}}
\def\diam{\mathop{\rm diam}}
\def\corsivo#1{{\sl #1}}

\subsection{Interaction range}
\label{ss.int.range}

We give here a formal characterization of finite range interaction,
pointing out some properties that will be useful in the rest of the
paper.  We first consider the case of an infinite chain, which is
easier to deal with.  Then we shall point out the differences with the
periodic case. 

\paragraph{The infinite chain.}
We start with some definitions. Let us label the variables as
$x_l,y_l$ with $l\in\ZZ$. Let us consider a monomial $x^jy^k$ (in
multi-index notation). We define the \emph{support} $S(x^jy^k)$ of the
monomial and the \emph{interaction distance} $\ell(x^jy^k)$ 
as follows: considering the exponents $(j,k)$ we set
\begin{equation*}
S(x^jy^k) = \{l\>:\>j_l\neq0 {\rm\ or\ } k_l\neq0 \}\ ,\quad
\ell(x^jy^k) = \diam\bigl(S(x^jy^k)\bigr) \ .
\end{equation*}
We say that the monomial is \emph{left aligned} in case
$S(x^jy^k)\subset \{0,\ldots,\ell(x^jy^k)-1\}$.

The definitions above is extended to a homogeneous polynomial $f$ by
saying that $S(f)$ is the union of the supports of all the monomials
in $f$, and that $f$ is left aligned if all its monomials are left
aligned.  The relevant property is that if $\tilde f$ is a seed of a
cyclically symmetric function $F$, then there exists also a left
aligned seed $f$ of the same function $F$: just left align all the
monomials in $\tilde f$.

For the seed $f$ of a function (using $z$ to collectively denote the
$x$ and $y$ variables, and $k$ the corresponding mulit-index) consider
the decomposition
\begin{equation}
\label{e.decomp}
f(z) =  \sum_{m\ge 0} f^{(m)}(z)\ ,\quad
 f^{(m)}(z) = \sum_{\ell(k)\le m} f_k z^k\ ,
\end{equation}
assuming that every $f^{(m)}$ is left aligned.  It would be
interesting to replace the inequality $\ell(k)\le m$ with equality,
but this is not compatible with the fact that the Poisson bracket can
possibly reduce the interaction range.  However, for our purposes it
is enough to assure two properties, namely: (i)~in $f^{(m)}$ there are
no terms with interaction range longer than $m$ (upper bound);
(ii)~the size of $f^{(m)}$, estimated with a norm, is of order
$\mu^m$, with some positive $\mu$.  This is what we are going to do.

For the Poisson bracket between two cyclically symmetric functions we
have
\begin{align*}
 \{f^{\oplus},g^{\oplus}\} &= \sum_{s,s'}\{\tau^{s}f,\tau^{s'}g\}
 =  \sum_{m,m'} \sum_{s,s'}\{\tau^{s}f^{(m)},\tau^{s'}g^{(m')}\}
\cr
 &=  \sum_{s} \tau^{s} 
     \Bigl(\sum_{m,m'} \sum_{s'}\{f^{(m)},\tau^{s'}g^{(m')}\}\Bigr)
  \ .
\end{align*}
The last expression immediately suggests to construct a seed by just
removing the translation $\tau^s$ and the sum over $s$.  However, we
remark 
 that the obvious equality
$$
\{\tau^{s+j}f^{(m)},\tau^{s'+j}g^{(m')}\} = 
 \tau^j\{\tau^s f^{(m)},\tau^{s'} g^{(m')}\}
$$
holds true.  Thus, we may replace any term
$\{f^{(m)},\tau^{s'}g^{(m')}\}$ for $s\in\ZZ$ with a translated one.

Let us exploit these facts.  Given $s\in\ZZ$, we concentrate our
attention on the expression $\{f^{(m)},\tau^{s'}g^{(m')}\}$.  The
following properties hold true.

\begin{enumerate}
\item If $s' < -m'$ or $s' >  m$ then one has
$\Poi{f^{(m)}}{\tau^{s'}g^{(m')}}=0$, for the two functions depend on
independent sets of variables.
\item If $s' < 0$ we may replace the seed $\Poi{f^{(m)}}{\tau^{s'}g^{(m')}}$
with $\Poi{\tau^{-s'}f^{(m)}}{g^{(m')}}$.
\item A seed for $\{f^{\oplus},g^{\oplus}\}$ is given by the $m+m'+1$ expressions
\begin{equation}
\label{dcdm.1}
\vcenter{\openup1\jot\halign{
\hfil$\displaystyle{#}$
&\hfil$\displaystyle{#}\hfil$
&\hfil$\displaystyle{#}\hfil$
&$\displaystyle{#}$\hfil
\cr
\{f^{(m)},g^{(m')}\}\,, & \{\tau f^{(m)},g^{(m')}\}\,, 
 & \ldots\,, & \{\tau^{m'}f^{(m)},g^{(m')}\} 
\cr
& \{f^{(m)},\tau g^{(m')}\}\,, 
 & \ldots\,, & \{f^{(m)},\tau^{m}g^{(m')}\} 
\cr
}}
\end{equation}
letting $m,m'\geq 0$.
\item Between the expressions in \eqref{dcdm.1} there are
\begin{equation}
\label{e.nutable}
\vcenter{\openup1\jot\halign{
\hfil$\displaystyle{#}$\hfil\quad
&$\displaystyle {#}$\hfil
&\quad{#}\hfil
\cr
 |m-m'|+1 & {\rm  with\quad}  \ell(\cdot) \le \max(m,m')\ , &plus \cr
 2 \hfil  & {\rm  with\quad}  \ell(\cdot) \le \max(m,m')+1\ , &plus \cr
 2 \hfil  & {\rm  with\quad}  \ell(\cdot) \le \max(m,m')+2\ , &plus \cr
      &\vdots \cr
 2 \hfil  & {\rm  with\quad}  \ell(\cdot) \le m+m'\ .\cr
}}
\end{equation}
\end{enumerate}
The third property follows from the first two, which are obvious. The
seed \eqref{dcdm.1} so found is left aligned. It may contain
duplicated monomials in some expression, but this is harmless because
we are only interested in bounding the interaction range. The last
property is just matter of counting.

\begin{figure}
\begin{center}
\includegraphics*[width=.9 \textwidth]{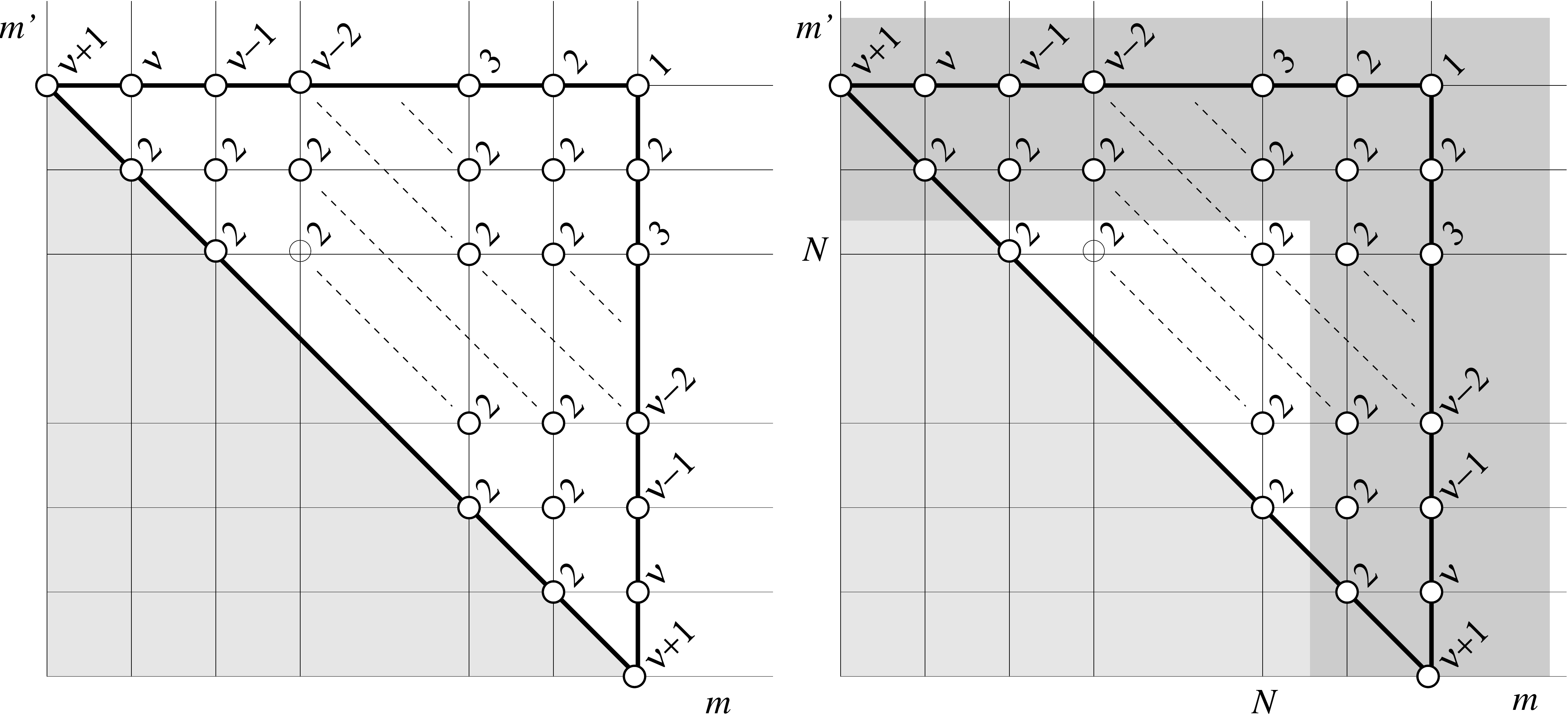}\\
\end{center}
\caption{ \label{fig:1}Graphical representation of the values of $m,\,
m'$ which concur in forming the seed $h^{(\nu)}$.  The white triangle
contains all nodes that must be selected according to
table~\ref{e.nutable}.  The left figure refers to the case of an
infinite chain.  The right figure shows which nodes are removed in the
case of a finite chain.  (See text for more details).}
\end{figure}

Denote now $h^{\oplus} =\{f^{\oplus},g^{\oplus}\}$. Letting $m$ and $m'$
to vary, we reorder the seed \eqref{dcdm.1} so that we can write
$h=\sum_{\nu\geq 0}h^{(\nu)}$.  To this end we collect together in
$h^{(\nu)}$ all expressions which according to~\eqref{e.nutable} have an
estimated upper bound of the interaction range equal to $\nu$.  This
assures on the one hand that the interaction range of $h^{(\nu)}$ does
not exceed $\nu$ and, on the other hand, that terms with interaction
range certainly less than $\nu$ are placed in some $h^{(\nu')}$ with
$\nu'<\nu$.  It is convenient to represent graphically the
table~\eqref{e.nutable} as a tridimensional diagram on $\NN^3$ by
putting $(m,m')$ on the horizontal plane and the admitted upper bounds
for $\ell(\cdot)$ on the vertical axis,
i.e. $\max(m,m'),\,\max(m,m')+1,\ldots,\, m+m'$.  To each non empty node
so identified we attach a weight given by the number of terms in the
first column of~\eqref{e.nutable}.  Then we make a section with the
horizontal plane of height $\nu$, thus obtaining the left diagram of
figure~\ref{fig:1} which represents schematically
all terms in~\eqref{e.nutable} that go into $h^{(\nu)}$.  The non empty
nodes on the selected plane satisfy $\max(m,m')\leq \nu\leq m+m'$,
namely they belong to the white triangle in the diagram.  The nodes of
the diagram together with their weight contain all the information we
need in order to estimate the norms.  The nodes inside the grey triangle
have $\ell(\cdot)$ certainly less than $\nu$, so we need not to include
them in $h^{(\nu)}$.  This rearrangement of seeds assures that the norm
of every term in $\eta^{(\nu)}$ has a factor $\mu^{\nu}$, as we shall
see later.

We conclude that 
\begin{equation}
\label{dcdm.4}
h^{(\nu)} = \sum_{m,m',s,s'}\{\tau^s f^{(m)},\tau^{s'}g^{(m')}\}
\end{equation}
the sum being extended to the nodes $m,m'$ in the diagram with the
translations $s,s'$ allowed for them according to the property 3.

\paragraph{The periodic chain.}
In view of the periodicity, the labels of the variables may be taken
to be $0,\ldots,N-1$, and the definitions of support, interaction
range and left alignment are easily adapted.  In particular, the
infinite sum on~\eqref{e.decomp} is truncated at $m=N$. Taking into
account the finite limits in the sums, we have
$$
h^{\oplus}=\{f^{\oplus},g^{\oplus}\} = 
  \sum_{s=0}^{N-1} \tau^{s} 
     \Bigl(\sum_{m,m'=0}^{N-1} \sum_{s'=0}^{N-1}\{f^{(m)},\tau^{s'}g^{(m')}\}\Bigr)
  \ .
$$
The seed's components $h^{(\nu)}$ are constructed in much the same way
with a minor change. Precisely in~\eqref{dcdm.1} we must
distinguish two different case. For $m+m' < N-1$ we get exactly the
same formula. For $m+m' \geq N-1$ we only have a subset of $N$
elements, namely
$$
\{f^{(m)},g^{(m')}\}\,,  \ldots\,,  \{f^{(m)},\tau^{N-1}g^{(m')}\} \ .
$$
This is represented in the right part of the diagram of
fig.~\ref{fig:1}, where the part to be omitted is covered in dark grey.

\paragraph{Exponential decay of interactions.}
We recall the definition given in Section~\ref{ss:form}.  The seed $f$
of a function $f$ is said to be of class $\Dscr(C_f,\sigma)$ in case
\begin{equation}
\label{dcdm.5}
\norm{f^{(m)}} \le C_f e^{-\sigma m}\ ,\quad C_f\gt 0\,,\> \sigma\gt 0\ ,
\end{equation}
where $f=\sum_{m}f^{(m)}$ is the expansion of $f$ in terms of
increasing interaction range, as in~(\ref{e.decomp}).

The following Lemma produces a general estimate of the Poisson bracket
specially adapted to the case of cyclically symmetric polynomials. It
is crucial for the control of the dependence on $N$ of the norms of
extensive functions generated by our perturbation scheme.

\begin{lemma}
\label{l.2.2}
Let $f(x,y)$ and $g(x,y)$ be homogeneous polynomials of degree $r$ and
$s$ respectively.  Then $\{f,g\}$ is a homogeneous polynomial of
degree $r+s-2$, and one has
\begin{displaymath}
\|\{f,g\}\| \le rs \|f\|\, \|g\|\ .
\end{displaymath}
Moreover, the seed $\{f,g^{\oplus}\}$ of $\{f^{\oplus},g^{\oplus}\}$
satisfies
\begin{equation}
\label{e.poisson.1}
\bigl\|\{f^{\oplus},g^{\oplus}\}\bigr\|^{\oplus} \le 
 rs \|f\|\, \|g\|.
\end{equation}
\end{lemma}

\begin{proof}
In order to prove the first inequality write the Poisson bracket as
$$
\{f,g\} = \sum_{j,k,j',k'} f_{j,k} g_{j',k'} 
 \sum_{l=1}^{n} \frac{j_l k'_l-j'_l k_l}{x_l y_l} 
 \; x^{j+j'}y^{k+k'}\ ,
$$
In view of the definition of the norm we may estimate
$$
\|\{f,g\}\| \le \sum_{j,k,j',k'} 
  |f_{j,k}|  \, |g_{j',k'}| 
 \sum_{l=1}^{n}(j_l k'_l+j'_l k_l)\ .
$$
Since  $j'_l\le s$ and
$k'_l\le s$ one has
$\sum_{l=1}^{n}(j_l k'_l+j'_l k_l) \le
 s\sum_{l=1}^{n}(j_l+k_l) = rs$, which 
readily gives the first inequality.  Coming
to~\eqref{e.poisson.1}, remark that  we may write
$$
f^{\oplus}(x,y) =
 \sum_{j,k}  f_{j,k} \sum_{m=1}^{N} (\tau^{m}x)^j (\tau^m y)^k\ ,
$$ 
meaning that all monomials $(\tau^{m}x)^j (\tau^m y)^k$ have the
 same coefficient.  Differentiating with respect to $x_l$ yields
$$
\frac{\partial f^{\oplus}}{\partial x_l} 
= \sum_{j,k} \sum_{m=1}^{N} 
 f_{j,k} \frac{(\tau^{-m} j)_l}{x_l}  (\tau^{m}x)^j (\tau^m y)^k\ .
$$
Using the cyclic decomposition of the Poisson bracket
$\{f^{\oplus},g^{\oplus}\} = \bigl(\{f,g^{\oplus}\}\bigr)^{\oplus}$
one gets
\begin{align*}
\{f,g^{\oplus}\} = \sum_{j,k,j',k'} f_{j,k} g_{j',k'}
 \sum_{l=1}^{N} \frac{j_l}{x_l y_l} \Bigg(
  &x^j y^k\sum_{m=1}^{N} (\tau^{-m} k')_l (\tau^m x)^{j'}(\tau^m  y)^{k'} +
\cr
 - &x^{j'} y^{k'}\sum_{m=1}^{N} (\tau^{-m} j')_l (\tau^m x)^{j'}(\tau^m y)^{k'}
  \Bigg)\ .
\end{align*}
The norm is thus estimated as
\begin{equation*}
\|\{f,g^{\oplus}\}\| \le \! \sum_{j,k,j',k'}
 |f_{j,k}| \; |g_{j',k'}|
  \sum_{l=1}^{N}\left( j_l \sum_{m=1}^{N} (\tau^{-m} k')_l +
   k_l \sum_{m=1}^{N} (\tau^{-m} j')_l\right)\ .
\end{equation*}
Remarking that
$\sum_{m=1}^{N} (\tau^{-m} j')_l = |j'|$  and
$\sum_{m=1}^{N} (\tau^{-m} k')_l = |k'|$, one has
$$
\sum_{l=1}^{N} ( j_l |k'| + k_l |j'|) \le
 (|j'|+|k'|) \sum_{l=1}^{N} (j_l + k_l) = 
  (|j'|+|k'|) (|j|+|k|)\ .
$$
In view of the definition of the norm one gets
$$
\|\{f,g^{\oplus}\}\| \le rs
 \|f\|\, \|g\| \ ,
$$
from which \eqref{e.poisson.1} follows.
\end{proof}

The next statements provide the basic estimates for controlling the
exponential decay in the framework of perturbation theory.

\begin{lemma}
\label{lem.poisson}
Let $F,\,G$ be cyclically symmetric homogeneous polynomials of degree
$r',r''$ respectively.  Let the seeds $f,\,g$ be of class
$\Dscr(C_f,\sigma')$ and $\Dscr(C_g,\sigma'')$, respectively, and let
$\sigma\lt\min(\sigma',\sigma'')$.  Then there exists $C_h\ge 0$ such
that the seed $h$ of $H=\Poi{F}{G}$ is of class $\Dscr(C_h,\sigma)$.
An explicit estimate is
\begin{equation*}
C_h = \frac{r' r'' C_f
  C_g}{(1-e^{-\max(\sigma',\sigma'')})(1-e^{-\max(\sigma',\sigma'')+\sigma})}\ .
\end{equation*}
\end{lemma}

\begin{proof}
According to~(\ref{dcdm.4}) the seed of $H$ may be written as
$$
h^{(\nu)}
 = \sum_{m',m''} \sum_{s'.s''} \Poi{\tau^{s'}f^{(m')}}{\tau^{s''}g^{(m'')}}\ ,
$$
where the sum must be extended to all nodes of the triangle of the
diagram~\ref{fig:1}.  In view of the general estimate of the Poisson
bracket in Lemma~\ref{l.2.2} we have
$$
\biggl\|{\sum_{s',s''} \Poi{\tau^{s'}f^{(m')}}{\tau^{s''}g^{(m'')}}}\biggr\|
 \le r' r'' C_f C_g e^{-m'\sigma'} e^{-m''\sigma''}\ ;
$$ This uses the cyclic symmetry and the fact that the sum over
 $s',s''$ is restricted to the values allowed by~(\ref{e.nutable}).
 Thus, for all nodes of the diagram we get a common factor $r' r'' C_f
 C_g$, and we must deal only with the exponentials.  Possibly
 exchanging the functions we may suppose that
 $\sigma'\gt\sigma''(>\sigma)$. {To get the estimate of $h^{(\nu)}$ we
   have to sum up all the couples $(m',m'')$ in the white triangle of
   FIG~\ref{fig:1}: we perform the summation by fixing the each
   diagonal segment $m'+m''=l$ and increasing
   $l=\nu,\ldots,2\nu$. Hence we can write
\begin{align*}
\sum_{l=\nu}^{2\nu}\sum_{m'+m''=l}e^{-m'\sigma'} e^{-m''\sigma''} &=
\sum_{l=\nu}^{2\nu}\sum_{m'=l-\nu}^\nu e^{-l\sigma''}
e^{-m'(\sigma'-\sigma'')} =\\ &= \sum_{l=0}^{\nu}\sum_{m'=l}^\nu
e^{-(l+\nu)\sigma''} e^{-m'(\sigma'-\sigma'')} = \\ &=
\sum_{l=0}^{\nu}\sum_{m=0}^{\nu-l} e^{-(l+\nu)\sigma''}
e^{-(m+l)(\sigma'-\sigma'')} = \\&= \sum_{l=0}^{\nu}\sum_{m=0}^{\nu-l}
e^{-\nu\sigma''}e^{-l\sigma'} e^{-m(\sigma'-\sigma'')}\ ,
\end{align*}
where we have first replaced $m''=l-m'$, then we have shifted back the
interval of the running index $l$ (thus exhibiting
$e^{-\nu\sigma''}$), and finally we have shifted back the interval of
the running index $m(=m')$.} Then we estimate
$$
\vcenter{\openup1\jot\halign{
\hfil$\displaystyle{#}$
&$\displaystyle{#}$\hfil
\cr
\sum_{l=0}^{\nu}\sum_{m=0}^{\nu-l}
e^{-\nu\sigma''}e^{-l\sigma'} e^{-m(\sigma'-\sigma'')}
&= e^{-\nu\sigma} e^{-\nu(\sigma''-\sigma)}\sum_{l=0}^{\nu} e^{-l\sigma'} 
                  \sum_{m=0}^{\nu-l} e^{-m (\sigma'-\sigma'')}\leq
\cr
&\le e^{-\nu\sigma} \sum_{l=0}^{\nu} e^{-l\sigma'} 
                   \sum_{m=0}^{\nu-l} e^{-m (\sigma'-\sigma)}<
\cr
&\lt \frac{e^{-\nu\sigma}}{(1-e^{-\sigma'})(1-e^{-(\sigma'-\sigma)})}\ .
\cr
}}
$$ 
The claim follows by replacing $\sigma'$ with
$\max(\sigma',\sigma'')$.  This completes the proof.
\end{proof}

\begin{corollary}
\label{cor.poisson}
If in lemma~\ref{lem.poisson} we have $\sigma'\neq\sigma''$ then we
may set $\sigma=\min(\sigma',\sigma'')$ and
\begin{equation*}
C_h = \frac{r' r'' C_f
  C_g}{(1-e^{-\max(\sigma',\sigma'')})(1-e^{-|\sigma'-\sigma''|})}\ .
\end{equation*}
\end{corollary}

\begin{proof}
Just set $\sigma=\sigma''$ and at the end replace $\sigma'-\sigma$
with $|\sigma'-\sigma''|$.
\end{proof}

{
\begin{corollary}
\label{cor.pois.Z0.corr}
If in lemma~\ref{lem.poisson} we have $\sigma'\gt\sigma''$  and
$f^{(0)}=0$, i.e., $f=\sum_{m\ge 1}f^{(m)}=O(e^{-\sigma'})$ then we
may set $\sigma=\sigma''$ and
\begin{equation*}
 C_h = \frac{2e^{-(\sigma'-\sigma'')}r' r'' C_f C_g}
 {(1-e^{-\sigma'})(1-e^{-(\sigma'-\sigma'')})} \ .
\end{equation*}
\end{corollary}
}
{
\begin{proof}
Set $\sigma=\sigma''$. Then hypothesis $f=\sum_{m'\geq 1}f^{(m')}$
implies that we must remove the element $(m',m'')=(0,\nu)$ from the
elements of the white triangle of FIG.~\ref{fig:1}: this element gives
a factor $e^{-\nu\sigma''}$. Hence the sum in Lemma \ref{lem.poisson}
becomes
\begin{align*}
\sum_{l=0}^{\nu} \sum_{m=0}^{\nu-l} e^{-(m+l)\sigma'}
e^{-(\nu-m)\sigma''} - 1 &= e^{-\sigma''\nu}\quadr{\sum_{l=0}^{\nu}
  e^{-l\sigma'} \sum_{m=0}^{\nu-l} e^{-m (\sigma'-\sigma'')}-1}<\\ 
&< e^{-\sigma''\nu}
\quadr{\frac{e^{-\sigma'}+e^{-(\sigma'-\sigma'')}}{(1-e^{-\sigma'})(1-e^{-(\sigma'-\sigma'')})}}<\\ 
&< e^{-\sigma''\nu} \frac{2e^{-(\sigma'-\sigma'')}}
    {(1-e^{-\sigma'})(1-e^{-(\sigma'-\sigma'')})} \ ,
\end{align*}
which readily gives the claim.
\end{proof}
}


\section{Normal form for the Quadratic Hamiltonian}
\label{s:quadratic}

Let us rewrite the Hamiltonian~\eqref{e.H} as a sum of its quadratic
and quartic parts $H=H_0+H_1$, where
\begin{equation}
\label{e.H.dec}
 H_0(x,y) := \frac12\sum_{j=1}^N \quadr{y^2_j + x^2_j +
   a(x_j-x_{j-1})^2} \ , \qquad H_1(x,y) := \frac14\sum_{j=1}^N x_j^4
 \ .
\end{equation}

The aim of this Section is to give the quadratic part a resonant
normal form so that it turns out to be written as $H_0 = H_{\Omega} +
Z_0$ with $Z_0$ an extensive function exhibiting an exponential decay
of the interaction among sites with their distance and
$\Poi{H_0}{Z_0}=0$.  That is, $Z_0$ is a first integral for $H_0$.
This result has been already stated in~\cite{GioPP12}; we give here a
different proof.  We will then apply the transformation also to the
quartic part of our Hamiltonian, showing that it still has an
exponential decay of the interactions.

\subsection{The normalizing transformation}
\label{ss:quadratic}
We introduce the positive parameters $\omega(a)>1$ and ${\mu}(a)<1/2$
\begin{equation*}
 \omega^2(a) := 1+2a\ ,\quad {\mu}:=\frac{a}{\omega^2}
\end{equation*}
ad rewrite the quadratic part of our Hamiltonian as
\begin{equation*}
H_0(x,y) = \frac12 y\cdot y +\frac12 x\cdot Ax,
\end{equation*}
where (recalling $\tau$ as the permutation matrix
generating~\eqref{e.perm})
\begin{equation}
\label{e.def-A}
A= \omega^2 \left(
\begin{matrix}
1    & -{\mu} & 0    & \ldots & 0    & -{\mu} \\
-{\mu} & 1    & -{\mu} & \ldots & 0    & 0    \\
0    & -{\mu} & 1    & \ddots & 0    & 0    \\
\vdots & \vdots & \ddots & \ddots & \ddots & \vdots\\ 
0    & 0    & 0    & \ddots & 1    & -{\mu} \\
-{\mu} & 0    & 0    & \ldots & -{\mu} & 1    \\
\end{matrix}
\right)
 = \omega^2\quadr{\Id - {\mu}(\tau + \tau^{\top})} \ ,
\end{equation}
which is clearly circulant and symmetric, and gives a finite range
interaction.  The latter form is particularly useful because it
exhibits the perturbation parameter ${\mu}$ that will be assumed to be
small. This particular form allows us to look at our model as a system
of identical harmonic oscillators with a small linear coupling. The
resulting complete resonance is one of the keys of our result.
Introduce the constant $\Omega$ as the average of the square roots of
the eigenvalues of $A$.

\begin{proposition}
\label{p.1}
For ${\mu}<1/2$ there exists a canonical linear transformation
which gives the Hamiltonian $H_0$ the particular resonant normal form
\begin{equation}
\label{e.dec.H0}
 H_0 = H_\Omega + Z_0 \ ,
\qquad
 \Poi{H_\Omega}{Z_0}=0
\end{equation}
with $H_\Omega$ and $Z_0$ cyclically symmetric with seeds
\begin{equation*}
 h_\Omega = \frac\Omega2(q_1^2+p_1^2) \ ,
\end{equation*}
\begin{equation}
\label{e.Z0}
\begin{aligned}
 \zeta_0 = \frac12 \!\sum_{j=1}^{\left\lfloor\frac{N}2\right\rfloor}
            &b_{j}(q_0q_j \!+\! p_0p_j \!+\! q_0q_{N-j} \!+\! p_0p_{N-j+1}) 
    +\delta b_{\frac{N}2+1}(q_0q_{\frac{N}2+1} \!+\! p_0p_{\frac{N}2+1})
\cr
 &|b_j({\mu})| = \mathcal{O}((2{\mu})^j) \ ,
 \qquad\qquad\qquad\qquad\ \ \ \delta = \begin{cases}
	   0 &N{\rm\ odd}
	   \cr
	   1 &N{\rm\ even}
	  \end{cases}
\end{aligned}
\end{equation}
The linear transformation is given by
\begin{equation}
\label{e.lin.trsf}
q=A^{1/4} x\ ,\qquad\qquad p=A^{-1/4}y\ ,
\end{equation}
where the circulant and symmetric matrix $A^{1/4}$ satisfies
\begin{equation}
\label{e.row1.A14}
 \bigl(A^{1/4}\bigr)_{1,j}= c_j({\mu}) (2{\mu})^{j-1} \ , \qquad 1\leq
 j\leq\left\lfloor{\frac{N}2}\right\rfloor+1 \ , \qquad |c_j({\mu})|\leq
 2\sqrt\omega \ .
\end{equation}
$H_1$ remains an extensive and cyclically symmetric function once
composed with the transformation~\eqref{e.lin.trsf}.
\end{proposition}

We remark that all the perturbative construction is performed after
the linear transformation~\eqref{e.lin.trsf}, but all the estimates
with the Gibbs measure of Section~\ref{s:5} are made in the original
variables.  We thus need some further properties of the transformation
itself, which are given in the following two results. Recalling that,
according to the notations of Section~\ref{ss.int.range}, we label the
coordinates with indices $0,\ldots,N-1$, {and introducing the decay
  rate $\sigma_0$
\begin{equation}
\label{e.mu.def}
\sigma_0:=-\ln(2\mu)\qquad\Rightarrow\qquad 2\mu=e^{-\sigma_0}\ ,
\end{equation}}
we have
{
\begin{proposition}
\label{p.chi0}
The linear canonical transformation \eqref{e.lin.trsf} is the flow at
time $t=1$ of the cyclically symmetric quadratic form $\Chi_0$
\begin{equation*}
 \Chi_0(x,y) := x\cdot B y \ ,
\qquad
 B := \frac14 \ln\tond{A}\ ,
\end{equation*}
where $B$ is a symmetric and circulant matrix characterized by
\begin{equation*}
 B_{1,j} = c_j({\mu}) (2{\mu})^{j} \ , \qquad
 |c_j({\mu})|\leq \frac12 C_0(a):=
 \frac14\Big|\ln\tond{\frac{\omega^2}{1-2{\mu}}}\Big| \ ,
\end{equation*}
for $1\leq j\leq\left\lfloor{\frac{N}2}\right\rfloor+1$; the seed of
$\Chi_0$ satisfies $\chi_0 \in \Dscr\bigl(C_0(a),\sigma_0\bigr)$ and
reads
\begin{equation}
\label{e.chi0}
 \chi_0 = \sum_{j=1}^{\left\lfloor\frac{N}2\right\rfloor}
          B_{1,j} (x_0y_j+y_0x_j) + \delta B_{1,N/2+1}x_0y_{N/2+1}
\qquad
 \delta = \begin{cases}
	   0 &N{\rm\ odd}
	   \cr
	   1 &N{\rm\ even}
	  \end{cases}
\end{equation}
\end{proposition}
\begin{lemma}
\label{l.inv.g}
Let $\rho^\oplus$ an homogeneous polynomial of degree $2r+2$ in
$\Dscr\bigl(C_\rho,\sigma_*\bigr)$ and assume
$\chi_0\in\Dscr\bigl(C_0(a),\sigma'\bigr)$ with $\sigma_*<\sigma'\leq
\sigma_0$, then
\begin{displaymath}
T_{\Chi_0}\rho \in \Dscr\bigl(e^{(r+1)\tilde
  C}C_\rho,\sigma_*\bigr)\ ,\qquad\qquad \tilde C\leq \frac{2
  C_0(a)}{(1-e^{-\sigma'})(1-e^{-(\sigma'-\sigma_*)})}\ .
\end{displaymath}
\end{lemma}
\bigskip
A fundamental point is represented by the decay properties of the
seeds of $Z_0$ and $H_1$.
\begin{lemma}
\label{l.exp.Z0}
The seeds of the functions $Z_0$ and $H_1$ satisfy
\begin{equation*}
\begin{aligned}
 \zeta_0&\in\Dscr\bigl(C_0(a),\sigma_0\bigr) \ , \qquad\quad C_0(a) =
 \mathcal{O}(1) \cr h_1&\in\Dscr\bigl(C_1(a),\sigma_1\bigr) \ ,
 \qquad\quad C_1(a)= \mathcal{O}(1)
\end{aligned}
{\rm for\ } a\to 0 \ , \qquad \sigma_1 := \frac12\sigma_0\ .
\end{equation*}
\end{lemma}
\begin{remark}
The seed $h_1$ cannot preserve the same exponential decay rate of the
linear transofrmation (see the corresponding proof in the Appendix);
however it is possible to show that
$h_1\in\Dscr\bigl(C_1(a),\sigma_1\bigr)$ for any
$\sigma_1<\sigma_0$. We make the choice $\sigma_1=\sigma_0/2$ in order
to explicitely relate $\sigma_1$ to the small natural parameter $a$ of
the model, since it will be useful in the estimates of the main
Theorems of the paper.
\end{remark}
}

The proofs of all the statements of this Section are deferred to
Appendix~\ref{app:quad}.


\section{Construction of an extensive first integral}
\label{s:construction}

We construct a formal first integral for the Hamiltonian \eqref{e.H} using the
Lie transform algorithm in the form introduced in \cite{GioG78}.  We include a
brief description, referring to the quoted paper for proofs.

Given a generating sequence $\{\Chi_s\}_{s\geq 1}$, we define the linear
operator $T_\Chi$ as
\begin{equation}
\label{e.TelChi}
T_\Chi=\sum_{s\geq 0}E_s,\qquad E_0=\Id,\qquad E_s= \sum_{j=1}^s
\lie{\Chi_j}E_{s-j}\ ,
\end{equation}
where $\lie{\Chi_j}\cdot=\Poi{\Chi_j}{\cdot}$ is the Lie derivative with
respect to the flow generated by $\Chi_j$.  The operator $T_\Chi$ turns out to
be invertible and to possess the relevant properties
\begin{equation}
\label{e.prop.Tchi}
T_\Chi(f\cdot g) = (T_\Chi f)\cdot (T_\Chi g)\ ,\qquad
T_\Chi\Poi{f}{g} = \Poi{T_\Chi f}{T_\Chi g}\ .
\end{equation}
Let now $Z$ satisfy the equation $T_\Chi Z=H$, and let $\Phi_0$ commute with
$Z$, i.e. $\Poi{\Phi_0}{Z}=0$.  Then in view of the second of
\eqref{e.prop.Tchi}, one immediately gets $\Poi{T_\Chi{\Phi_0}}{H}=0$, i.e.,
$\Phi= T_\Chi \Phi_0$ is a first integral for the Hamiltonian $H$.

The operator $T_\Chi$ is defined here at a formal level, and it is know
that using it in normal form theory usually produces non convergent
expansions. However we may well use it in formal sense, as explained by
Poincar\'e (Ch. VIII in~\cite{Poi93}). What we actually do is truncate
the all expansions at a given order, so that all the equalities above
are true up to terms of order larger than $r$ (i.e. of degree larger
that $2r+2$ in our polynomial expansions). E.g., the sentence above
``$\Phi= T_\Chi \Phi_0$ is a first integral for $H$'' should be
interpreted as ``having determined $\Chi_1,\dots,\Chi_r$, then truncate
$\Phi^{(r)}=\Phi_0+\Phi_1+\cdots+\Phi_r$, so that we have
$\{\Phi^{(r)},H\} = \mathcal{O} (r+1)$''. The statement of
Proposition~\ref{prop.gen} below must be interpreted in this sense.

In this Section we prove the following

\begin{proposition}
\label{prop.gen}
Consider the Hamiltonian $H=h^{\oplus}_{\Omega}+\zeta^{\oplus}_0 +
h^{\oplus}_1$ with seeds $h_{\Omega}=\frac{\Omega}{2}(x_0^2+y_0^2)$,
the quadratic term $\zeta_0$ of class $\Dscr(C_0,\sigma_0)$ with
$\zeta_0^{(0)}=0$, and the quartic term $h_1$ of class
$\Dscr(C_1,\sigma_1\,)$, {with $\sigma_0>\sigma_1>\ln(4)$}. Pick a
positive {$\sigma_*<\sigma_1$}. Then there exist positive $\gamma$,
$\mu_{*}$ and $C_*$
\begin{equation}
\label{const.prop.gen}
{\begin{aligned} \mu_* &=
    \frac{\Omega(1-e^{-\sigma_0})(1-e^{-(\sigma_0-\sigma_*)})}{8C_0e^{\sigma_1}}\ ,
\cr
 \gamma &= 2\Omega\Bigl(1-\frac{r\mu}{\mu_*}\Bigr) \ ,
\cr
 C_* &= \frac{C_1}{\gamma(1-e^{-\sigma_0})(1-e^{-(\sigma_0-\sigma_*)})} \ .
\end{aligned}}
\end{equation}
such that for any positive integer $r$ satisfying
\begin{equation}
\label{e.muperr}
2r\mu\lt\mu_*\ ,
\end{equation}
there exists a finite generating sequence $\Chi= \{\chi^{\oplus}_1,
\ldots, \chi^{\oplus}_r\}$ of a Lie transform such that $T_{\Chi}Z - H
= \mathcal{O} (r+1)$, i.e. the remainder $\mathcal{O} (r+1)$ starts
with terms of degree bigger than $2r+4$ and $Z$ is an extensive
function of the form
\begin{equation*}
 Z = h^{\oplus}_{\Omega}+\zeta^{\oplus}_0 
     +\ldots+ \zeta^{\oplus}_r
\end{equation*}
with $\lie{\Omega}Z_s=0$ for $s=0,\ldots,r$, $Z_s=\zeta_s^{\oplus}$ of
degree $2s+2$.

Moreover, defining
\begin{align}
\label{e.Cr}
C_r:&=64r^2C_* \ ;
\\
\label{succ.sigma}
 \sigma_s :&= \frac{s\sigma_* + (r-s)\sigma_1}{r}
 \quad{\rm for}\ s=2,\ldots,r\ ,
\end{align}
the following statements hold true:
\begin{enumerate}[label=(\roman{*}), ref=(\roman{*})]
\item the seed $\chiph_s$ of $\Chi_s$ is of class $\Dscr(C_r^{s-1}
      \frac{C_1}{\gamma s}, \sigma_s)$;
\item the seed $\zeta_s$ of $Z_s$ is of class $\Dscr (C_r^{s-1}C_1, \sigma_s)$;
\item if $\Phi=\ph^{\oplus}$ is a homogeneous polynomial of degree $2m$ and of
      class $\Dscr(C_{\ph},\sigma_0)$ then for $s=0,\ldots,r$ one has that
      $E_s\Phi$ is of class $\Dscr(F_r^{s} C_{\ph},\sigma_s)$ with $F_r =
      16(m+2)r^2C_*$;
\item setting $\Phi=H_{\Omega}$ in the previous point we have that $E_s
      H_{\Omega}$ is of class $\Dscr(F_r^{s-1}C_1,\sigma_s)$;
\item setting $\Phi_0=H_{\Omega}$ and considering the first $r+1$ terms in the
      expansion of $T_{\Chi}\Phi_0$, namely $\Phi^{(r)}=\Phi_0+\ldots+\Phi_r$
      with $\Phi_s=E_s\Phi_0$, we have
\begin{equation*}
 \dot\Phi^{(r)} = \Poi{H_1}{\Phi_r}
\end{equation*}
      which is a cyclically symmetric homogeneous polynomial of degree
      $2r+4$ and of class $\Dscr(C_{\rho},\sigma_*)$ with 
\begin{equation*}
 C_{\rho}= \frac{8(r+2)(16r^2C_*)^{r-1}C_1^2}
            {(1-e^{-\sigma_0})(1-e^{-(\sigma_0-\sigma_*)})} \ .
\end{equation*}
\end{enumerate}
\end{proposition}

\noindent
The rest of this Section is devoted to the proof of the proposition.
We first include a formal part, where we illustrate in detail the
process of construction of the normal form and introduce an
appropriate framework which allows us to control how the interaction
range propagates.  Then we give quantitative estimates paying
particular attention to the exponential decay of interactions with the
distance.

\subsection{Formal algorithm and solution of the homological equation}
\label{ss.formal.tchi}
We now translate the equation $T_\Chi Z=H$ into a formal recursive
algorithm that allows us to construct both $Z$ and $\Chi$. We take
into account that our Hamiltonian has the particular form
$H=H_0+H_1\,$, where $H_1$ is a homogeneous polynomial of degree 4.

For  $s\geq 1$ the generating function $\Chi_{s}$ and the normalized
term $Z_s$ must satisfy the recursive set of {\sl homolo\-gical equations}
\begin{equation}
\label{e.om.2s}
 \lie{H_0}\Chi_{s} = Z_{s} + \Psi_{s} \ ;
\end{equation}
where 
\begin{equation}
\label{e.Psi.2s}
\begin{aligned}
 \Psi_{1} &= H_1,\\ \Psi_{s} &= \frac{s-1}{s}\lie{\Chi_{s-1}}H_1 +
 \sum_{j=1}^{s-1} \frac{j}{s} E_{s-j}Z_{j} \ ,\qquad s\geq2 \ .
\end{aligned}
\end{equation}
A justification of this algorithm is the following. Using the
definition~(\ref{e.TelChi}) of $T_{\chi}$ we expand the equation $T_\chi Z=H$
into the recursive set of equations
\begin{equation}
\label{formal.1}
\begin{aligned}
 &Z_0 = H_0 \ ,
\cr
 &Z_{1} + E_1 Z_0 =  H_1 \ ,
\cr
 &E_s Z_0 + \sum_{l=1}^{s} E_{s-l}Z_{l} + Z_{s} = 0
 \quad{\rm for}\ s\gt 1
\end{aligned}
\end{equation}
In view of $E_1Z_0 =\lie{\chiph_1} H_0$ the second equation is readily written
as $\lie{H_0}\chiph_1 = Z_1 - H_1$, which is the homological equation at order
$s=1$.  Then, using the definition of $E_s$, we replace $E_s Z_0 =
\sum_{l=1}^{s-1}\frac{l}{s}\lie{\chiph_l} E_{s-l}Z_0$ in the third
of~(\ref{formal.1}) and get the homological equation $\lie{H_0}\chiph_{s} =
Z_{s} + \Psi_{s}$, where
\begin{equation*}
\Psi_{s} = \sum_{l=1}^{s-1}\frac{l}{s}\lie{\chiph_{l}} E_{s-l}Z_0 +
          \sum_{l=1}^{s-1} E_{s-l}Z_{l}\ .
\end{equation*}
The expression for $\Psi_{s}$ may be simplified thanks to the equations of the
previous orders as follows.  Replacing $E_{s-l}Z_0$ as given
by~(\ref{formal.1}) in the first sum calculate
\begin{equation*}
\begin{aligned}
\sum_{l=1}^{s-1}\frac{l}{s}\lie{\chiph_{l}} E_{s-l}Z_0
 &=\frac{s-1}{s}\lie{\chiph_{s-1}} H_1
  -\sum_{l=1}^{s-1} \frac{l}{s} \lie{\chiph_{l}}
    \sum_{j=1}^{s-l} E_{s-l-j} Z_{j}
\cr
 &= \frac{s-1}{s}\lie{\chiph_{s-1}} H_1
   -\sum_{j=1}^{s-1} \frac{s-j}{s}
     \sum_{l=1}^{s-j} \frac{l}{s-j} \lie{\chiph_{l}} E_{s-j-l} Z_{j}
\cr
 &= \frac{s-1}{s}\lie{\chiph_{s-1}} H_1
   -\sum_{j=1}^{s-1} \frac{s-j}{s} E_{s-j} Z_{j}\ ,
\end{aligned}
\end{equation*}
where the definition of the operator $E_s$ has been used in the last equality.
Then replace the latter expression in the r.h.s.~of $\Psi_{s}$ above and get
the wanted expression~\eqref{e.Psi.2s}.

Our aim is to solve the homological equation \eqref{e.om.2s} with the
prescription that $\lie{\Omega}Z_s=0$ where $\lie{\Omega}\,\cdot\, :=
\{H_{\Omega},\cdot\}$ is the Lie derivative along the vector field
generated by $H_{\Omega}$ as defined in \eqref{e.dec.H0}.  Thus the next
step is to point out the properties of the operator $\lie{\Omega}$, and
discuss the solution of the homological equation.

\subsubsection{The linear operator $\lie{\Omega}$}
\label{ssect.2.1}

It is an easy matter to check that $\lie{\Omega}$ maps the space of
homogeneous polynomials into itself.  It is also well known that
$\lie{\Omega}$ may be diagonalized via the canonical transformation
\begin{equation}
 x_j = \frac{1}{\sqrt{2}}(\xi_j+i\eta_j)\ ,\quad
  y_j = \frac{i}{\sqrt{2}}(\xi_j-i\eta_j)\ ,\quad
  j=1,\ldots,N\ ,
\label{pertur.51}
\end{equation}
where $(\xi,\eta)\in\complessi^{2n}$ are complex variables.  A straightforward
calculation gives
\begin{equation*}
 \lie{\Omega}\xi^j\eta^k =
  i\Omega\,(|k|-|j|)\,\xi^j\eta^k\ ,
\end{equation*}
where $|j|=|j_1|+\ldots+|j_N|$ and similarly for $|k|\,$.

A relevant general property is that if $f(x,y)=\sum_{j,k} c_{j,k} x^jy^k$ (in
multi-index notation) is a real polynomial, then the
transformation~(\ref{pertur.51}) produces a polynomial $g(\xi,\eta)=\sum_{j,k}
b_{j,k}\xi^j\eta^k$ with complex coefficients $b_{j,k}$ satisfying
\begin{equation*}
 b_{j,k} = - b_{k,j}^*\ .
\end{equation*}
Conversely, this is the condition that the coefficients of $g(\xi,\eta)$ must
satisfy in order to assure that transforming it back to real variables $x,y$
we get a real polynomial.

Let us denote by $\Pscr^{(s)}$ the (finite) linear space of the
homogeneous polynomials of degree $s$ in the $2n$ canonical variables
$\xi_1,\ldots,\xi_n,\eta_1,\ldots,\eta_n\,$.  The kernel and the range
of $\lie{\Omega}$ are defined in the usual way, namely
\begin{equation*}
 \Nscr^{(s)} = \lie{\Omega}^{-1}(0)\ ,\quad
  \Rscr^{(s)} = \lie{\Omega}(\Pscr^{(s)})
\end{equation*}
The property of $\lie{\Omega}$ of being diagonal implies
\begin{equation*}
 \Nscr^{(s)}\cap\Rscr^{(s)} = \{0\}\ ,\quad
  \Nscr^{(s)}\oplus\Rscr^{(s)} = \Pscr^{(s)}\ .
\end{equation*}
Thus the inverse $\lie{\Omega}^{-1}:\range\rightarrow\range$ is uniquely
defined on the restriction $\range$ of $\pspazio$.  It will also be useful to
introduce the projectors on the range and on the kernel defined as
\begin{equation*}
 \Pi_{\range}=\lie{\Omega}^{-1}\lie{\Omega}\ ,\quad
  \Pi_{\nucleo}= \Id - \Pi_{\range},\qquad\qquad .
\end{equation*}
so that we have $\Pi_{\range}+\Pi_{\nucleo} =\Id$.

\begin{lemma}
\label{l.poitable}
Let $f\in \Pscr^{(s)}$ and $g\in \Pscr^{(r)}$.  Then the following composition
table applies:
\begin{equation}
\label{pertur.54a}
\vcenter{\tabskip=0pt
\def\tablerule{\noalign{\hrule}}
\halign{
 \hbox to 4 em{\hfil$\displaystyle{#}\hfil$}\vrule\vrule
&\hbox to 5 em{\hfil$\displaystyle{#}\hfil$}\bigg\vert
&\hbox to 5 em{\hfil$\displaystyle{#}\hfil$}\vrule\vrule
\cr
\{{\cdot},{\cdot}\} & \Nscr^{(r)} & \Rscr^{(r)}\cr
\tablerule
\tablerule
\nucleo &\Nscr^{(r+s-2)} &\Rscr^{(r+s-2)}\cr\tablerule
\range  &\Rscr^{(r+s-2)} &\Pscr^{(r+s-2)}\cr\tablerule\tablerule
\cr}}
\end{equation}
\end{lemma}

\begin{proof}
For any pair of functions $f,\,g\,$, by Jacobi's identity for Poisson
brackets we have $\lie{\Omega}\{f,g\} = \{\lie{\Omega}f,g\} +
\{f,\lie{\Omega}g\}$.  If $f,g$ are in the respective kernels, then
$\lie{\Omega} \{f,g\} =0\,$, which proves that $\{f,g\} \in
\Nscr^{(r+s-2)}\,$.  If $g \in \Rscr^{(r)}$ then we may write $g=
\lie{\Omega} \lie{\Omega}^{-1}g\,$, and so if $f\in\Nscr^{(s)}$ we have,
still using Jacobi's identity, $\{f,\lie{\Omega} \lie{\Omega}^{-1}g\}
\!=\lie{\Omega}\{f,\lie{\Omega}^{-1}g\}$ in view of $\lie{\Omega}f=0\,$,
which proves that $\{f,g\}\in\Rscr^{(r+s-2)}\,$.  If $f,g$ are in the
respective ranges, then nothing can be said in general.  This gives the
table.
\end{proof}

\subsubsection{The linear operator $\lie{H_0}$}
\label{ssect.2.2}
We come now to the solution of the homological equation~(\ref{e.om.2s}).  In
view of~\eqref{e.dec.H0} we have $\lie{H_0} = \lie{\Omega} + \lie{Z_0}$, so
that we immediately get
\begin{displaymath}
 \lie{H_0} = \lie{\Omega}\tond{\Id + \lie{\Omega}^{-1}\lie{Z_0}} \ .
\end{displaymath}
Thus we have
\begin{equation}
\label{e.inv.lieH0}
 \lie{H_0}^{-1} = \tond{\Id + K}^{-1}\lie{\Omega}^{-1}\ ,
\qquad\qquad
 K := \lie{\Omega}^{-1}\lie{Z_0}\ ,
\end{equation}
and using the Neumann's series we can write
\begin{displaymath}
 \tond{\Id + K}^{-1} = \sum_{l\geq 0}(-1)^l K^l\ .
\end{displaymath} 

Let us consider $\lie{H_0}$ on the (finite dimensional) topological
space $\pspazio$; with the notation $\norm{\cdot}_{\rm op}$ we mean the
dual norm of a linear operator acting on $\pspazio$ (they are all
continuous). The following Proposition claims that, although we lack
informations about its Kernel and Range, we can invert $\lie{H_0}$ on
$\range$. This is one of the crucial tecnical points of the paper,
leading eventually to the independence of the two perturbative
parameters $a$ and $1/\beta$. See also the forthcoming
Remark~\ref{r.picc.div} on the control of small divisors.

\begin{proposition}
\label{p.2.1}
If the restriction of $K$ to $\range$ satisfies
\begin{equation}
\label{e.norm.K}
\norm{K}_{\rm op}<1 \ ,
\end{equation}
then for any $g\in\range$, there exists an element $f\in\range$ such that
\begin{equation*}
 \tond{\Id+K}f = g
  \qquad {\rm with} \qquad
  f = \sum_{l\geq 0}(-1)^l K^l g \ .
\end{equation*}
\end{proposition}

\begin{proof}
Let us take $g \in\range$, then from \eqref{pertur.54a} one has
$\lie{Z_0}g\in\range$, and also $\lie{\Omega}^{-1}\lie{Z_0}g =
Kg\in\range$; in other words
\begin{displaymath}
 K:\range\rightarrow\range \ .
\end{displaymath}
The sequence $\{K^l g\}$ is composed of elements of $\range$, and the same
holds for the finite sum
\begin{displaymath}
 f_n = \sum_{l=0}^n (-1)^l K^l g\in\range,\qquad\qquad n\geq 1 \ .
\end{displaymath}
Condition \eqref{e.norm.K} provides the convergence of the sequence
$f_n\to f$, with $f$ which belongs to $\range$, since it is a closed
subset of $\pspazio$. To prove that $f$ solves the required equation, we
consider the sequence
\begin{displaymath}
 g_n = \tond{\Id + K}f_n\in\range \ ;
\end{displaymath}
from the definition of $f_n$ we have $g_n = g + (-1)^n K^{n+1}g$, and
since $K^{n+1}g$ vanishes, the sequence $g_n$ converges to $g$. But the
continuity of $K$ implies also $g_n = \tond{\Id + K}f_n \to \tond{\Id +
K}f$, so the uniqueness of the limit gives the thesis.
\end{proof}

\subsection{Quantitative estimates and exponential decay of interactions}
Here we complete the formal setting of the previous sections by
producing all estimates of the norms of the relevant functions.  We
also prove the crucial property that the exponential decay of
interactions is preserved by our construction.

Recalling the definition~\eqref{succ.sigma} of $\sigma_s$, so that
$\sigma_1\gt\ldots\gt\sigma_r=\sigma_*\,$, our aim is to show that the
functions $\Chi_s$, $\Psi_s$ and $Z_s$ that are generated by the
formal construction are of class $\Dscr(\cdot,\sigma_s)$, with some
constant to be evaluated in place of the dot.

We shall repeatedly use the following elementary estimate. By the
general inequality
$$
 1-e^{-x} \ge x\frac{1-e^{-a}}{a} \quad{\rm for}\ 0\le x\le a\ . 
$$
we have
$$
\vcenter{\openup1\jot\halign{
\hfil$\displaystyle{#}$
&$\displaystyle{#}$\hfil
&\quad{\rm for}\ $\displaystyle{#}$\hfil 
\cr
1-e^{-\sigma_j} 
 &\ge \frac{\sigma_j(1-e^{-\sigma_0})}{\sigma_0}
  & 1\le j\le r\ ,
\cr
1 - e^{-(\sigma_j-\sigma_k)}
 &\ge \frac{(\sigma_j-\sigma_k)(1-e^{-(\sigma_0-\sigma_*)})}{\sigma_0-\sigma_*}
  & 1\le j\lt k\le r\ .
\cr
}}
$$
Moreover, in view of the definition~\eqref{succ.sigma} of
$\sigma_0,\ldots,\sigma_r$ for $0\le j\lt s\le r$ we get
\begin{equation}
\label{eq.denpois}
\vcenter{\openup1\jot\halign{
\hfil$\displaystyle{#}$
&$\displaystyle{#}$\hfil
\cr
1-e^{-\max(\sigma_j,\sigma_{s-j})} 
 &\ge \frac{1 -e^{-\sigma_0}}{\sigma_0}
   \max(\sigma_j,\sigma_{s-j})
 \gt \frac{(1 -e^{-\sigma_0})}{2}\ ,
\cr
1 - e^{-(\sigma_j-\sigma_k)}
 &\ge \frac{k-j}{r}(1-e^{-(\sigma_0-\sigma_*)})\ .
\cr
}}
\end{equation}

%
%
\paragraph{Estimate of the homological equation.}
\label{par:homeq}
We first consider the operator $\lie{\Omega}^{-1}$.

\begin{lemma}
\label{l.2.3}
Let $F=f^{\oplus}\in\Rscr^{(r)}$ be a cyclically symmetric homogeneous
polynomial of degree $r$ of class $\Dscr(C_f,\sigma)$.  Then there
exists a cyclically symmetric homogeneous polynomial $\Phi
= \ph^{\oplus}\in\range$ which solves $\lie{\Omega}\Phi = F$ and is
of class $\Dscr(C_{\ph},\sigma)$ with
\begin{equation*}
C_{\ph} \leq \frac{C_f}{2\Omega}\ .
\end{equation*}
\end{lemma}

\noindent
The proof is a straightforward consequence of the diagonal form of
$\lie{\Omega}$.

Coming to the inversion of $\lie{H_0}$, we state the following

\begin{lemma}
\label{l.yal}
Let $G=g^{\oplus}\in\Rscr^{(2s+2)}$ be a cyclically symmetric
homogeneous polynomial of degree $2s+2$ of class
$\Dscr(C_g,\sigma_s)$.  Let $K$ as defined in~\eqref{e.inv.lieH0} and
assume {
\begin{equation}
\label{CK}
C_K := \frac{4 C_{0} 
e^{-(\sigma_0-\sigma_1)}}{\Omega(1-e^{-\sigma_0})(1-e^{-(\sigma_0-\sigma_*)})}
\le \frac{1}{2r}\ .
\end{equation}}
Then there exists a cyclically symmetric homogeneous polynomial $\Xscr =
\chi^{\oplus}\in\Rscr^{(2s+2)}$ which solves $\lie{H_0}\Xscr = G$;
moreover $\chi$ is of class $\Dscr(C_g/\gamma,\sigma_s)$ with
\begin{equation}
\label{e.omol.1}
 \gamma = 2\Omega(1-rC_K)\ .
\end{equation}
\end{lemma}

\begin{remark}
\label{r.picc.div}
In Proposition~\ref{p.2.1} we ask $\norm{K}_{\rm op}<1$ to simply
perform the inversion. In the above Lemma~\ref{l.yal},
condition~\eqref{CK} reads as $\norm{K}_{\rm op}<1/2$, and this stronger
requirement is to control the small divisors~\eqref{e.omol.1}.
\end{remark}

\noindent
We emphasize that in view of the first of~\eqref{const.prop.gen} we
have $C_K=\mu/\mu_*\,$.  Therefore, condition~\eqref{CK} reads
$2r\mu<\mu_*$, which is the smallness condition for $\mu$ of
proposition~\ref{prop.gen}.  Furthermore this gives the value of
$\gamma$ in~\eqref{const.prop.gen}.

We also emphasize that the constant $\gamma$ is evaluated as
independent of $s$, but seems to depend on the degree $r$ of
truncation of the first integral.  However, in view of the condition
on $\mu$ we have $\Omega\leq\gamma\leq2\Omega$.

\begin{proof}
Recall that $\zeta_0$ is of class $\Dscr(C_0,\sigma_0)$, as stated in
lemma~\ref{l.exp.Z0}.  By corollary~\ref{cor.pois.Z0.corr}, with
$Z_0,\,\sigma_0$ and $\sigma_s$ in place of $f,\,\sigma'$ and
$\sigma''$, respectively, we see that $\lie{Z_0}g$ is of class
$\Dscr(C'C_g,\sigma_s)$ with
{
\begin{displaymath}
C' \le \frac{4 (s+1) C_0
  e^{-(\sigma_0-\sigma_s)}}{(1-e^{-\sigma_0})(1-e^{-(\sigma_0-\sigma_s)})}
\le \frac{8r C_0
  e^{-(\sigma_0-\sigma_1)}}{(1-e^{-\sigma_0})(1-e^{-(\sigma_0-\sigma_*)})}\ ,
\end{displaymath}}
where the second of \eqref{eq.denpois} has been used.  By using
lemma~\ref{l.2.3} we get that $Kg$ is of class $\Dscr(rC_KC_g,\sigma_s)$
with $C_K$ given by~\eqref{CK}.  This also implies that $\norm{K}_{\rm
op}\le rC_K$.  In view of condition~\eqref{CK} we may apply
proposition~\ref{p.2.1}, thus concluding that the inverse of $\lie{H_0}$
is well defined.  With an explicit calculation we also calculate
$\norm{K^m}_{\rm op}\le (rC_K)^m$, thus concluding that $\lie{H_0}^{-1}
g$ is of class $\Dscr(C_{g}^{\phantom{s}}/\gamma,\sigma_s)$ with
$\gamma$ as in~\eqref{e.omol.1}, as claimed.
\end{proof}

Having thus proved that the homological equation can be solved, the
statement~(i) of proposition~\ref{prop.gen} follows.

\paragraph{Iterative estimates on the generating sequence.}
We recall that the generating sequence is found by recursively solving
the homological equations $\lie{H_0}\chiph_{s} =Z_{s}+\Psi_{s}$ for
$s=1,\ldots,r$ with
\begin{equation}
\label{est.gs.1}
\vcenter{\openup1\jot\halign{
\hfil$\displaystyle{#}$
&$\displaystyle{#}$\hfil
&$\displaystyle{#}$\hfil
\cr
\Psi_1 &= H_1\ ,
\cr
\Psi_{s} &=
 \frac{s-1}{s}\lie{\Chi_{s-1}} H_1 
  + \sum_{l=1}^{s-1} \frac{l}{s} E_{s-l} Z_{l}\ ,
\cr
E_s Z_{l} &= \sum_{j=1}^{s} \frac{j}{s}\lie{\Chi_{j}}E_{s-j}Z_{l}
&\quad{\rm for}\ s\ge 1\ .
\cr
}}
\end{equation}
Our aim is to find positive constants $C_{\psi,1},\ldots,C_{\psi,r}$
so that $\Psi_{s}$ is of class $\Dscr(C_{\psi,s},\sigma_s)$.  In
view of lemma~\ref{l.2.3} this implies that $Z_{s}$ of class
$\Dscr(C_{\zeta,s},\sigma_s)$ with $C_{\zeta,s}=C_{\psi,s}$ and
$\chiph_{s}$ of class $\Dscr(C_{\chi,s},\sigma_s)$ with
$C_{\chi,s}=C_{\psi,s}/\gamma\,$.  Meanwhile we also find constants
$C_{\zeta,s,l}$ such that $E_{s} Z_{2l}$ is of class
$\Dscr(C_{\zeta,s,l},\sigma_{s+l})$ whenever $s+l\le r$.

We look for a constant $B_r$ and two sequences 
$\{\eta_s\}_{1\le s\le r}$ and $\{\theta_{s}\}_{1\le s\le r}$ such
that 
\begin{equation}
\label{est.gs.5}
\vcenter{\openup1\jot\halign{
\hfil$\displaystyle{#}$
&$\displaystyle{#}$\hfil
&$\displaystyle{#}$\hfil
\cr
C_{\psi,1} 
 &\le \eta_1C_1\ ,
\quad
C_{\zeta,0,1} \le \eta_1\theta_0 C_1\ ,
\cr
C_{\psi,s} &\le \frac{\eta_s}{s} {B_r^{s-1}C_1}
&\quad {\rm for}\ s\gt 1\ ,
\cr
C_{\zeta,s,l} &\le \theta_{s} \eta_l {B_r^{s+l-1}C_1}
&\quad {\rm for}\ s\ge 1\,,\> l\ge 1\ .
\cr
}}
\end{equation}
In view of $\Psi_1=H_1$ and of $E_0Z_{1}=Z_{1}$ we can choose
$\eta_1=\theta_0=1$.  By~(\ref{est.gs.1}) and using lemmas~\ref{l.2.3}
and~\ref{lem.poisson} together with corollary~\ref{cor.poisson} we get
the recursive relations
\begin{equation}
\label{est.gs.2}
\vcenter{\openup1\jot\halign{
\hfil$\displaystyle{#}$
&$\displaystyle{#}$\hfil
\cr
C_{\zeta,s,l} 
&\le \frac{4}{s} \sum_{j=1}^{s-1} 
 \frac{j(s+l-j)\eta_j\eta_l\theta_{s-j}}
      {(1-e^{-\max(\sigma_j,\sigma_{s+l-j})+\sigma_{s+l}})
       (1-e^{-\max(\sigma_j,\sigma_{s+l-j})})}
        \frac{B_r^{s+l-2} C_1^2}{\gamma}\ .
\cr
C_{\psi,s} 
&\le \biggl(
  \frac{8(s-1) \eta_{s-1} C_1}{s
        (1-e^{-(\sigma_0-\sigma_{s})})(1-e^{-\sigma_0})} 
  + \sum_{l=1}^{s-1} \frac{lB_r}{s} \eta_l\theta_{s-l}
  \biggr) 
 \frac{B_r^{s-2}C_1}{\gamma}\ ,
\cr
}}
\end{equation}
We observe that, from the first of \eqref{eq.denpois}, we have
\begin{align*}
1-e^{-\max(\sigma_j,\sigma_{s+l-j})}
&\geq \frac{1-e^{-\sigma_0}}{\sigma_0}\max\{\sigma_j,\sigma_{s+l-j}\}
=\\ &= \frac{1-e^{-\sigma_0}}{\sigma_0}\tond{\sigma_0
- \frac{\sigma_0-\sigma_*}{r}\min\{j,s+l-j\}}
>\\&> \frac{1-e^{-\sigma_0}}{\sigma_0} \tond{\frac{\sigma_0+\sigma_*}2}\
,
\end{align*}
since $\min\{j,s-j\}\leq r/2$, thus it gives
\begin{displaymath}
1-e^{-\max(\sigma_j,\sigma_{s+l-j})} > \frac{1-e^{-\sigma_0}}{2}\ .
\end{displaymath}
Using the second of~\eqref{eq.denpois} in a similar way to deal with
\begin{displaymath}
1-e^{-[\max(\sigma_j,\sigma_{s+l-j})-\sigma_{s+l}]} \geq \frac{s+l-\min\{j,s+l-j\}}{r}\tond{1-e^{-(\sigma_0-\sigma_*)}}\
,
\end{displaymath}
and setting 
\begin{equation*}
B_r = \frac{16 C_1 r}{\gamma
  (1-e^{-(\sigma_0-\sigma_*)})(1-e^{-\sigma_0})}\ .
\end{equation*}
we get
\begin{equation*}
\vcenter{\openup1\jot\halign{
\hfil$\displaystyle{#}$
&$\displaystyle{#}$\hfil
\cr
C_{\zeta,l,s} 
&\le \frac{1}{s} \sum_{j=1}^{s} 
      {j\eta_j\eta_l\theta_{s-j}}\,
       {B_r^{s+l-1} C_1}\ ,
\cr
C_{\psi,s} 
&\le \biggl(
      \frac{1}{s}\eta_{s-1} 
      +\sum_{l=1}^{s-1} 
        \frac{l}{s} \eta_l\theta_{s-l}
     \biggr)
  {B_r^{s-1} C_1}\ ,
\cr
}}
\end{equation*}
Therefore the required inequalities~(\ref{est.gs.5}) are satisfied by
the sequences recursively defined as
$$
\vcenter{\openup1\jot\halign{
\hfil$\displaystyle{#}$
&$\displaystyle{#}$\hfil
&$\displaystyle{#}$\hfil
\cr
\theta_{s} &=
\sum_{j=1}^{s} \frac{j}{s} \eta_j\theta_{s-j} 
&\quad{\rm for}\ s\ge 1\ ,
\cr
\eta_s &=
 \eta_{s-1} +\sum_{j=1}^{s-1} j \eta_j\theta_{s-j} &\quad{\rm for}\
      s\ge 2\ .
\cr
}}
$$
starting with $\eta_{1} =\theta_{0} = 1$.  Actually, in order to find
an estimate for the generating function it is enough to investigate
the sequence $\eta_1,\ldots,\eta_r\,$.  To this end, after
multiplication by a factor $1/s$, we subtract the second relation from
the first one, thus getting
$$
\theta_1 = 1\ , \quad \theta_s = \tond{\frac{s+1}{s}}\eta_s
- \frac{1}{s}\eta_{s-1} < 2\eta_s - \frac{1}{s}\eta_{s-1}\ .
$$
Then we substitute the latter expression for $\theta_{s-j}$ in the
second of the relations above, and get
$$
\eta_s \lt \eta_{s-1} + \sum_{j=1}^{s-1}{2j} \eta_j\eta_{s-j} < 3\sum_{j=1}^{s-1}{j} \eta_j\eta_{s-j}\ .
$$
Hence the wanted inequality~(\ref{est.gs.5}) for $C_{\psi,s}$ is
satisfied by the sequence
$$
\eta_1 = 1\ ,\quad
\eta_s = 3\sum_{j=1}^{s-1} j\eta_j\eta_{s-j} = 3s\sum_{j=1}^{\lfloor
s/2\rfloor}\eta_j\eta_{s-j}\ ,\qquad s\geq 2\ .
$$
By induction it is possible to prove that $\eta_s \leq 9^{s-1} s!$
for all $s=1,\ldots,r$: indeed it holds
\begin{displaymath}
 x_s \leq 9^{s-1}\frac{s}3\sum_{j=1}^{\lfloor s/2\rfloor}
          j! (s-j)! \leq 9^{s-1} s! \ ,
\end{displaymath}
provided
\begin{equation*}
 \sum_{j=1}^{\lfloor s/2\rfloor} j! (s-j)! =
 \sum_{j=2}^{\lfloor s/2\rfloor} j! (s-j)!\leq 2(s-1)! \ ;
\end{equation*}
the latter being true since for $4\leq s\leq r$ and $2\leq j\leq \lfloor
s/2\rfloor$
\begin{displaymath}
j! \frac{(s-j)!}{(s-1)!}
= \prod_{i=0}^{j-2}\tond{\frac{j-i}{s-j-i}} \leq \tond{\frac23}^{j-1}\ .
\end{displaymath}
Then, by $s!\leq (\sqrt e)^{-(s-1)} s^s$ we obtain for $1\leq s\leq r$
\begin{displaymath}
\eta_s \leq \tond{\frac{9}{\sqrt e}}^{s-1} s^s < 4^{s-1} r^{s-1}
\end{displaymath}
Replacing this and~\eqref{est.gs.2} in the inequality~(\ref{est.gs.5}) for
$C_{\psi,s}$ and recalling that $C_{\chi,s}\le C_{\psi,s}/\gamma$ we have
\begin{equation*}
 C_{\chi,s} \le (64r^2C_*)^{s-1}\frac{C_1}{\gamma s} \ ,
  \qquad
  C_* = \frac{C_1}{\gamma(1-e^{-\sigma_0})(1-e^{-(\sigma_0-\sigma_*})} \ .
\end{equation*}
The proves the statement~(ii) of proposition~\ref{prop.gen} with the estimated
value of $C_*$ in~\eqref{const.prop.gen}.  The statement~(iii) also follows in
view of $C_{\zeta,s}\le C_{\psi,s}\,$.

\paragraph{Estimate of the truncated first integral.}
We give an estimate for the first $r$ terms of $T_{\Chi}\Phi$ where
$\Phi$ is a homogeneous polynomial, as specified in the statement~(iv)
of proposition~\eqref{prop.gen}.  We look for a sequence $C_{\ph,s}$
of constants such that $E_s\Phi$ is of class
$\Dscr(C_{\ph,s},\sigma_s)$ for $s=0,\ldots,r$.  Of course we have
$C_{\ph,0}=C_{\ph}$, so we look for a recursive estimate for $s\gt
0$ using lemma~\ref{lem.poisson} and the definition~\eqref{e.TelChi}
of $T_{\Chi}$.  Recalling~\eqref{eq.denpois} we have that $E_s\Phi$ is
of class $\Dscr(A,\sigma_s)$ with a constant $A$ satisfying
$$
\vcenter{\openup1\jot\halign{
\hfil$\displaystyle{#}$
&$\displaystyle{#}$\hfil
\cr
A &\le
 \sum_{j=1}^{s} \frac{j}{s}\cdot
   \frac{4(j+1)(s-j+m+1)}{(1-e^{-\max(\sigma_j,\sigma_{s-j})+\sigma_s})
     (1-e^{-\max(\sigma_j,\sigma_{s-j})})} 
  \cdot \frac{C_1}{j\gamma} (C_r)^{j-1} C_{\ph,s-j}
\cr
&\le 
 \sum_{j=1}^{s} \frac{(s-j+m+1)}{\max(j,s-j)} \cdot
  \frac{16C_1 r}{\gamma(1-e^{-\sigma_0 })(1-e^{-(\sigma_0-\sigma_* )})}
   \, (C_r)^{j-1} C_{\ph,s-j}
\cr
}}
$$
Thus, recalling the definition~\eqref{const.prop.gen} of $C_*$, we may
set
\begin{equation*}
 C_{\ph,s} = \frac{1}{4} \sum_{j=1}^{s} \frac{s-j+m+1}{\max(j,s-j)} 
    (C_r)^{j} C_{\ph,s-j}\ .
\end{equation*}
For $s=1$ this gives 
\begin{equation}
\label{chi.phi.1}
 C_{\ph,1} = (m+1) r^2 C_* C_{\ph} \ ,
\end{equation}
so that the claim is true with $F_r$ as given in~\eqref{const.prop.gen}.
For $s\gt 1$ we extract from the sum the term $j=1$ and replace the
index $j$ with $j+1$ in the rest of the sum, thus getting
$$
\vcenter{\openup1\jot\halign{
\hfil$\displaystyle{#}$
&$\displaystyle{#}$\hfil
\cr
C_{\ph,s} &\le
 \frac{(s+m)}{4(s-1)} C_r C_{\ph,s-1} +
  \frac{C_r}{4} 
   \sum_{j=1}^{s-1} \frac{s-j+m}{\max(j+1,s-1-j)} 
    (C_r)^{j} C_{\ph,s-1-j}\ ,
\cr
&\le \frac{s+m}{4(s-1)} C_r C_{\ph,s-1}  
  +\frac{1}{4} C_r C_{\ph,s-1}
\cr
&\le \frac{m+2}{4} C_r C_{\ph,s-1}\ .
\cr
}}
$$
This proves the statement~(iv) of proposition~\ref{prop.gen}.  

Concerning the statement~(v), a remark that $\lie{\Chi_1}H_{\omega} =
-\lie{\Omega}\Chi_1=-Z_1-H_1$ in view of the homological equation at order 1.
Therefore we may replace~\eqref{chi.phi.1} with $C_{\ph,1} = C_1$.  For $s\gt
1$ the argument above for a generic function $\Phi$ requires only a minor
modification and one obtains the same recursive relation for $C_{\ph,s}\,$,
where we just replace a different value for $C_{\ph,1}\,$.  This proves the
claim.

\paragraph{Estimate of the time derivative of the approximate first integral.}

We come to the statement~(vi) of proposition~\ref{prop.gen}.  Recall
that by construction we have $T_{\Chi}Z -H=\Oscr(r+2)$, meaning that
its expansion starts with terms of degree at least $2(r+2)$.  Since
$\lie{\Omega}Z_s=0$ for $s=0,\ldots,r$ and recalling the general
property $T_{\Chi}\Poi{f}{g}=\Poi{T_{\Chi}f}{T_{\Chi}g}$ we
immediately have
$$
\Poi{H}{T_{\Chi}\Phi_0} = T_{\Chi}\Poi{Z}{\Phi_0} = \Oscr(r+2)
$$
On the other hand, since $T_{\Chi}\Phi_0-\Phi^{(r)}=\Oscr(r+2)$, we
also have $\Poi{H}{\Phi^{(r)}} = \Oscr(r+2)$.  Substituting the
expansions $H=H_0+H_1$ and $\Phi^{(r)}=\Phi_0+\ldots+\Phi_r$ we get
$\dot\Phi^{(r)} = \Poi{H}{\Phi^{(r)}}=\Poi{H_1}{\Phi_r}$, which is an
extensive homogeneous polynomial of degree $2r+4$, as claimed.
Recalling that $H_1$ is of class $\Dscr(C_1,\sigma_0)$ and $\Phi_r$ is
of class $\Dscr(F_r^{r-1}C_{1},\sigma_r)$, in view of the
statement~(v), a straightforward application of
lemma~\ref{lem.poisson} gives
$$
C_{\rho} \le \frac{8(r+2)F_r^{r-1}
C_1^2}{(1-e^{-\sigma_0})(1-e^{-(\sigma_0-\sigma_*)})}\ ,
$$
and the result follows by just replacing the estimated value of $F_r$
from statement~(iv), with $m=0$.

This concludes the proof of proposition~\ref{prop.gen}.


\section{Long time estimates and statistical control of
  fluctuations}
\label{s:5}

In this Section we actually present, and prove, the main result of the paper
in its complete and detailed form; the results given in the introduction,
i.e. Theorems~\ref{t.main.power} and~\ref{t.main.exp}, are simplified
statements with some particular choices of the parameters involved.

We first stress that, although the whole perturbative construction of
our conserved quantity $\Phi\equiv\Phi^{(r)}$ is based upon an initial
normal form transformation of the quadratic part of the Hamiltonian,
i.e. there is a change of coordinates at the very beginning of our
procedure, we will state our result and the corresponding proof in the
original\footnote{We will take care of this via the application of
Lemma~\ref{l.inv.g} throughout the proof.} variables $z=(x,y)$.

As explained in the introduction, our aim is to show that $\Phi$ is a good
adiabatic invariant over a long time scale: to this purpose we introduce its
variation over a time interval
\begin{equation*}
 \Delta_t\Phi(z):= \Phi\bigl(\phi^t(z)\bigr) - \Phi((z)) \ ,
\end{equation*}
where $\phi^t(z)$ is the Hamiltonian flow. We will show that $\Delta_t\Phi$
remains small, compared with the phase variance of $\Phi$, over a long time
scale, for a set of initial data $z$ of large Gibbs measure. This kind of
control is quite weak for all the times between $0$ and $t$, since the set of
large measure is in principle allowed to change if we change the final $t$ in
order control the intermediate times. We thus give two stronger estimates: the
first deal with $\overline{\Delta_t\Phi}$. Its smallness imply that for every
large deviation of a given sign at intermediate times must correspond a
similar deviation with the opposite sign. An even stronger control is obtained
with the smallness of $\sigma^2_t[\Delta_t\Phi]$: in this case we have that
$\Delta_s\Phi$ is small also for all $s\in(0,t)$.

\subsection{Main result}
\label{ss:adiab}

Concerning the time scale over which we are able to control the
evolution of our adiabatic invariant, we have actually two types of
estimates, as a power law and a stretched exponential, each with its own
set of hypothesis and constants, but with a similar formulation; we thus
present the two results together.  In order to simplify the statement,
we find it convenient to formulate in advance the hypothesis and
definitions under which the result holds in those two cases. In
particular we define the time scale $\bar t$ and the corresponding
bounds on $\beta$.

Given the constants\footnote{See Propositions~\ref{prop.gen},
\ref{p.rho2.up}, \ref{p.lower.sig} and Lemmas~\ref{l.part.N} and
\ref{l.un-nono}} $a_0$, $\mu_*$, $\mu_2$, $\beta^*$, $\beta_0$,
$\beta_1$, $\beta_3$, $\beta_4$, $K_1$, defining
\begin{equation*}
 \mu_0 := \frac{a_0}{1+2a_0} \ ,
\quad
 \mu_1 := \frac{\left(1-\frac3{4(\max\{\beta_0,1\})^2}\right)^8}
               {64 K_1^8(1+4a_0)^4} \ ,
\end{equation*}
we introduce

\begin{description}
\item[{\bf HD1 (power law estimate)}] there exist $\beta_2>0$ and $r^*(\mu) =
	   \frac{\mu_*}{2\mu}$ such that for any integer $r\in[1,r^*)$ and for
	   any $\nu\in (0,1]$, defining 
\begin{align*}
 \beta_* &:= \max\left\{\beta_0,\beta_1,\beta_2, \beta_3,\beta_4,
                        (\beta^* r^3)^{1/\nu},\sqrt{3/2} \right\} \ ,
\cr
 \mu^* &:= \min\left\{\mu_0,\mu_1,\mu_2,\frac18 \right\} \ ,
\end{align*}
then
\begin{displaymath}
{\rm assume}\qquad
\beta_* \leq \beta \ ,
\qquad{\rm and\ define}\qquad
\begin{aligned}
 \lambda&:=r(1-\nu)+1-\nu/2 \ ,
\cr
 \overline t&:={\beta^\lambda} \ .
\end{aligned}
\end{displaymath}
\item[{\bf HD2 (exponential estimate)}] there exists $\mu_3>0$ such that
	   defining 
\begin{align*}
 \beta_* &:= \max\left\{\beta_0,\beta_1,\beta_3,\beta_4,64e\beta^*,
                        \sqrt{3/2}\right\}
\cr
 \mu^* &:= \min\left\{\mu_0,\mu_1,\mu_2,\mu_3,
                      \mu_*\sqrt[3]{\frac{e\beta^*}{\beta_*}},
                      \frac18 \right\} \ ,
\end{align*}
then
\begin{displaymath}
{\rm assume}\qquad
\beta_* \leq \beta < e\beta^*\;\tond{\frac{\mu_*}{\mu}}^3
\qquad{\rm and\ define}\qquad
\begin{aligned}
 \kappa&:=\frac32\sqrt[3]{ \frac{\beta}{e\beta^*} } \ ,
\cr
 \overline t&:=\frac{\kappa^{9/2}e^{\kappa/2}}\beta \ .
\end{aligned}
\end{displaymath}
\end{description}
We are now ready to state the result:
\begin{theorem}
\label{t.Adiab0}
For either the hypothesis and definitions of case {\bf HD1} or those of
case {\bf HD2}, there exist constants $K>1$ such that, for all
$\mu<\mu^*$, and for any positive $\delta$ one has
\begin{equation*}
\begin{aligned}
 m\Bigl( z\in\RR^{2N}\ \colon\  \left|\Delta_t\Phi(z)\right| \geq
  \delta\sigma[\Phi] \Bigr)
&\leq  \frac{12K}{\delta^2}\tond{\frac{t}{\overline t}}^{2} \ ,
\cr
 m\Bigl( z\in\RR^{2N}\ \colon\  \left|\overline{\Delta_t\Phi}(z)\right| \geq
  \delta\sigma[\Phi] \Bigr)
&\leq  \frac{3K}{\delta^2}\tond{\frac{t}{\overline t}}^{2} \ ,
\cr
 m\Bigl( z\in\RR^{2N}\ \colon\  \sigma^2_t[\Delta_t\Phi(z)] \geq
  \delta\sigma^2[\Phi] \Bigr)
&\leq  \frac{4K}{\delta}\tond{\frac{t}{\overline t}}^{2} \ .
\end{aligned}
\end{equation*}
\end{theorem}

\begin{remark}
The estimates contained in the above theorem can be seen as the
generalizations of Propositions 2, 3 and 4 of paper~\cite{GioPP12}.  The
result of paper~\cite{CarM12} can be compared with the first estimate,
with the hypothesis and definition set {\bf HD2} in the case of
vanishing coupling constant, the only difference being a slightly
improved exponent for the argument of the exponential ($1/3$ in our
case, $1/4$ in their result).
\end{remark}

\begin{proof}
The first two estimates of the theorem are actually Tchebychev estimates,
while the third is a Markov estimate. We recall that, in our notations, given
any measurable function $f$ and any real $\eta>0$, for $p=1$ and 2
respectively, Markov and Tchebychev estimates are
\begin{equation*}
 m\Bigl(z\in\RR^{2N}\ \colon\ |f(z)| \geq \eta \Bigr) \leq
  \frac{\inter{|f|^p}}{\eta^p} \ .
\end{equation*}
Choosing $\eta=\delta\sigma[\Phi]$, and using Tchebychev for the first two
estimates, and with $\eta=\delta\sigma^2[\Phi]$ and using Markov for the third
one, one has to control respectively the following three quantities
\begin{equation}
\label{e.3ratios}
 \frac{\inter{\left(\Delta_t\Phi\right)^2}}{\delta^2\sigma^2[\Phi]} \ ,
\qquad\qquad
\frac{\inter{\left(\overline{\Delta_t\Phi}\right)^2}}{\delta^2\sigma^2[\Phi]}\ ,
\qquad\qquad
  \frac{\inter{\sigma^2_t[\Delta_t\Phi]}}{\delta\sigma^2[\Phi]} \ .
\end{equation}

By the rough inequality $\sigma^2_t[\Delta_t\Phi(z)] \leq
\overline{(\Delta_t\Phi)^2}(z)$ it is clear that to estimate the three
quantities above, we need to control the phase average of
$\left(\Delta_t\Phi\right)^2$, $\left(\overline{\Delta_t\Phi}\right)^2$ and
$\overline{(\Delta_t\Phi)^2}$. By defining
\begin{equation}
\label{e.R.1}
R := \Poi{\Phi^{(r)}}{H} = \Poi{\Phi_r}{H_1},
\end{equation}
we have $\Delta_t\Phi(z) = -\int_0^t R\circ\phi^s(z)ds$, so that we
may write
\begin{align*}
 \inter{\left(\Delta_t\Phi\right)^2} &= 
  \inter{\int_{[0,t]^2} \left(R\circ\phi^{s_1}\right) 
         \left(R\circ\phi^{s_2}\right) ds_1ds_2} =
\cr
  &=\int_{[0,t]^2} \inter{\left(R\circ\phi^{s_1}\right) 
         \left(R\circ\phi^{s_2}\right) } ds_1ds_2 \leq
\cr
  &\leq \int_{[0,t]^2} \norm{\left(R\circ\phi^{s_1}\right)}_{L^2}
         \norm{\left(R\circ\phi^{s_2}\right)}_{L^2} ds_1ds_2 =
\cr
  &= \int_{[0,t]^2} \norm{R}^2_{L^2} ds_1ds_2 =
        t^2 \inter{R^2} \ .
\end{align*}
where we used Fubini's theorem, Schwartz inequality and the invariance
of the measure under the Hamiltonian flow. For the second quantity we
need to add a further double integration over time to perform the time
average, but the scheme is the same:
\begin{align*}
 \inter{\left(\overline{\Delta_t\Phi}\right)^2} &= 
  \inter{\frac1{t^2}\int_{[0,t]^2} \Delta_{s_1}\Phi\Delta_{s_2}\Phi ds_1ds_2} =
\cr
  &= \frac1{t^2}\int_{[0,t]^2} \inter{\int_{[0,s_1]\times[0,s_2]}
            \left(R\circ\phi^{\tau_1}\right)
            \left(R\circ\phi^{\tau_2}\right) d\tau_1d\tau_2} \, ds_1ds_2 \leq
\cr
  &\leq      \frac{t^2}4 \inter{R^2} \ .
\end{align*}
In the third case we instead have a single time integration from the
time average:
\begin{align*}
 \inter{\overline{(\Delta_t\Phi)^2}} 
  = \inter{\frac1{t}\int_{[0,t]} (\Delta_{s}\Phi)^2 ds}
 &= \frac1{t}\int_{[0,t]} \inter{(\Delta_{s}\Phi)^2} ds \leq
\cr
 &\leq \frac1{t}\int_{[0,t]} s^2\inter{R^2}ds = 
       \frac{t^2}3 \inter{R^2} \ .
\end{align*}

For all the three quantities~\eqref{e.3ratios}, everything we thus
need to control the quotient of $\inter{R^2}$ over $\sigma^2[\Phi]$,
with an upper bound for the numerator and a lower bound for the
denominator: for the former we apply Proposition~\ref{p.rho2.up}, and
for the latter we use Proposition~\ref{p.lower.sig}. In particular,
the hypothesis and definition sets {\bf HD1} and {\bf HD2} imply the
hypothesis of Proposition~\ref{p.rho2.up} part {\sl 1} and,
respectively, part {\sl 2}.

Indeed if {\bf HD1} holds, from $\mu<\mu^*$ we have that
$a<\min\{a_0,1/6\}$ ($\mu^*<\min\{\mu_0,1/8\}$) and\footnote{The
  constant $D$ is defined in Appendix~\ref{app.gibbs} and recalled in
  Section 5.2, while $\mu_\flat$ in~\eqref{e.decomp.rho}.}
$D^2\mu_\flat<1/2$ ($\mu<\mu_1$). The condition $\mu<\mu_2$ is
required in Lemma \ref{l.un-nono}, for the lower bound of
$\sigma^2[\Phi]$. With respect to the hypothesis of
Proposition~\ref{p.lower.sig} we observe that if {\bf HD1} holds, the
bounds on $\beta$ are satisfied (for $\beta$ large enough, setting
$\beta_2$ as the threshold) since we need it to be scaling like $r^3$
and according to {\bf HD1} we have it scaling as $r^{(3/\nu)}$ with
$\nu\leq1$.

If {\bf HD2} holds, since in that case from
Proposition~\ref{p.rho2.up} we set the optimal integer $r$ as
$\lfloor\kappa/3\rfloor$, the condition on $\beta$ translate in the
following inequality
\begin{equation*}
 1 > \frac34 e^{\tilde C -1} \tond{1+2\sqrt[3]{\frac{e\beta^*}\beta}}
\end{equation*}
which is true since $\tilde C$ vanishes with $\mu$ (set here $\mu_3$ as the
threshold).

Using also the constants $\Omega$ (Proposition~\ref{p.1}), $C_1$
(Proposition~\ref{prop.gen}) and $K_2$ (Proposition~\ref{p.aver.2}), setting
\begin{equation*}
K:=\frac{3^2 5 e^6 }{2} \frac{K_1^2 K_2 \Omega^4}{C_1^2}\cdot
\left\{\begin{aligned}
&2{\beta^*}^3  \qquad &{\rm case\ {\bf HD1}}
\cr
&\frac{3^8}{e^4}  \qquad &{\rm case\ {\bf HD2}}
\end{aligned}
\right.
\end{equation*}
we have the thesis.
\end{proof}

\paragraph{Proof of Theorem~\ref{t.main.power}}
After observing that $\sigma^2_t \left[\Delta_t\Phi\right] = \sigma^2_t
\left[\Phi\right]$, use the third estimate of Theorem~\ref{t.Adiab0},
hypothesis and definitions set {\bf HD1}, with $r=\lfloor r^*\rfloor$,
$\nu=\frac12$, $\delta = \beta^{-1/2}$ and letting only $\beta^{r/2}$ in the
time scale $\bar t$. \qed

\paragraph{Proof of Theorem~\ref{t.main.exp}}
Apply Theorem~\ref{t.Adiab0}, hypothesis and definitions set {\bf
  HD2}, third estimate, with $\delta = \beta^{-1/2}$ and letting only
$e^{c\kappa}$ in the time scale $\bar t$ with a constant $c$ slightly
smaller than $1/2$ in order to get the correct power of $\beta$
outside the exponential factor. The upper bound on
$\sqrt{a}\sqrt[3]\beta$ represents last condition in {\bf HD2} using
the definition \eqref{e.mu.def} of $\mu$. \qed

The rest of the Section is devoted to the proofs of the upper bound of
$\inter{R^2}$, in subsection~\ref{sss:av.R2}, and of the lower bound of
$\sigma^2[\Phi]$ in subsection~\ref{sss:lower.sig}. Due to its relevance and
to the slightly different techniques involved, we anticipate in
subsection~\ref{ss.gibbs} the result on the control of the decay of
correlations.

%
%

\subsection{Decay of correlations}
\label{ss.gibbs}

The main result of this Section is an estimate of the correlation
between two polynomials with disjoint supports: we show such a
correlation to be (at least) small as $a^d$ where $d$ is the distance
between the two supports.

Let $\beta_0$, $a_0$ and $K_1$ be the constants of Lemma~\ref{l.part.N}.
Consider also these other constants\footnote{$A_j$ are actually functions of
$a$, $B$ is a function of $\beta$ and $D$ is a function of both the
parameters, but all these quantities are asymptotically constants as
$a\to0$ and $\beta\to\infty$.}  defined in Appendix~\ref{app.gibbs}:
\begin{equation*}
A_1=\sqrt{1+4a} \ ,
\qquad
A_2=\sqrt{1-2a} \ ,
\qquad
B=1-\frac3{4\beta^2} \ ,
\qquad
D=\frac{K_1A_1(a)}{B(\beta)}  \ ;
\end{equation*}
Introduce the following:
\begin{equation}
\label{e.C.et.alter}
 \mu_\sharp:=a(2+a) \ ,
\qquad
 K_2:=\frac{4(2K_1)^{2+8a_0}}{(1-2a_0)^4(1-a_0)^6} \ .
\end{equation}

\begin{proposition}
\label{p.aver.2}
Let $N$ be the length of the periodic chain.  Let $\phi$ and $\psi$ be two
homogeneous polynomial of degree $2r$ and $2s$ respectively, and interaction
length $m$ and $m'$ respectively; suppose their supports are disjoint, and
denote by $d$ their distance\footnote{If $p=\min S(\phi)$, $q=\max S(\phi)$,
$t=\min S(\psi)$ and $u=\max S(\psi)$, with $q<t$, then
$d=\min(t-q-1,N-u+p-1)$.}, then for any $\beta>\beta_0$ and $a<a_0$ it holds
\begin{equation*}
 \left|\inter{\phi\psi} - \inter{\phi}\inter{\psi}\right| \leq
 K_2 \;
 \left[D^{m+m'+2d+4}\right] \;
 \mu_\sharp^d \;
 \left[\frac{2^{r+s}r!s!}{(A_2^2\beta)^{r+s}}\right] \;
 \norm{\phi}\norm{\psi} \ .
\end{equation*}
\end{proposition}

Before entering into the details of the proof it is necessary to
introduce another notation for the measure that will be useful also in
the sequel of this Section. To this purpose we split the original
Hamiltonian in a different way. We recall that $H$ is naturally split in
two different terms $H_0$ and $H_1$ (see \eqref{e.H.dec}) according to
the degree, but here we want to put into evidence the coupling terms of
$H$. There are two possible choices, the first being to separate all the
terms depending on the coupling constant $a$, i.e. $\frac{a}2
\sum_j(x_{j+1}-x_j)^2$.  We instead separate the diagonal and the
off-diagonal part of $A$ (see \eqref{e.def-A} and \eqref{e.dec.H0a})
like in Proposition \ref{p.1}, but maintaining the original variables;
in this way we put into evidence the real coupling terms. Accordingly we
define, on a subset of variables, the uncoupled component of the Gibbs
measure by
\begin{equation*}
 dV_s^{(m)}:=\prod_{j=s}^{s+m-1}
             e^{-\beta \left[(1+2a)\frac{x_j^2}{2} + \frac{x_j^4}{4}\right]}
	     dx_j \ ,
\end{equation*}
which depend only on $m$ variables, and the coupling part by
\begin{equation*}
 [p,q]:=e^{\beta ax_px_{p+1}} \cdots e^{\beta ax_{q-1}x_q}    \ ,
\end{equation*}
for\footnote{We remark that in this notation, given the generic dimension $l$
of the space, the indexes must be considered modulo $l$.} $p<q\leq p+N$.  We
observe that, for any $m<l$, it is possible to factorize both the component of
the measure: $dV_s^{(l)}=dV_s^{(m)}dV_{s+m}^{(l-m)}$ and
$[s,s+l]=[s,s+m][s+m,s+l]$.  Whenever $S(\phi)\subset\{p,\ldots,q\}$, we will
write $[p,\phi,q]:=\phi[q,p]$, to stress the bound on the support. The full
Gibbs measure, again ignoring the $y$ variables, is then given by
$[0,l]dV_0^{(l)}$ for a system with periodic boundary conditions\footnote{for
a system with free boundary conditions the measure is $[0,l-1]dV_0^{(l)}$, so
the partition function will be
\begin{equation}
\label{e.calZH}
{\mathcal Z}_l := \int_{\RR^l}[0,l-1]dV_0^{(l)}
    =  \int_{\RR^l}e^{-\beta {\mathcal H}(x)}dx \ .
\end{equation}
}, so the
partition function will be
\begin{equation*}
Z_l := \int_{\RR^l}[0,l]dV_0^{(l)}
    =  \int_{\RR^l}e^{-\beta H(x)}dx \ .
\end{equation*}

\begin{proof}
The proof consists of two main steps: in the first one we show the presence of
several cancellations, while in the second one we actually estimate the
remaining terms.

Without any loss of generality we may assume $\phi$ to be left aligned, so
that its support is contained in $\{0,\ldots,m-1\}$; denote by $t$ the minimal
index in the support of $\psi$, which will be therefore contained in
$\{t,\ldots,t+m'-1\}$.

As a first step we rescale all the variables by a factor $\sqrt\beta$. We need
to introduce a corresponding notation for the relevant objects with the
rescaled variables, and to this purpose we will systematically add a $\scal$
as a superscript when needed. We remark that the powers of $\beta$ appearing
as a multiplying factor will not be included in the $\scal$-objects. For
example we have:
\begin{equation}
\label{e.mis.scal}
 [p,q]^\scal:=e^{ax_px_{p+1}} \cdots e^{ax_{q-1}x_q} \ ,
 \qquad
 dV_s^{(m)\scal}:=\prod_{j=s}^{s+m-1}
      e^{-\left[(1+2a)\frac{x_j^2}{2} + \frac{x_j^4}{4\beta^2}\right]}
      dx_j   \ .
\end{equation}
If we perform such a scaling on the correlation we get
\begin{equation*}
 \inter{\phi\psi} - \inter{\phi}\inter{\psi} =
  \frac1{\beta^{r+s}}
\Bigl( \inter{\phi\psi}^\scal - \inter{\phi}^\scal\inter{\psi}^\scal \Bigr)
\end{equation*}

We need to introduce a further notation related to the coupling terms
$[p,q]^\scal$ of the measure. These terms are products of factors of the form
$e^\alpha$, each being the coupling term between two consecutive sites of the
chain; in order to possibly decouple the chain in several positions we use the
trivial identity $e^\alpha = 1 + \left(e^\alpha-1\right)$, so that for example
$[0,m]^\scal$ turns out to be the sum of $2^m$ terms each of which is the
product of $m$ factors: for every $j=0,\ldots,m-1$ the factor can be either
$1$ or $e^{a x_jx_{j+1}}-1$. We will identify each term in the sum with a
string of $m$ symbols in $\{0,1\}$: $0$ for the factor $e^{a x_jx_{j+1}}-1$,
and $1$ for the factor $1$. For example, for $m=4$, a possible factor is
\begin{equation*}
 {\tt k} = {\tt 0010} \rightarrow [0,m]_{\tt k}^\scal = 
  \left(e^{a x_0x_1}-1\right) \cdot 
  \left(e^{a x_1x_2}-1\right) \cdot
  1 \cdot
  \left(e^{a x_3x_4}-1\right) \ .
\end{equation*}
Thus we may write $[0,m]^\scal=\sum_{{\tt k}}[0,m]_{\tt k}^\scal$: we will use
this kind of expansion for the coupling terms involving sites outside the
support of the polynomials.  \emph{We remark that a factor of the type
$e^\alpha-1$ is roughly of order $\alpha$, i.e. in our case of order $a$, so
the number of ``zeros'' in the sequence {\tt k} can be used to quantify the
smallness of the corresponding term.}

Let us now rewrite the correlation collecting the partition function in
the denominators
\begin{equation}
\label{e.corr.s1}
 \inter{\phi\psi}^\scal - \inter{\phi}^\scal\inter{\psi}^\scal = 
  \frac{\dinter{\phi\psi}^\scal Z^\scal - 
  \dinter{\phi}^\scal\dinter{\psi}^\scal}{Z^{\scal 2}} \ ,
\end{equation}
where we used the notation $\dinter{\phi} := Z\cdot\inter{\phi}=\int \phi(x)
e^{-\beta H(x)} dx$, and let us concentrate our attention on the numerator.

According to the supports of $\phi$ and $\psi$, we split the coupling
part of the measure in the following way
\begin{equation*}
[0,N]^\scal = [0,m-1]^\scal \cdot [m-1,t]^\scal \cdot
              [t,t+m'-1]^\scal \cdot [t+m'-1,N]^\scal \ ;
\end{equation*}
moreover, in every integral, we will expand the ``holes'' between the supports:
\begin{equation*}
[m-1,t]^\scal = \sum_{{\tt j}}[m-1,t]^\scal_{\tt j}\ ,
\qquad
[t+m'-1,N]^\scal = \sum_{{\tt k}}[t,t+m'-1]^\scal_{\tt k}\ .
\end{equation*}
We rewrite the two addenda of the numerator of~\eqref{e.corr.s1} as
\begin{equation}
\label{e.canc.int}
 \begin{aligned}
  \dinter{\phi\psi}^\scal Z^\scal_N =
  \sum_{{\tt j}} \sum_{{\tt k}} &\sum_{{\tt j}'} \sum_{{\tt k}'} 
\cr
  \int_{\reali^N} [0, \phi, m\!-\!1]^\scal &[m\!-\!1,t]^\scal_{\tt j} 
     [t, \psi, t+m'\!-\!1]^\scal [t+m'\!-\!1,N]^\scal_{\tt k}
     dV^{(N)\scal}_0 \times
\cr
  \int_{\reali^N} [0,m\!-\!1]^\scal &[m\!-\!1,t]^\scal_{{\tt j}'} 
     [t,t+m'\!-\!1]^\scal [t+m'\!-\!1,N]^\scal_{{\tt k}'}
     dV^{(N)\scal}_0\ ,
\cr
  \dinter{\phi}^\scal \dinter{\psi}^\scal = 
  \sum_{{\tt j}} \sum_{{\tt k}} &\sum_{{\tt j}'} \sum_{{\tt k}'} 
\cr
  \int_{\reali^N} [0, \phi, m\!-\!1]^\scal &[m\!-\!1,t]^\scal_{\tt j} 
     [t,t+m'\!-\!1]^\scal [t+m'\!-\!1,N]^\scal_{\tt k}
     dV^{(N)\scal}_0 \times
\cr
  \int_{\reali^N} [0,m\!-\!1]^\scal &[m\!-\!1,t]^\scal_{{\tt j}'} 
     [t, \psi, t+m'\!-\!1]^\scal [t+m'\!-\!1,N]^\scal_{{\tt k}'}
     dV^{(N)\scal}_0
 \end{aligned}
\end{equation}

Given such a decomposition of the correlation
$\inter{\phi\psi}^\scal-\inter{\phi}^\scal\inter{\psi}^\scal$, and using
the bitwise ``and'' operator $\wedge$ (see~\eqref{e.andb} in
Appendix~\ref{app.canc} for a formal definition), we give the main claim
of the proof:
\begin{enumerate}
\item all the terms such that ${\tt j} \wedge {\tt j}'\neq0$ AND ${\tt
      k} \wedge {\tt k}'\neq0$ cancel;
\item all the terms such that ${\tt j} \wedge {\tt j}'=0$ OR ${\tt
      k} \wedge {\tt k}'=0$ are (at least) of order $a^d$.
\end{enumerate}
The idea behind the cancellations is that ${\tt j} \wedge {\tt j}'\neq0$
ensure the presence (see Remark~\ref{r.canc}) of at least a ``1'' in the same
position in both {\tt j} and {\tt j}': this correspond to the absence of the
coupling term so that both the integrals in each of the expressions
in~\eqref{e.canc.int} can be splitted in the same position; the same happens
with {\tt k} and {\tt k}'. This opportunity to cut the integrals in the holes
between the supports of $\phi$ and $\psi$ allows us to rearrange the terms in
order to show that actually all these terms cancel. The formal proof is
deferred to Appendix~\ref{app.canc}.

Concerning the second part instead, the idea is that ${\tt j} \wedge {\tt
j}'=0$ ensure the presence of enough ``zeros'', each contributing with an
order in $a$.  In order to make the argument more precise, let us first assume
that the strings {\tt j} and {\tt j}' are not longer than {\tt k} and {\tt
k}'; thus, according to the statement of the proposition, the length $l$ of
both {\tt j} and {\tt j}' is equal to $d+2$.

Let us consider first the case in which ${\tt j} \wedge {\tt j}'=0$.  Since we
must consider all the cases for ${\tt k}$ and ${\tt k}'$ we actually don't
expand the second hole between the supports of $\phi$ and $\psi$. Moreover,
 instead of expanding the whole term $[m-1,t]$, we will write
\begin{equation*}
 [m-1,t]^\scal = [m-1,m]^\scal\;
                 \sum_{\tt j}[m,t-1]^\scal_{\tt j}\;[t-1,t]^\scal \ ,
\end{equation*}
and similarly for the same hole in the other integral. Please note that,
despite the use of the same letter, now the string {\tt j} has length exactly
$d$: expanding only on the ``interior'' of the hole simply means that we will
include in our estimates some terms that actually could be avoided because
they cancel.

The strategy is thus to apply Lemma~\ref{l.cutting} to
$\inter{\phi\psi}^\scal_N$ and $\inter{1}^\scal_N$, cutting the chain in
four parts for the first average and into two parts for the second one:
\begin{equation}
\label{e.corr.est}
\begin{aligned}
 \inter{\phi\psi}^\scal_N \inter{1}^\scal_N \leq  \,\,&
\cr
 \leq \; K_1^{m+m'+d}
 &\inter{\phi e^{a \left(x_0^2+x_{m-1}^2\right)}}^\scal_m     \,
 \frac1{Z^\scal_d} \sum_{{\tt j}} \int_{\RR^d}
 e^{\frac{a}2 \left(x_m^2+x_{t-1}^2\right)} [m,t-1]^\scal_{{\tt j}}
 dV_m^{(d)\scal} \,
\cr
 &\inter{\psi e^{a \left(x_t^2+x_{t+m'-1}^2\right)}}^\scal_{m'}
 \inter{e^{a \left(x_{t+m'}^2+x_{N-1}^2\right)}}^\scal_{N-m-m'-d}
\cr
 K_1^{d}
 &\inter{e^{a \left(x_t^2+x_{m-1}^2\right)}}^\scal_{N-d}     \,
 \frac1{Z^\scal_d} \sum_{{\tt j}'} \int_{\RR^d}
 e^{\frac{a}2 \left(x_m^2+x_{t-1}^2\right)} [m,t-1]^\scal_{{\tt j}'}
 dV_m^{(d)\scal}
 \, .
\end{aligned}
\end{equation}

We first deal with the sum given by the expansion of the smaller holes;
for the other terms we will apply some Lemmas proven in the
Appendix. Introducing the notation $|{\tt j}|$ to count the number of
``zeros'' in the string {\tt j}, according to Remark~\ref{r.canc} we
have $d\leq |{\tt j}| + |{\tt j}'|$; clearly it also holds $|{\tt j}| +
|{\tt j}'|\leq 2d$.

We need to estimate the generic term $[m,t-1]^\scal_{{\tt j}}$; we will
use the inequality $e^\alpha-1\leq\alpha e^\alpha$ on each of the
``zero'' factors, and then apply the estimate which decouple the
measure:
\begin{equation*}
 [m,t-1]^\scal_{{\tt j}} \leq a^{|{\tt j}|}
  \prod_{{\tt j}_n=0}x_nx_{n+1}  e^{a x_nx_{n+1}}
 \leq
  a^{|{\tt j}|} \prod_{{\tt j}_n=0}x_nx_{n+1}
  e^{a \frac12\left(x_n^2+x_{n+1}^2\right)} \ .
\end{equation*}
With the use of the previous estimate, the measure within $(m,t-1)$ can
be factorized so that we have, for $b\in\{0,1,2\}$, the following
integrals:
\begin{equation*}
 \int_\RR |x|^b e^{-\left(\alpha_b\frac{x^2}2+\frac{x^4}{4\beta^2}\right)}
  \leq
  \left(\frac2{\alpha_b}\right)^{\frac{b+1}2}
  \!\Gamma\left(\frac{b+1}2\right)
  \qquad\quad  \alpha_b=1+(2-b)a \ ,
\end{equation*}
with the exception of the boundary sites $m$ and $t-1$ where
$\alpha_b=1+(1-b)a$.  Since it could be not completely trivial to
control, for all the possible strings {\tt j}, how many factors have
$b=0$ or $b=1$ or $b=2$, the idea is to evaluate at the same time the
terms coming from $[m,t-1]^\scal_{{\tt j}}$ with those coming from
$[m,t-1]^\scal_{{\tt j}'}$.  Let us consider a site $n$ with $m< n<
t-1$, and suppose that from $[m,t-1]^\scal_{{\tt j}}$ we have a term
with $b=0$: this is possible if and only if ${\tt j}_n={\tt 1}$ and
${\tt j}_{n-1}={\tt 1}$. But then, from ${\tt j}\wedge{\tt j}'=0$, one
has ${\tt j}'_n={\tt 0}$ and ${\tt j}'_{n-1}={\tt 0}$ which implies that
the corresponding term from $[m,t-1]^\scal_{{\tt j}'}$ will be an
integral with $b=2$. Thus, the product of the two terms will be bounded
by $2\pi$. If we start $b=1$, this is compatible with both {\tt 01} and
{\tt 10} as substrings of {\tt j}, which imply respectively {\tt 10} or
{\tt 00}, and {\tt 01} or {\tt 00} for {\tt j}': thus, from
$[m,t-1]^\scal_{{\tt j}'}$, the contribution will be with $b=1$ or
$b=2$, and the product can be bounded in the same way. As a last case,
starting from $b=2$, this requires {\tt 00} for {\tt j} which is
compatible with all the cases in {\tt j}': again, the product of the two
terms will bounded by $2\pi$. For the two boundary sites one has the
same kind of control\footnote{the fact that we need to control also the
position $n-1$ is not a problem because, even if we expand only over
$(m,t-1)$, the condition ${\tt j}\wedge{\tt j}'=0$ actually holds for
the strings over $(m-1,t)$.} simply with a factor $(1-a)^{\frac32}$ on
each site for the worst case.

We thus have the following estimate
\begin{align*}
 \int_{\RR^d}
  e^{\frac{a}2 \left(x_m^2+x_{t-1}^2\right)}
  [m,t-1]^\scal_{{\tt j}} dV_{m}^{(d)\scal}
 &\int_{\RR^d}
  e^{\frac{a}2 \left(x_m^2+x_{t-1}^2\right)}
  [m,t-1]^\scal_{{\tt j}'} dV_{m}^{(d)\scal} \leq
\cr
 \leq
  &\frac{a^{|{\tt j}|+|{\tt j}'|}}{(1-a)^6}  \left(2\pi\right)^d \ .
\end{align*}
We need now to control the sum over all the possible strings such that ${\tt
j}\wedge{\tt j}'=0$. We exploit the fact that $d\leq|{\tt j}|+|{\tt j}'|\leq
2d$ and we count the number of configurations for the couple of strings {\tt
j} and {\tt j}' with a given value of $|{\tt j}|+|{\tt j}'|$. It is easy to
verify that
\begin{equation*}
 \# \Bigl\{ \left({\tt j},{\tt j}'\right)\ \colon\ 
     |{\tt j}|+|{\tt j}'|=d+i\Bigr\} = \binom{d}{i}2^{d-i} \ ,
\end{equation*}
so we end up with
\begin{equation}
\begin{aligned}
\label{e.stimajj}
 \sum_{{\tt j},{\tt j}'}
  \int_{\RR^d}
   e^{\frac{a}2 \left(x_m^2+x_{t-1}^2\right)}
   [m,t-1]^\scal_{{\tt j}} &dV_{m}^{(d)\scal}
 \int_{\RR^d}
   e^{\frac{a}2 \left(x_m^2+x_{t-1}^2\right)}
   [m,t-1]^\scal_{{\tt j}'} dV_{m}^{(d)\scal}
 \leq
\cr
 \leq &\frac{(2\pi a)^d}{(1-a)^6} \sum_{i=0}^d \binom{d}{i} 2^{d-i} a^i =
 \frac{\left[2\pi a (2+a)\right]^d}{(1-a)^6} \ .
\end{aligned}
\end{equation}

We are now ready to go on with the estimate~\eqref{e.corr.est}; using
Lemma~\ref{c.aver.2} for the averages involving $\phi$ and $\psi$,
Lemma~\ref{c.aver.4} for the averages without polynomials,
Lemma~\ref{l.m.1bis} and \ref{l.stimaZ} for the remaining partition functions
and \eqref{e.stimajj} for the remaining terms, one has
\begin{equation*}
\begin{aligned}
 \inter{\phi\psi}^\scal_N \leq
 K_1^{m+m'+2d}                                                \,
 &\left(\frac{A_1}{B}\right)^m \frac{2^r r!}{A_2^{2+2r}}\norm{\phi}     \,\,
 \left(\frac{A_1}{B}\right)^{m'} \frac{2^s s!}{A_2^{2+2s}}\norm{\psi} \,\times
\cr
 &(2K_1)^{2+8a_0}\left(\frac{D}{A_2}\right)^4
 \left(\frac{A_1}{\sqrt{2\pi}B}\right)^{2d} 
 \frac{\left[2\pi a (2+a)\right]^d}{(1-a)^6}  \ .
\end{aligned}
\end{equation*}
The same strategy can be applied to
$\inter{\phi}^\scal\inter{\psi}^\scal$, cutting, for each average, the
term containing the polynomial, the term containing the hole to be
expanded, and the rest of the chain: it is easy to realize that the same
factors present in~\eqref{e.corr.est} arise, with the same estimates, so
that we will simply add a factor 2.

We are now left only with the case ${\tt j}\wedge{\tt j}'\neq 0$, which
implies ${\tt k}\wedge{\tt k}'=0$. By our assumption, the strings {\tt k} are
longer, so we these remaining terms are even smaller as powers of $a$, but we
may simply repeat the same procedure working on {\tt k}, but expanding a
substring of length $d$: we get the same results so we will close our proof
adding another factor 2 to the estimate.
\end{proof}

%
%

\subsection{Upper bound of $\inter{R^2}$}
\label{sss:av.R2}

In this part we prove that $\inter{R^2}$ is of order $\mathcal{O}
\left(N / \beta^{\lambda_1} \right)$, with ${\lambda_1}\in[4,2r+4)$, if
we impose a suitable lower bound for $\beta$, or of order $\mathcal{O}
\left(N e^{-\sqrt[3]\beta} / \beta^3\right)$ if we have both a lower and
an upper bound for $\beta$. One remarkable point is the proportionality
to $N$ instead of $N^2$, and the other relevant aspect is the dependence
on the specific energy via the parameter $\beta$. Both these points are
a joint consequence of the control of the decay of correlations, as
given by Proposition \ref{p.aver.2}, and the decay of the interaction
range preserved throughout the whole perturbative
construction. Concerning the dependence on $\beta$, the fact that the
exponential estimate does not hold for vanishing specific energies is
due to condition~\eqref{e.muperr} which gives an upper bound to the
(optimal) perturbative order we can reach at fixed coupling $\mu$.

In~\eqref{e.R.1} we defined the remainder as
\begin{equation*}
R = \Poi{\Phi^{(r)}}{H} = \Poi{\Phi_r}{H_1},
\end{equation*}
which represents the rate of time variation of the almost conserved
integral $\Phi^{(r)}$. By recalling that in our perturbative
construction we maintain the cyclical symmetry, we have that
$\Phi_r=\varphi_r^\oplus$ and $H_1=h_1^\oplus$, thus we can write $R =
\rho^\oplus$ where $\rho := \Poi{\varphi_r}{h_1^\oplus}$.  In the
proof we need to exploit the decay properties of $\rho$ (see again
Proposition~\ref{prop.gen}, point (vi)); with the choice {
\begin{equation}
\label{e.sigma*}
\sigma_* := \sigma_0/4<\sigma_1\ ,
\end{equation}}
the decomposition~\eqref{e.decomp} is written in this case as:
\begin{equation}
\label{e.decomp.rho}
 \rho = \sum_{l=1}^N \rho^{(l)} \ , \qquad \norm{\rho^{(l)}} \leq
 C_\rho \mu_\flat^l \ , \qquad \mu_\flat:=e^{-\sigma_*}
 \ .
\end{equation}

Introduce the following constant\footnote{More precisely $\beta^*$ is
asymptotically constant with $a\to0$: we have $2^6e^{-1}C_1/\Omega \leq
\beta^* < 2^9e^2C_1/\Omega$ for $0\leq a<1/4$.} quantities
\begin{equation}
\label{e.betalambda}
 K_4:=K_3\frac{3^3e^6\Omega^6}{2^{9}C_1^2} \ ,
  \quad
 \beta^*:=\frac{2^6 e^{\tilde C-1}}{(1-\mu_\flat^2)(1-\mu_\flat)A_2^2}
  \;\frac{C_1}{\Omega}\ ,
  \quad
 {\lambda_1}:=2r(1-\nu) + 4-\nu \ .
\end{equation}


\begin{proposition}
\label{p.rho2.up}
The following different estimates hold:
\begin{enumerate}
\item for any $\nu\in(0,1]$, $a<\min\{a_0,1/4\}$ and $\beta >
      \max\{\beta_0,(\beta^*r^3)^{1/\nu}, \sqrt{3/2}\}$ such that
      $D^2\mu_\flat<1$, for any integer $r<\mu_*/(2\mu)$ one has
\begin{equation*}
 \inter{R^2} \leq \frac{N}{\beta^{\lambda_1}}
 \quadr{\frac{K_4{\beta^*}^3}{(1-D^2\mu_\flat)^2}} \ ;
\end{equation*}
\item for any $a<\min\{a_0,1/4\}$ and $\beta \geq \max\{\beta_0, 64e\beta^*,
     \sqrt{3/2}\}$ such that $D^2\mu_\flat<1$, there exists $\kappa$ such that
     taking $r=\lfloor \kappa/3\rfloor$, then
\begin{equation*}
 \inter{R^2} \leq N \frac{e^{-\kappa}}{\kappa^9}  
  \quadr{\frac{3^8 K_4 }{2e^4(1-D^2\mu_\flat)^2}}  \ ,
\qquad \left\{
\begin{aligned}
\kappa&:=\frac32\sqrt[3]{ \frac{\beta}{e\beta^*} } \ ,
\   &\beta < e\beta^*\;\tond{\frac{\mu_*}{\mu}}^3 \, ,
\cr 
\kappa&:=\frac32 \frac{\mu_*}{\mu} \ ,
\   &\beta \geq e\beta^*\;\tond{\frac{\mu_*}{\mu}}^3 \, .
\end{aligned} \right.
\end{equation*}
\end{enumerate}
\end{proposition}

The proof will be carried out in several steps, actually working first
on the dependence on $N$ and then on the scaling in $\beta$; in
particular, for the first point, it is useful to rewrite $\inter{R^2}$ as
the sum of two term to be dealt with separately. To this purpose we
exploit the cyclic symmetry of $R$, i.e. using the
decomposition~\eqref{e.cycl-fun}:
\begin{equation*}
R = \rho^\oplus \qquad\Longrightarrow\qquad
R = \sum_{j=0}^{N-1} \rho_j,\quad \rho_j:= \rho\circ\tau^{j}\ .
\end{equation*}

It is thus possible to write $R^2 = \sum_{j=0}^{N-1} \rho_j^2 + 2\sum_{0\leq
i<j\leq N-1}\rho_i\rho_j$, which implies, by translational
invariance\footnote{All the terms $\inter{\rho_j^2}$ are equivalent because of
the same symmetry of the measure and of $\rho$ itself.}, the following
expression
\begin{equation}
\label{e.R2.aver}
\inter{R^2} = N\inter{\rho^2} + 2\sum_{0\leq i<j\leq
  N-1}\inter{\rho_i\rho_j}\ .
\end{equation}
While the first addendum is clearly proportional to $N$, the second
appears to be proportional to $N^2$: we will now show that this is
actually not the case.  As a first step we recall that
$\inter{\rho\rho_j}=\inter{\rho\rho_{N-j}}$ thus we may write
\begin{equation}
\label{e.riordino}
 \sum_{0\leq i<j\leq N-1}\inter{\rho_i\rho_j}
= \sum_{j=1}^{N-1} (N-j)\inter{\rho\rho_j}
= N \sum_{j=1}^{[N/2]}\inter{\rho\rho_j} \ ,
\end{equation}
where the last equality, in the case $N$ odd, comes from
Lemma~\ref{l.sum.corr.1}. We remark that, although the expression for the even
case (see Lemma~\ref{l.sum.corr.1}) is slightly different, in the subsequent
estimates it will be bounded from above by the odd one.

The strategy is first to give an estimate of the generic term
$\inter{\rho\rho_j}$ for all $j\leq [N/2]$, and then perform the sum.
Defining the following constant
\begin{equation}
\label{e.K3}
K_3:=2^8K_1^4K_2 >4D^4K_2 \ ,
\end{equation}
the inequality holding since $a<1/4$ and $\beta>\sqrt{3/2}$. We have
\begin{lemma}
\label{l.corr.1}
For any $a<\min\{a_0,1/4\}$ and $\beta>\max\{\beta_0,\sqrt{3/2}\}$ such that
$D^2\mu_\flat<1$, one has
\begin{equation}
\label{e.est.rhrhj}
\inter{\rho\rho_j} < \frac34 K_3 C_\rho^2 g_2(r)
\quadr{\frac{\tond{D^2\mu_\flat}^j+\tond{D^2\mu_\flat}^{N-j}}{1-D^2\mu_\flat}}\qquad\qquad
j\leq N/2\ ,
\end{equation}
with
\begin{equation}
\label{e.g1}
g(r):=\frac{4^{2r+4}(r+2)!^2}{(A_2^2\beta)^{2r+4}}\ .
\end{equation}
\end{lemma}

\begin{proof}
Using~\eqref{e.decomp.rho} we may write $\inter{\rho\rho_j} =
\sum_{i=0}^{2N}\sum_{l+l'=i}\inter{\rho^{(l)}\rho_j^{(l')}}$. The idea
is to split such a sum into two parts: in the first one, choosing $l$
and $l'$ small enough, we will include (part of the) terms for which
$\rho^{(l)}$ and $\rho_j^{(l')}$ have disjoint supports and we will
exploit the decay of correlations as given by
Proposition~\ref{p.aver.2}; in the second term, with $l$ and $l'$
bounded from below, we will instead exploit the
decay~\eqref{e.decomp.rho} of $\rho$ itself. We thus define
\begin{displaymath}
A:= \sum_{l+l'< j}\inter{\rho^{(l)}\rho_j^{(l')}}\ ,\qquad B:=
\sum_{l+l'\geq j}\inter{\rho^{(l)}\rho_j^{(l')}}\ .
\end{displaymath}

\paragraph{A term:} 
We clearly have $l<j$; moreover, since we restrict to $j\leq N/2$ due to
Lemma \ref{l.sum.corr.1}, we also have $j+l'\leq N-1$. Thus $\rho^{(l)}$
and $\rho_j^{(l')}$ have disjoint supports.  By using Proposition
\ref{p.aver.2}, with $d=\min\{j-l,N-(j+l')\}$, we obtain
\begin{equation*}
 \left|\inter{\rho^{(l)}\rho_j^{(l')}}\right|
 \leq K_2 \mu_\sharp^{d} g(r) D^{2d+l+l'+4}
 \norm{\rho^{(l)}}\norm{\rho^{(l')}}\ ;
\end{equation*}
we remark that, in the dependence on $r$ (recall, from~\eqref{e.R.1}, that $R$
has degree $2r+4$) we could expect a factor $2^{2r+4}$ in $g(r)$, but we
overestimate with $4^{2r+4}$, because the B term needs exactly such a
dependence.

Since\footnote{although this does not ``sounds'' good $\ldots$}
$\mu_\sharp<\mu_\flat$ for (at least) $a<1/4$, using
also~\eqref{e.decomp.rho} one has
\begin{equation*}
 \left|\inter{\rho^{(l)}\rho_j^{(l')}}\right|
 \leq K_2 C_\rho^2 g(r) {\mu_\flat^{d+l+l'}} D^{2d+l+l'+4}\ .
\end{equation*}
It is useful to distinguish the cases $d=j-l$ and $d=N-(j+l')$, that we
call A1 and A2 respectively.
In the sub-case A1, one has $l+l'+d=j+l'$ and $l+l'+2d=j+l'+d$, thus we get
\begin{displaymath}
 \sum_{l+l'<j,A1} \left|\inter{\rho^{(l)}\rho_j^{(l')}}\right|
 \leq K_2D^4 C_\rho^2 g(r)\tond{D\mu_\flat}^j
 \sum_{l+l'< j} D^{l'+d}\mu_\flat^{l'}\ .
\end{displaymath}
We deal with the remaining sum in this way
\begin{align*}
 \sum_{l+l'< j} D^{l'+d}\mu_\flat^{l'} &=
 \sum_{i=0}^{j-1}\sum_{l'=0}^i\tond{D\mu_\flat}^{l'}D^{j+l'-i}=
 D^j\sum_{i=0}^{j-1} D^{-i}\sum_{l'=0}^i\tond{D^2\mu_\flat}^{l'} =
\cr
 &< \frac{D^j}{1-D^2\mu_\flat}\sum_{i=0}^{j-1}D^{-i}<
 \frac{2 D^j}{1-D^2\mu_\flat}\ ,
\end{align*}
where we exploit also that $C>2$ (see \eqref{e.C.et.alter}). Then
\begin{equation}
\label{e.caseA1}
 \sum_{l+l'<j,A1} \left|\inter{\rho^{(l)}\rho_j^{(l')}}\right| <
 {2K_2D^4 C_\rho^2 g(r)}\frac{\tond{D^2\mu_\flat}^j}{1-D^2\mu_\flat}\ .
\end{equation}

In the other sub-case A2, just replacing $l$ with $l'$ and $j$ with
$N-j$, we get
\begin{displaymath}
 \left|\inter{\rho^{(l)}\rho_j^{(l')}}\right|
 \leq K_2D^4 C_\rho^2 g(r) {\mu_\flat^{N-j+l}} D^{N-j+l+d}\ .
\end{displaymath}
With the same approach used above one obtains
\begin{displaymath}
 \sum_{l+l'\leq j-1} D^{l+d}\mu_\flat^{l} = D^{N-j}\sum_{i=0}^{j-1}
 \sum_{l=0}^iD^{-i}\tond{D^2\mu_\flat}^{l} < \frac{2
 D^{N-j}}{1-D^2\mu_\flat} \ ,
\end{displaymath}
hence the contribution coming from A2 is
\begin{equation}
\label{e.caseA2}
 \sum_{l+l'<j,A2} \left|\inter{\rho^{(l)}\rho_j^{(l')}}\right|
 < {2K_2D^4 C_\rho^2
 g(r)}\frac{\tond{D^2\mu_\flat}^{N-j}}{1-D^2\mu_\flat}\ ,
\end{equation}
and combining A1 with A2 we get
\begin{equation}
\label{e.est.A}
 \sum_{l+l'<j} \left|\inter{\rho^{(l)}\rho_j^{(l')}}\right|
 <\frac{K_3}2 C_\rho^2 g(r) \quadr{\frac{\tond{D^2 \mu_\flat}^j+
 \tond{D^2 \mu_\flat}^{N-j}}{1-D^2 \mu_\flat}}\ .
\end{equation}

\paragraph{B term:} since we aim at exploiting the decay of $\rho$, by
Schwartz inequality we rewrite the correlation as
\begin{displaymath}
 \sum_{l+l'\geq j} \left|\inter{\rho^{(l)}\rho_j^{(l')}}\right|
 \leq \sum_{l+l'\geq j}\sqrt{\inter{(\rho^{(l)})^2}\inter{(\rho^{(l')})^2}}
 \ ;
\end{displaymath}
using Proposition \ref{p.aver.1} it holds
\begin{displaymath}
\inter{(\rho^{(l)})^2}<
 \frac{(2K_1)^{1+4a_0}}{A_2^{4r+12}}D^{2+l}\tond{\frac{2}{\beta}}^{2r+4}
 (2r+4)!\norm{\rho^{(l)}}^2\ ;
\end{displaymath}
the analogous estimate holds for $\inter{(\rho^{(l')})^2}$. Then by
observing that
\begin{displaymath}
\frac{(2K_1)^{1+4a_0}}{A_2^{4}}<K_2\ ,\qquad\qquad (2r+4)!\leq
2^{2r+4}(r+2)!^2\ ,
\end{displaymath}
one obtains
\begin{displaymath}
\sqrt{\inter{(\rho^{(l)})^2}\inter{(\rho^{(l')})^2}} < K_2 D^2
D^{(l+l')/2} \tond{\frac{4}{A_2^2 \beta}}^{2r+4}
(r+2)!^2\norm{\rho^{(l)}}\norm{\rho^{(l')}}\ .
\end{displaymath}
If we notice that $K_2 D^2 < K_3/4$, and using~\eqref{e.g1}, then
\begin{equation}
\label{e.riciclami}
\sqrt{\inter{(\rho^{(l)})^2}\inter{(\rho^{(l')})^2}} < \frac{K_3}4
  C_\rho^2 g(r) D^{(l+l')/2} \mu_\flat^{l+l'}\ ;
\end{equation}
the sum over $l$ and $l'$ gives
\begin{equation}
\label{e.est.B}
\sum_{l+l'\geq j} \left|\inter{\rho^{(l)}\rho_j^{(l')}}\right| <  
\frac{K_3}4 C_\rho^2 g(r) \frac{\tond{D \mu_\flat}^j}{1 - D \mu_\flat}\ .
\end{equation}
To get \eqref{e.est.rhrhj}, simply add \eqref{e.est.B} to \eqref{e.est.A}.
\end{proof}

We are now ready to state and prove an estimate showing how
$\inter{R^2}$ is proportional to $N$.

\begin{lemma}
\label{l.rho.up}
For $a<\min\{a_0,1/4\}$ and $\beta>\max\{\beta_0,\sqrt{3/2}\}$ such that
$D^2\mu_\flat<1$, one has
\begin{equation}
\label{e.rho.up}
\inter{R^2} \leq N \frac{K_3 C_\rho^2 g(r)}{(1-D^2\mu_\flat)^2}\ .
\end{equation}
\end{lemma}

\begin{proof}

As we showed before (see \eqref{e.R2.aver} and the subsequent
paragraphs), we have
\begin{equation}
\label{e.resto.quad}
\inter{R^2} \leq N\inter{\rho^2} +
N\sum_{j=1}^{[N/2]}|\inter{\rho\rho_j}|\ .
\end{equation}

Concerning the first term we may proceed like in the proof
Lemma~\ref{l.corr.1}, in the case of B term, the only difference being
the index of the sum; we thus follow such a proof up to
formula~\eqref{e.riciclami}, and then
\begin{equation}
\label{e.stima.rho.quad}
\begin{aligned}
\inter{\rho^2}
 < \sum_{i=0}^{N-1}\sum_{l+l'=i}\left|\inter{\rho^{(l)}\rho^{(l')}}\right|
&< \sum_{i=0}^{N-1}\sum_{l+l'=i} \frac14 K_3 C_\rho^2 g(r)
   \tond{D^{\frac12}\mu_\flat}^{l+l'}
\cr
&< \frac14 K_3 \frac{C_\rho^2g(r)}{(1-D^2\mu_\flat)}\ .
\end{aligned}
\end{equation}

In order to control the second term of \eqref{e.resto.quad}, from
Lemma~\ref{l.corr.1} we have to estimate the following sum
\begin{equation}
\label{e.stimella}
\sum_{j=1}^{[N/2]}\quadr{\tond{D^2\mu_\flat}^j+\tond{D^2\mu_\flat}^{N-j}} =
\frac{\tond{D^2\mu_\flat} - \tond{D^2\mu_\flat}^{N}}{1-D^2\mu_\flat} <
\frac{\tond{D^2\mu_\flat}}{1-D^2\mu_\flat}\ ,
\end{equation}
which leads us to
\begin{displaymath}
\sum_{j=1}^{[N/2]} |\inter{\rho\rho_j}| \leq \frac34 K_3
C_\rho^2g(r)\quadr{\frac{\tond{D^2\mu_\flat}}{(1-D^2\mu_\flat)^2}}\ .
\end{displaymath}
Adding \eqref{e.stima.rho.quad}, and then multiplying by $N$, one has
\eqref{e.rho.up}.
\end{proof}

The final step in the proof of Proposition~\ref{p.rho2.up} consists in
showing the correct dependence on $\beta$: to this purpose we estimate
the numerator of~\eqref{e.rho.up}.
\begin{lemma}
\label{l.rho2.up} 
The following two estimates hold:
\begin{enumerate}
\item for any $\nu\in(0,1]$, $a<\min\{a_0,1/4\}$ and
      $\beta>\max\{\beta_0,(\beta^*r^3)^{1/\nu},\sqrt{3/2}\}$ one has
\begin{equation}
\label{e.upper.R2}
 K_3 C_\rho^2 g(r) \leq
 K_4 \frac{{\beta^*}^3}{\beta^{\lambda_1}} \ .
\end{equation}
\item for any $a<\min\{a_0,1/4\}$ and $\beta \geq \max\{\beta_0,
      64e\beta^*,\sqrt{3/2}\}$, there exists
      $\kappa$ such that taking $r=\lfloor \kappa/3\rfloor$, then
\begin{equation}
\label{e.upper.R2.exp}
 K_3 C_\rho^2 g(r) \leq 
 \frac{3^8 K_4 }{2e^4}\frac{e^{-\kappa}}{\kappa^9}   \ ,
\qquad \left\{
\begin{aligned}
\kappa&:=\frac32\sqrt[3]{ \frac{\beta}{e\beta^*} } \ ,
\qquad  &\beta < e\beta^*\;\tond{\frac{\mu_*}{\mu}}^3 \ ,
\cr 
\kappa&:=\frac32 \frac{\mu_*}{\mu} \ ,
\qquad  &\beta \geq e\beta^*\;\tond{\frac{\mu_*}{\mu}}^3\ .
\end{aligned} \right.
\end{equation}
\end{enumerate}
\end{lemma}

\begin{proof}
As a first step, we have to work on $C_\rho$: from Proposition \ref{prop.gen}
(vi)
\begin{displaymath}
C_\rho \leq \frac{8(r+2)F_r^{r-1}C_1^2 e^{(r+2)\tilde
    C}}{(1-\mu_\flat^2)(1-\mu_\flat)}\ ,
\end{displaymath}
where the correction $e^{(r+2)\tilde C}$ comes from the transformation back to
the original variables (see Lemma \ref{l.inv.g}), and for $F_r=16r^2C_*$,
using~\eqref{const.prop.gen} and~\eqref{e.muperr}, it holds
\begin{displaymath}
 \frac{8 r^2 C_1}{\Omega(1-\mu_\flat^2)(1-\mu_\flat)} \leq
 F_r\leq \frac{16 r^2 C_1}{\Omega(1-\mu_\flat^2)(1-\mu_\flat)} \ .
\end{displaymath}

Using the above estimates and rearranging some terms we have
\begin{equation*}
 C_\rho^2 g(r) \leq 
  \frac{\Omega^6(1-\mu_\flat^2)^4(1-\mu_\flat)^4}{2^{12}C_1^2}
  \frac{(r+2)^2}{r^{12}} \tond{\frac{4F_re^{\tilde C}}{A_2^2\beta}}^{2r+4}
  \quadr{(r+2)!}^2 \ ;
\end{equation*}
this can be further simplified as
\begin{equation}
\label{e.biforc}
 C_\rho^2 g(r) \leq 
  \frac{3^3\Omega^6}{2^{9}C_1^2} a_r^2 \, 
\qquad
  a_r := \tond{e\beta^* \frac{r^2}{\beta}}^{r+2} \frac{r!}{r^3}
  \ .
\end{equation}
We now proceed in two different ways to get the two different
estimates~\eqref{e.upper.R2} and~\eqref{e.upper.R2.exp}. For the first
one, with a power law dependence on $\beta$, using the standard upper
bound $r!  \leq e \sqrt{r} \tond{\frac{r}{e}}^r$, we have
\begin{equation*}
 a_r \leq e^3 {\beta^*}^2 \frac{r^{3/2}}{\beta^2}
 \tond{\beta^*\frac{r^3}{\beta}}^r =
 e^3 {\beta^*}^{3/2} \tond{\beta^*\frac{r^3}{\beta^\nu}}^{r+1/2}
 \beta^{-[r(1-\nu) + 2 - \frac\nu2]} \ ,
\end{equation*}
so that
\begin{equation*}
 K_3 C_\rho^2 g(r) \leq 
  K_3 \frac{3^3e^6\Omega^6}{2^{9}C_1^2} \quadr{
  {\beta^*}^{3/2} \tond{\beta^*\frac{r^3}{\beta^\nu}}^{r+1/2}
 \beta^{-\lambda_1/2} }^2 
 \leq  K_4 \frac{{\beta^*}^3}{\beta^{\lambda_1}} \ ,
\end{equation*}
where we used $\beta>(\beta^*r^3)^{1/\nu}$ to estimate from above
$\tond{\beta^*\frac{r^3}{\beta^\nu}}^{r+1/2}$.

For the exponential estimate~\eqref{e.upper.R2.exp}, we work
on~\eqref{e.biforc} in a different way, trying to optimize the order $r$. As
usual, the sequence $a_r$ will be initially decreasing, and then definitively
increasing: we are interested in the larger $r\geq2$ such that $a_r<a_{r-1}$,
which translates in
\begin{equation*}
 \frac{e\beta^*}{\beta}r^3\tond{\frac{r}{r-1}}^{2r-1} < 1
 \quad\Leftarrow\quad
 r < r_{{\rm opt}} := 
   \left\lfloor \sqrt[3]{\frac\beta{8e\beta^*}} \right\rfloor \ .
\end{equation*}
The lower bound on $\beta$ in the hypothesis implies that $r_{{\rm opt}} \geq
2$; moreover, the corresponding upper bound is enough to conclude that
$2r_{{\rm opt}}\mu \leq \mu_*$ as required in condition~\eqref{e.muperr} of
Proposition~\ref{prop.gen}.  Using again $r!  \leq e \sqrt{r}
\tond{\frac{r}{e}}^r$, and using $r_{{\rm opt}}^3 \leq \frac\beta{8e\beta^*}$,
we have
\begin{equation*}
 a_r \leq \tond{\frac{r^2}{8r_{{\rm opt}}^3}}^{r+2} 
  \frac{e \sqrt{r} r^r e^{-r}}{r^3}
 = \frac{e r\sqrt{r}}{64r_{{\rm opt}}^6}
   \tond{\frac{r^3}{8er_{{\rm opt}}^3}}^r \ ,
\end{equation*}
so that
\begin{equation*}
K_3 C_\rho^2 g(r_{{\rm opt}}) \leq 
 \frac{K_4}{2^{12} e^4} \frac1{r_{{\rm opt}}^9}
 \tond{8 e}^{-2r_{{\rm opt}}} \ ,
\end{equation*}
and the thesis follow by using the lower bound $(2+2\ln8)r_{{\rm opt}} \geq
F(\beta)$.

If instead $\beta > e\beta^*\;(\mu_*/\mu)^3$, then $r_{{\rm opt}} > r_{{\rm
max}} :=\left\lfloor\frac{\mu_*}{2\mu}\right\rfloor$ allowed by
Proposition~\ref{prop.gen}; in this case it holds $r_{{\rm max}}^3 \leq
\frac\beta{8e\beta^*}$, so that the previous estimates on $a_r$ holds true
 with $r_{{\rm max}}$ instead of $r_{{\rm opt}}$
and the thesis follow by using the lower bound $(2+2\ln8) r_{{\rm max}} \geq
3/2 (\mu_*/\mu)$.
\end{proof}

%
%

\subsection{Lower bound of $\sigma^2[\Phi]$}
\label{sss:lower.sig}

In this Section we give the required lower bound for $\sigma^2[\Phi]$:
besides the fact that an estimate from below can be less trivial than one from
above, the main point here is that $\Phi$ is non homogeneous.

As for the estimate of the previous subsection, one relevant aspect is the
proportionality with $N$, and the other one is the scaling with $\beta$.

Given the constants $\mu_2$ and $\beta_4$ of Lemma~\ref{l.un-nono}, we
have the following
\begin{proposition}
\label{p.lower.sig}
There exists a positive constant $\beta_3$, such that if $\mu<\mu_2$,
$D^2\mu_\flat \leq 1/2$ and $\beta>\max \{\beta_0,\beta_3,\beta_4,
6e^{\tilde C} \beta^* r^2(r+1), \sqrt{3/2}\}$, then
\begin{equation*}
\sigma^2[\Phi] \geq \frac{N}{\beta^2}\frac{\Omega^2}{10} \ .
\end{equation*}
\end{proposition}

\begin{proof}
We exploit once again the cyclic symmetry of $\Phi$, i.e. $\Phi =
\varphi^\oplus = \sum_j \varphi\circ \tau^{j}$. Recalling the standard
notation $\sigma(f,g)=\inter{fg}-\inter{f}\inter{g}$ for the covariance, we
thus write\footnote{We avoid the notation $\varphi_{j}=\varphi\circ \tau^{j}$
because we will need the index to indicate the homogeneity degree.}
\begin{align*}
 \sigma^2[\Phi]
 &= \sum_j\sigma^2\quadr{\varphi\circ\tau^{j}} + 2\sum_{h<k}
    \sigma\tond{\varphi\circ\tau^{h},\varphi\circ\tau^{k}}
\cr
 &\geq N\sigma^2[\varphi] - 2\sum_{h<k}
       \left|\sigma\tond{\varphi\circ\tau^{h},\varphi\circ\tau^{k}}\right| \ ,
\end{align*}
where we used the translation invariance to extract the factor $N$
from the first sum.  The proof will be carried out in the two
forthcoming Lemmas, whose application immediately gives the
estimate~\eqref{p.lower.sig}: the first dealing with the variance
$\sigma^2[\varphi]$ and the second for the estimate of the covariance
terms. We exploit the different scalings in $\beta$ and the
possibility to take $\mu$ arbitrary small: this ensures, provided
$\beta$ is large enough (setting $\beta_2$ as the threshold), the
asymptotic bound $\sigma^2[\Phi]\gtrsim N\sigma^2[\varphi_0]$, being
$\varphi_0$ the quadratic part of the seed $\varphi$.
\end{proof}

\begin{lemma}
\label{l.un-nono}
There exist positive constants $\mu_2$ and $\beta_4$, such that if
$\mu<\mu_2$, $D\mu< 1/2$, and $\beta>\max \{\beta_0,\beta_4,
6e^{\tilde C} \beta^* r^2(r+1),\sqrt{3/2}\}$, then
\begin{equation*}
\sigma^2[\varphi] \geq \frac{\Omega^2}{9\beta^2} \ .
\end{equation*}
\end{lemma}

\begin{proof}
We exploit the expansion in homogeneous parts of $\ph=\sum_{s=0}^r\ph_s$: the
strategy is to show that, under the assumed hypothesis, the leading term of
the variance of $\ph$ is given by the variance of $\varphi_0$. We thus first
expand
\begin{equation}
\label{e.R.sigma}
 \sigma^2[\varphi] = \sum_{s,s'=0}^r \sigma(\varphi_s,\varphi_{s'}) 
 = \sigma^2[\varphi_0] + R_{\sigma} \ ,
 \qquad
 R_{\sigma}:= \sum_{i=1}^{2r}\sum_{s+s'=i}\sigma(\varphi_s,\varphi_{s'})\ ,
\end{equation}
so that we will get the thesis via $\sigma^2[\varphi] \geq
\sigma^2[\varphi_0] - \left|R_{\sigma}\right|$, by giving a lower
bound for $\sigma^2[\varphi_0]$ and an upper bound for
$\left|R_{\sigma}\right|$.

\paragraph{Step 1 (lower bound for $\sigma^2[\varphi_0]$).}
From Proposition~\ref{p.1} and Proposition~\ref{prop.gen} we have
$\varphi_0 = h_\Omega=\Omega(q_0^2+p_0^2)/2$, which has interaction
length equal to one in the variables $(q,p)$; once we transform it
back to $(x,y)$ variables, from a direct calculation using the decay
properties of $A^{1/4}$ (see formula~\eqref{e.row1.A14} of
Proposition~\ref{p.1}) we have that
$\varphi_0\in\Dscr(C_{\ph_0},\sigma_0)$ with $C_{\ph_0}\leq
\frac{\Omega}{1-2\mu}$, hence the decomposition $\varphi_0 =
\sum_{l\geq 0}\varphi_0^{(l)}$ holds, where
\begin{displaymath}
\varphi_0^{(0)}(x,y) = \frac\Omega2 \tond{x_0^2+y_0^2} \ ;
\end{displaymath}
thus
\begin{displaymath}
 \sigma^2[\varphi_0] = \sigma^2\quadr{\varphi_0^{(0)}} +
 2 \!\sum_{l+l'\geq 1}\sigma\tond{\varphi_0^{(l)},\varphi_0^{(l')}} \geq
 \sigma^2\quadr{\varphi_0^{(0)}} - 2 \!\sum_{l+l'\geq 1}
  \left|\sigma \tond{\varphi_0^{(l)},\varphi_0^{(l')}} \right| \, .
\end{displaymath}
We will now show that the first term is different from zero, while the
second is of order $\mathcal{O}(a)$ and asymptotically
vanishing. Indeed, from the factorization of the Gibbs measure with
respect to the $x$ and $y$ variables, one has
$\sigma^2\quadr{x_0^2+y_0^2} = \sigma^2\quadr{x_0^2} +
\sigma^2\quadr{y_0^2}$; it thus follows
\begin{displaymath}
\sigma^2\quadr{\ph_0^{(0)}}
= \frac{\Omega^2}4\tond{\sigma^2 \quadr{x_0^2}
+ \sigma^2 \quadr{y_0^2}} > \frac{\Omega^2}4 \sigma^2\quadr{y_0^2} 
= \frac{\Omega^2}{8\beta^2}\ .
\end{displaymath}
The covariance part is dealt with exploiting the exponential decay of the
homogeonous polynomials $\varphi_0^{(l)}$; in particular, each covariance is
estimated via Schwartz inequality as $|\sigma(\varphi,\psi)|\leq
|\inter{\varphi\psi}|+|\inter{\varphi}\inter{\psi}| \leq
\sqrt{\inter{\varphi^2}\inter{\psi^2}} + |\inter{\varphi}\inter{\psi}|$, and
then applying Lemmas \ref{l.phi.aver} and \ref{l.phi2.decay} we have
\begin{equation*}
\begin{aligned}
 \left|\sigma\tond{\ph_0^{(l)},\ph_0^{(l')}}\right| &\leq
 \sqrt{\inter{\tond{\ph_0^{(l)}}^2} \inter{\tond{\ph_0^{(l')}}^2}} +
 \left| \inter{\ph_0^{(l)}} \inter{\ph_0^{(l')}} \right| \leq
\cr
 &\leq \frac1{\beta^2}\Biggl\{ \frac{2(2K_1)^{1+4a_0}}{A_2^4} \;
       D^{2+\frac{l+l'}2} \; \; \sqrt{\norm{\tond{\ph_0^{(l)}}^2} \norm{\tond{\ph_0^{(l')}}^2}}
       +
\cr
 &\phantom{\leq\Biggl\{} + \frac{(2K_1)^{2+8a_0}}{A_2^8} \;
       D^{4+l+l'} \; \norm{\ph_0^{(l)}} \norm{\ph_0^{(l')}} \;\Biggr\} \tond{\frac2{A_2^2}}^{2} \leq
\cr
 &\leq \frac{K_3}{\beta^2} C_{\ph_0}^2 (D e^{-\sigma_0})^{l+l'} \ ;
\end{aligned}
\end{equation*}
the sum over $l$ and $l'$ gives
\begin{displaymath}
\sum_{l+l'\geq 1}(D e^{-\sigma_0})^{l+l'} \leq \sum_{i\geq 1} i(D
e^{-\sigma_0})^i < \frac{D e^{-\sigma_0}}{1-D e^{-\sigma_0}} =
\frac{2\mu D}{1-2\mu D}\ .
\end{displaymath}
If we collect the two estimates we have
\begin{equation}
\label{e.sigma2.phi0}
\sigma^2[\varphi_0] \geq \frac{\Omega^2}{\beta^2}\tond{\frac18 - 
\frac{2 K_3 D \mu}{(1-2\mu)^2(1-2\mu D)}}  \ .
\end{equation}

\paragraph{Step 2 (upper bound for $\left|R_{\sigma}\right|$).}
For the remainder, we will estimate each term in the sum: we first
control the covariance terms again via Schwartz inequality
\begin{displaymath}
|\sigma(\varphi_s,\varphi_{s'})|\leq
|\inter{\varphi_s\varphi_{s'}}|+|\inter{\varphi_s}\inter{\varphi_{s'}}|
\leq \sqrt{\inter{\varphi_s^2}\inter{\varphi_{s'}^2}} +
|\inter{\varphi_s}\inter{\varphi_{s'}}|\ .
\end{displaymath}

We apply Lemma~\ref{l.phi.aver} and Corollary~\ref{l.phi2.aver} to both
$\ph_s$ and $\ph_{s'}$; in particular, using
estimate~\eqref{e.inter.ph2} and using $\sqrt{(2s+2)!(2s'+2)!} <
2^{s+s'+2}(s+s'+2)!$, we have
\begin{equation*}
\sqrt{\inter{\varphi_s^2}\inter{\varphi_{s'}^2}} \leq
\quadr{\frac{4D^2(2K_1)^{1+4a_0}}{A_2^4}}
\frac{C_{\ph_s}C_{\ph_{s'}}}{1-\mu_\flat}
\quadr{\tond{\frac4{A_2^2\beta}}^{s+s'+2}(s+s'+2)!}
\ ;
\end{equation*}
where we used $\sigma_s>\sigma_*$, and the
definition~\eqref{e.decomp.rho} of $\mu_\flat$; in a similar way, using
estimate~\eqref{e.inter.ph}
\begin{equation*}
|\inter{\varphi_s}\inter{\varphi_{s'}}| \leq
\quadr{\frac{4D^4(2K_1)^{2+8a_0}}{A_2^8}}
C_{\ph_s}C_{\ph_{s'}}
\quadr{\tond{\frac2{A_2^2\beta}}^{s+s'+2}(s+s'+2)!}
\ .
\end{equation*}
Summing up the previous formulas, and using the
definitions~\eqref{e.C.et.alter} and~\eqref{e.K3} of $K_2$ and $K_3$,
\begin{equation}
\label{e.covs.3}
\sqrt{\inter{\varphi_s^2}\inter{\varphi_{s'}^2}} +
|\inter{\varphi_s}\inter{\varphi_{s'}}| \leq
K_3 \frac{C_{\ph_s}C_{\ph_{s'}}}{1-\mu_\flat}
\quadr{\tond{\frac4{A_2^2\beta}}^{s+s'+2}(s+s'+2)!}
\ .
\end{equation}

As we did in the proof of Lemma~\ref{l.rho2.up}, we now make explicit
the dependence of $C_{\ph_s}$ from the index $s$: from points (iv), (v)
and (vi) of Proposition~\ref{prop.gen}
\begin{displaymath}
 C_{\ph_s} = e^{(s+1)\tilde C} F_r^{s-1} C_1 \ ,
 \qquad
 \frac{24r^2C_1}{\Omega\tond{1-\mu_\flat^2}(1-\mu_\flat)} \leq
 F_r \leq \frac{48r^2C_1}{\Omega\tond{1-\mu_\flat^2}(1-\mu_\flat)} \ ,
\end{displaymath}
where the factor $e^{(s+1)\tilde C}$ comes from the transformation back
to the original variables (see
Lemma~\ref{l.inv.g}). Using~\eqref{e.covs.3} in~\eqref{e.R.sigma} we have
\begin{equation}
\label{e.covs.4}
\begin{aligned}
 \left|R_{\sigma}\right| &\leq 
  \frac{K_3C_1^2}{1-\mu_\flat} \sum_{i=1}^{2r}\sum_{s+s'=i} 
  e^{(s+s'+2)\tilde C}  F_r^{s+s'-2} 
  \quadr{\tond{\frac4{A_2^2\beta}}^{s+s'+2}(s+s'+2)!} \leq
\cr
 &\leq \frac{K_3C_1^2}{F_r^4(1-\mu_\flat)} \sum_{i=1}^{2r}
  \quadr{\tond{\frac{4F_re^{\tilde C}}{A_2^2\beta}}^{i+2}(i+2)!} \leq
\cr
 &\leq \frac{K_3}{2^{12}3^4}\tond{\frac{\Omega}{C_1}}^4
   \frac1{r^8} \sum_{i=1}^{2r}
  \quadr{\tond{\frac{2^63C_1e^{\tilde C}r^2}
              {\Omega\tond{1-\mu_\flat^2}(1-\mu_\flat)A_2^2\beta}}^{i+2}
         (i+2)!}
\ .
\end{aligned}
\end{equation}

The sum in the previous formula is of the form $\sum_{i=1}^{2r}a_i$ with $a_i=
X^{i+2}(i+2)!$; we have that $a_i$ is monotone decreasing, for
$i=1,\ldots,2r$, if $X(2r+2)<1$. But for hypothesis we have $\beta>6e^{\tilde
C} \beta^* r^2(r+1)$, so (recalling the definition~\eqref{e.betalambda} of
$\beta^*$) it holds $a_i\leq a_1$ and thus $\sum_{i=1}^{2r}a_i \leq 2r
a_1=12X^3 r$. Inserting such an estimate in~\eqref{e.covs.4} we have
\begin{equation}
\label{e.covs.4.1}
 \left|R_{\sigma}\right| \leq
  \frac{2^8K_3e^{3\tilde C}}{\tond{1-\mu_\flat^2}^3(1-\mu_\flat)^3A_2^6}
  \tond{\frac{\Omega}{C_1}}\frac1{r\beta^3}
\end{equation}
\paragraph{Step 3.}
With such an upper bound for $\left|R_{\sigma}\right|$ proportional to
$\beta^{-3}$ and with the lower bound~\eqref{e.sigma2.phi0}
proportional to $\beta^{-2}$, there exist $\beta_4$ and $\mu_2$ such
that if $\mu<\mu_2$ and $\beta>\beta_4$ the thesis follows.
\end{proof}

\begin{lemma}
Given $\beta> \max\{\beta_0,6e^{\tilde C} \beta^*
r^2(r+1),\sqrt{3/2}\}$ and $\mu$ such that $D^2\mu_\flat<1$ one has
\begin{equation*}
\Biggl|\sum_{h<k}\sigma\left(\varphi\circ\tau^{h},
   \varphi\circ\tau^{k}\right)\Biggr| \leq
N \quadr{\frac{\tond{D^2\mu_\flat}}{(1-D^2\mu_\flat)^2}}
\frac{K_3e^{2\tilde C}}{\tond{1-\mu_\flat^2}^2(1-\mu_\flat)^2A_2^4}
  \tond{\frac{\Omega}{C_1}}^2\frac1{r^3\beta^2} \ .
\end{equation*}
\end{lemma}

\begin{proof}
We first proceed like in formula~\eqref{e.riordino} exploiting the
translation invariance
\begin{equation*}
 \sum_{h<k} \sigma\tond{\varphi\circ\tau^{h},\varphi\circ\tau^{k}} =
  N\sum_{j=1}^{[N/2]} \sigma\tond{\varphi,\varphi\circ\tau^{j}} \ ,
\end{equation*}
where we used Lemma~\ref{l.sum.corr.1} for the part
$\inter{(\varphi\circ\tau^{h}) (\varphi\circ\tau^{k})}$ and an easy
direct calculation for the terms $\inter{\varphi\circ\tau^{h}}
\inter{\varphi\circ\tau^{k}}$. We will follow the strategy of the proof
of Lemma~\ref{l.corr.1}: at variance with that situation, here $\ph$ is
not homogeneous so, besides the expansion~\eqref{e.decomp} on the the
interaction lengths, we also need to expand over the different degrees
$\ph=\sum_{s=0}^r \ph_s$. We will use the notation $\ph_{s,j}:=
\ph_s\circ\tau^{j}$ to indicate at the same time the degree and the
translation.
\begin{equation*}
 \sigma\tond{\varphi,\varphi\circ\tau^{j}} =
  \sum_{s,s'=0}^r \sigma\tond{\ph_s,\ph_{s',j}} =
  \sum_{s,s'=0}^r \sum_{i=0}^{2N} \sum_{l+l'=i}
     \sigma\tond{\ph_s^{(l)},\ph_{s',j}^{(l')}} \ .
\end{equation*}
We split again the sum over $i$ into two parts:
\begin{equation*}
 A:= \sum_{l+l'<j} \sigma\tond{\ph_s^{(l)},\ph_{s',j}^{(l')}} \ ,
\qquad
 B:= \sum_{l+l'\geq j} \sigma\tond{\ph_s^{(l)},\ph_{s',j}^{(l')}} \ .
\end{equation*}
In the $A$ term, $\ph_s^{(l)}$ and $\ph_{s',j}^{(l')}$ have disjoint
supports, so we can apply Proposition~\ref{p.aver.2}
\begin{equation*}
\begin{aligned}
 |A| &\leq \!\!\sum_{l+l'<j}\!\! K_2 D^{l+l'+2d+4} \; \mu_\sharp^d \;
          \tond{\frac2{A_2^2\beta}}^{s+s'+2}\!\!\!(s+1)!(s'+1)! \;
	  \norm{\ph_s^{(l)}} \norm{\ph_{s',j}^{(l')}} \leq
\cr
 &\leq 2 K_2 D^4 C_{\ph_s} C_{\ph_{s'}}
 \tond{\frac2{A_2^2\beta}}^{s+s'+2}(s+s'+2)! 
 \sum_{l+l'<j} D^d \tond{D\mu_\flat}^{l+l'+d}  \leq
\cr
 &\leq K_3 C_{\ph_s} C_{\ph_{s'}} \tond{\frac2{A_2^2\beta}}^{s+s'+2}(s+s'+2)! 
       \quadr{\frac{\tond{D^2 \mu_\flat}^j + \tond{D^2 \mu_\flat}^{N-j}}
                   {1-D^2 \mu_\flat}} \ ,
\end{aligned}
\end{equation*}
where we used also the definitions~\eqref{e.C.et.alter} and~\eqref{e.K3}
of $K_2$ and $K_3$, the fact that $\mu_\sharp<\mu_\flat$ for $a<1/4$,
and $e^{-\sigma_s}<\mu_\flat$, and we worked out the sum like
in~\eqref{e.caseA1}, \eqref{e.caseA2} and \eqref{e.est.A}.

In the $B$ term, where $l+l'$ is relatively large, we exploit the exponential
decay; after the usual application of Schwartz inequality, by
Proposition~\ref{p.aver.1} we have
\begin{equation*}
\begin{aligned}
 \left|\sigma\tond{\ph_s^{(l)},\ph_{s',j}^{(l')}}\right| &\leq
 \sqrt{\inter{\tond{\ph_s^{(l)}}^2} \inter{\tond{\ph_{s',j}^{(l')}}^2}} +
 \left| \inter{\ph_s^{(l)}} \inter{\ph_{s',j}^{(l')}} \right| \leq
\cr
 &\leq \Biggl\{ \frac{(2K_1)^{1+4a_0}}{A_2^4} \; D^{2+\frac{l+l'}2} \;
       \sqrt{(2s+2)!(2s'+2)!} \;\times
\cr
 &\phantom{\leq\Biggl\{}
  \sqrt{\norm{\tond{\ph_s^{(l)}}^2} \norm{\tond{\ph_{s',j}^{(l')}}^2}} \;+\; 
  \norm{\ph_s^{(l)}} \norm{\ph_{s',j}^{(l')}} \;\times
\cr
 &\phantom{\leq\Biggl\{}  + \frac{(2K_1)^{2+8a_0}}{A_2^8}  D^{4+l+l'}
       (s+1)!(s'+1)! \Biggr\}  \tond{\frac2{A_2^2\beta}}^{\! s+s'+2}\!\!\! \leq
\cr
 &\leq K_3 (s+s'+2)! \tond{\frac4{A_2^2\beta}}^{s+s'+2}
            C_{\ph_s} C_{\ph_{s'}} (D\mu_\flat)^{l+l'} 
\end{aligned}
\end{equation*}
where we used again $\sqrt{(2s+2)!(2s'+2)!} < 2^{s+s'+2}(s+s'+2)!$,
$e^{-\sigma_s}<\mu_\flat$ and the definitions~\eqref{e.C.et.alter}
and~\eqref{e.K3} of $K_2$ and $K_3$. We thus have
\begin{equation*}
|B| \leq K_3 (s+s'+2)! \tond{\frac4{A_2^2\beta}}^{s+s'+2} 
 C_{\ph_s} C_{\ph_{s'}} \frac{(D\mu_\flat)^j}{1-D\mu_\flat} \ .
\end{equation*}

In order to sum over the various degrees, we observe that the terms
depending on $s$ and $s'$ are the same as in formula~\eqref{e.covs.3},
the only difference being here the sum starts from 0; at the same time,
since the dependence on $j$ is factorized, we deal with the sum over the
translation like in~\eqref{e.stimella}:
{\small
\begin{equation*}
\begin{aligned}
 \Biggl|&\sum_{h<k}\sigma\left(\varphi\circ\tau^{h},
   \varphi\circ\tau^{k}\right)\Biggr|
  \leq N \sum_{j=1}^{[N/2]} |\sigma\tond{\varphi,\varphi\circ\tau^{j}} |
  \leq N \sum_{j=1}^{[N/2]} \sum_{i=0}^{2r}\sum_{s+s'=i} |A|+|B| \leq
\cr
 &\leq
  N \sum_{j=1}^{[N/2]}
               \frac{(D^2\mu_\flat)^j+(D^2\mu_\flat)^{N-j}}{1-D^2\mu_\flat}
  \; 2K_3C_1^2  \sum_{i=0}^{2r} e^{(i+2)\tilde C}  F_r^{i-2} 
  \quadr{\tond{\frac4{A_2^2\beta}}^{\! i+2}\!\!\!(i+2)!} \leq
\cr
 &\leq 
 N \quadr{\frac{\tond{D^2\mu_\flat}}{(1-D^2\mu_\flat)^2}}
 \frac{2K_3C_1^2}{F_r^4} \sum_{i=0}^{2r}
  \quadr{\tond{\frac{4F_re^{\tilde C}}{A_2^2\beta}}^{i+2}(i+2)!} \leq
\cr
 &\leq 
 N \quadr{\frac{\tond{D^2\mu_\flat}}{(1-D^2\mu_\flat)^2}}
 \frac{K_3}{2^{12}3^4}\tond{\frac{\Omega}{C_1}\!}^{\!4}
   \!\frac1{r^8} \sum_{i=0}^{2r}
  \quadr{\tond{\frac{2^63C_1e^{\tilde C}r^2}
        {\Omega\tond{1-\mu_\flat^2}(1-\mu_\flat)A_2^2\beta}}^{\! i+2}
        \!\!\!(i+2)!}
.
\end{aligned}
\end{equation*}
} Also in this case we have a sum of the form $\sum_{i=0}^{2r}a_i$
just like in~\eqref{e.covs.4}, with $i$ starting from 0 instead of 1;
hence $\sum_{i=0}^{2r}a_i<(2r+1)a_0 < 8rX^2$. Since the leading terms
is of order $\mathcal{O}(1/\beta^2)$, a similar negative contribution
as in \eqref{e.sigma2.phi0} holds; thus with a possibly different
$\mu_2$ we have the thesis.
\end{proof}


\section{Appendix}
\label{s:6}

\subsection{Proofs of Section \ref{s:quadratic}}

We provide here the proofs of all the statements claimed in Section
\ref{s:quadratic} concerning the linear normalizing transformation
$A^{1/4}$.

\subsubsection{Proof of Proposition \ref{p.1}}
\label{app:quad}

If we apply the canonical linear change of coordinates $q = A^{1/4}x$ and $y =
A^{1/4}p$ then $H_0$ reads
\begin{displaymath}
H_0(q,p) = \frac12 p \cdot A^{1/2}p + \frac12 q\cdot A^{1/2} q\ .
\end{displaymath}
Since $A$ is circulant and symmetric, the same holds also for both
$A^{1/2}$ and $A^{1/4}$ (see property 4. stated after
Definition~\ref{d.circulant}); thus we can isolate the
diagonal\footnote{The coefficient of the diagonal is also the first
element of the first row, and thus, from property 3. of circulant
matrices stated after Definition~\ref{d.circulant}, it is the average of
the eigenvalues of $A^{1/2}$.} part of $A^{1/2}$
\begin{displaymath}
A^{1/2} = \Omega\Id + B\ ,
\end{displaymath}
and write
\begin{equation}
\label{e.dec.H0a}
 H_0 = H_\Omega + Z_0 \ ,
\qquad{\rm where}\quad
H_\Omega = \quadr{\frac\Omega2(q_1^2+p_1^2)}^\oplus \ ,
\quad
Z_0=\frac12p \cdot B p + \frac12 q\cdot B q \ .
\end{equation}
Hence it follows \eqref{e.Z0}, with $b_j({\mu})$ being the first half-row of the
circulant and symmetric matrix $B$. A direct computation gives\footnote{The
same conclusion follows also from the theory of the linear centralizer
unfolding (see \cite{Arn84}), or from the normal form construction performed
in Section 2 of \cite{GioPP12}} $\Poi{H_\Omega}{Z_0}=0$. The exponential decay
of the elements $b_j\sim {\mu}^j$ follows from the same argument we are going to
use to show the exponential decay of the elements of the matrix $A^{1/4}$.

Introducing $T=\tau+\tau^\top$, recalling~\eqref{e.def-A}, and by
expanding with respect to ${\mu}$ we get
\begin{displaymath}
A^{1/4} = \sqrt\omega\quadr{ \Id + \sum_{k\geq
    1}\binom{1/4}{k}(-1)^k({\mu} T)^k} \ .
\end{displaymath}
The exponential decay follows from the analysis of the first $[N/2]-1$
powers of the symmetric matrix $T$; indeed, for $k\leq [N/2]-1$ one has
\begin{displaymath}
T^k_{j,j+i} = 0,\qquad |i-[N/2]-1|\leq [N/2]-k-1\ ,
\end{displaymath}
which allows to claim immediately
\begin{displaymath}
\bigl(A^{1/4}\bigr)_{1,j} = c_j({\mu}){\mu}^{j-1},\qquad\qquad
j=1,\ldots,[N/2]+1\ .
\end{displaymath}

We want to prove that $|c_j({\mu})| = \mathcal{O}(2^j)$; we first
observe that for any $1\leq j\leq [N/2]+1$ it holds
\begin{displaymath}
\bigl(A^{1/4}\bigr)_{1,j} = \sqrt\omega \sum_{h\geq j-1}(-1)^h {\mu}^h
T^h_{1,j}\binom{1/4}{h}\ ,
\end{displaymath}
so that by inserting $\binom{1/4}{j} = \frac{(-1)^{j}(4j-5)!!}{4^j
  j!}$ it reads
\begin{align*}
\bigl(A^{1/4}\bigr)_{1,1} &= \sqrt\omega\quadr{1 + \sum_{h\geq
    1}\frac{{\mu}^h T_{1,1}^h}{h!}\frac{(4h-5)!!}{4^h}}\ ,\cr
\bigl(A^{1/4}\bigr)_{1,j} &= \sqrt\omega \sum_{h\geq j-1}\frac{{\mu}^h
  T_{1,j}^h}{h!}\frac{(4h-5)!!}{4^h}<0\ , \qquad j\geq 2\ .
\end{align*}
We stress that for $j\geq 2$ the element $\bigl(A^{1/4}\bigr)_{1,j}$
is negative. From the decomposition
\begin{displaymath}
T^k = \sum_{l=0}^k \binom{k}{l}\tau^{k-l}\tau^{-l} = \sum_{l=0}^k
\binom{k}{l}\tau^{k-2l},
\end{displaymath}
it is possible to give an estimate uniform in $i$ and $j$
\begin{displaymath}
 0<T^k_{ij} = \sum_{l=0}^k \binom{k}{l} \tau^{k-2l}_{ij} 
 < \sum_{l=0}^k \binom{k}{l} = 2^k \ ,
\end{displaymath}
since the elements of $\tau^{k-2l}$ are only $1$ and $0$. We can thus
estimate
\begin{align*}
\big|\bigl(A^{1/4}\bigr)_{1,j}\big| &\leq
-\sqrt\omega(2{\mu})^{j-1}\sum_{k\geq
  0}\frac{(2{\mu})^k}{(k+j-1)!}\frac{(4k+4j-9)!!}{4^{k+j-1}}<\\ &<
\sqrt\omega(2{\mu})^{j-1}\quadr{1-\sum_{k\geq
    1}\frac{(2{\mu})^k}{k!}\frac{(4k-5)!!}{4^{k}}} <
2\sqrt\omega(2{\mu})^{j-1}\ .
\end{align*}
The estimate holds also for $j=1$. We can furthermore apply the same
argument to $A^{1/2}$
\begin{equation*}
A^{1/2} = \Omega\Id + B,\qquad\qquad |b_j|\leq
2\omega(2{\mu})^{j-1},\qquad j\geq 2
\end{equation*}
thus showing that $B$ is a ${\mu}$ perturbation of $\Omega\Id$. The
behaviour of the coefficients $b_j$ also justify \eqref{e.Z0}. {The
  shape of the seed $\zeta_0$ is a direct consequence of
  \eqref{e.dec.H0a}: the vector $Bq$ is determined by its first
  element, being $(Bq)_j = \tau^{j-1}(Bq)_1$, and since we can write
  $q_j=\tau^{j-1}q_1$ for any $j$, then it follows that $q_0(Bq)_1$
  can be chosen as a seed for $Z_0$. The same holds for the conjugated
  variables $p$.} \qed

\subsubsection{Proof of Proposition \ref{p.chi0}}
\label{app:quad.chi0}

The flow of the Hamiltonian $\chi$ of~\eqref{e.chi0} is given by the
exponentials $x(t) = e^{Bt} x_0$ and $y(t)= e^{-Bt}y_0$, so the time
$t=1$ flow gives $A^{1/4}=e^B$ hence the definition $B:= \frac14
\ln\tond{A}$ which, due to \eqref{e.def-A}, reads
\begin{equation*}
B := \frac14\ln(\omega^2)\Id + \frac14\ln\tond{\Id-\mu
  T}\ ,\qquad\qquad \ln\tond{\Id-\mu T} := -\sum_{k\geq 1}\frac{1}k
\mu^k T^k\ .
\end{equation*}
From the series definition, we immediately get that $B$ is circulant
and symmetric. The exponential decay of its elements and the
subsequent upper bound can be obtained with the same estimates used in
the proof of Proposition \ref{p.1}. {The final decomposition and the
  claim $\chi_0\in\Dscr\bigl(C_0(a),\sigma_0\bigr)$ follow from the
  same argument used for \eqref{e.Z0.seed.1}: the first choice for the
  seed is $\chi_0=x_0(By)_1$, which gives $\chi_0^{(m)} =
  B_{1,m}(x_0y_m + y_0x_m)$. Then one uses translations to get
  \eqref{e.chi0}.} \qed

\subsubsection{Proof of Lemma \ref{l.inv.g}}
We apply Corollary \ref{cor.poisson} with $f=\Chi_0$, thus obtaining
$\lie{\Chi_0}\rho\in\Dscr\bigl(C C_\rho,\sigma_*\bigr)$ with
\begin{displaymath}
C\leq \frac{(r+1)2
  C_0}{(1-e^{-\sigma'})(1-e^{-(\sigma'-\sigma_*)})}\ .
\end{displaymath}
Then we recall that $T_{\Chi_0}\rho = \sum_{l\geq
  0}\frac1{l!}\lie{\Chi_0}^l\rho$, hence for any homogeneous term in the
series it holds
\begin{displaymath}
\big\|{\bigtond{\lie{\Chi_0}^l\rho}^{(m)}}\big\| \leq C^l C_\rho
e^{-m\sigma_*}\ .
\end{displaymath}
\qed

\subsubsection{Proof of lemma~\ref{l.exp.Z0}}
\label{app.exp.Z0}
The claim involving the seed of $Z_0$ is based on the choice of
$\zeta_0$ in Proposition~\ref{p.1}, namely formula \eqref{e.Z0}. One
has to observe that any term of the form $b_j(q_0q_{N-1-j} +
p_0p_{N-1-j})$ in the seed $\zeta_0$ can be replaced with the
corresponding $b_j(q_0q_{j} + p_0p_{j})$, obtained with the
translation $\tau^j$. Hence another seed for $Z_0$ (in the odd case,
for example) reads
\begin{equation}
\label{e.Z0.seed.1}
\zeta_0 = \sum_{m=1}^{[N/2]}\zeta_0^{(m)}\
,\qquad \zeta_0^{(m)}=2\sum_{m=1}^{[N/2]}b_m(q_0q_{m} + p_0p_{m})\ .
\end{equation}
The estimate for $\zeta_0$ then follows from
\begin{displaymath}
\norm{\zeta_0^{(m)}} = 4|b_m|\leq 8\omega\,\tond{2\mu}^m\ .
\end{displaymath}

{ In order to prove the claim involving $H_1 = h_1^{\oplus}$, we adapt
  the proof of Lemma \ref{l.inv.g}, working on the series expansion of
  the seed $T_{\Chi_0}h_1$. The proof exploits some peculiarities of
  the Lie derivatives $\lie{\Chi_0}^k h_1$, which are derived by
  $\chi_0^{(0)}=0$ and $h_1^{(m)}=0$ for all $m\geq 1$.  We have:
\begin{description}
\item[$\lie{\Chi_0}h_1$] we can write explicitely
\begin{displaymath}
\lie{\Chi_0}h_1^{(m)} = \sum_{l=1}^{m}B_{1,l}x_0^3 x_l\ ,\qquad
m=1,\ldots,[N/2]\ ,
\end{displaymath}
where $B_{1,l}$ are the coefficients of the matrix $B$ defining
$\Chi_0$ in Proposition \ref{p.chi0}. Then for any $m\geq 1$ it holds
\begin{displaymath}
\norm{\lie{\Chi_0}h_1^{(m)}} = \norm{\chi_0^{(m)}}\leq C_0(a)
e^{-\sigma_0m} = C_0(a) e^{-(\sigma_0-\sigma'')} e^{-\sigma'' m}\ ,
\end{displaymath}\
thus
$\lie{\Chi_0}h_1\in\Dscr(C_{1,1}:=C_0e^{-(\sigma_0-\sigma'')},\sigma'')$
for any $\sigma''<\sigma_0$. In particular
$\lie{\Chi_0}h_1\in\Dscr(C_{1,1},\sigma_1)$.
\item[$\lie{\Chi_0}^2h_1$]the second Lie derivative represents the
  Poisson bracket between $\Chi_0$ with $\chi_0\in\Dscr(C_0,\sigma_0)$
  and the previous Lie derivative seed $\lie{\Chi_0}h_1$, hence
  $\lie{\Chi_0}^2h_1\in\Dscr(C_{1,2},\sigma_1)$. Remarkably, the
  expansion of both the functions start with $m=1$, being
  $\chi_0^{(0)} = \lie{\Chi_0}h_1^{(0)}=0$, thus it is possible to
  apply Corollary \ref{cor.pois.Z0.corr} in order to obtain
\begin{displaymath}
C_{1,2} = \frac{16 e^{-(\sigma_0-\sigma_1)}
  C_{1,1}C_0}{(1-e^{-\sigma_0})(1-e^{-(\sigma_0-\sigma_1)})}\ .
\end{displaymath}
\item[$\lie{\Chi_0}^3h_1$]the third derivative represents the Poisson
  bracket between $\lie{\Chi_0}^2h_1\in\Dscr(C_{1,2},\sigma_1)$ and
  $\Chi_0$. Thus surely we have
  $\lie{\Chi_0}^3h_1\in\Dscr(C_{1,3},\sigma_1)$. Corollary
  \ref{cor.pois.Z0.corr} provides
\begin{displaymath}
C_{1,3} = \frac{16 e^{-(\sigma_0-\sigma_1)}
  C_0}{(1-e^{-\sigma_0})(1-e^{-(\sigma_0-\sigma_1)})}C_{1,2}C_0\ ;
\end{displaymath}
\item[$\lie{\Chi_0}^kh_1$]the iterated Lie derivative is still the
  Poisson bracket between an homogeneous polinomial
  $\lie{\Chi_0}^{k-1}h_1\in\Dscr(C_{1,k-1},\sigma_1)$ of degree $4$
  and $\Chi_0$. Corollary \ref{cor.pois.Z0.corr} provides
\begin{displaymath}
C_{1,k} = \frac{16 e^{-(\sigma_0-\sigma_1)}
    C_0}{(1-e^{-\sigma_0})(1-e^{-(\sigma_0-\sigma_1)})} C_{1,k-1}C_0\ ,
\end{displaymath}
which iteratively yields the estimate
\begin{equation}
\label{e.iter.C1k}
C_{1,k}\leq \quadr{\frac{16 e^{-(\sigma_0-\sigma_1)}
    C_0}{(1-e^{-\sigma_0})(1-e^{-(\sigma_0-\sigma_1)})}}^kC_0\ .
\end{equation}
\end{description}
By collecting all the Lie derivatives in the definition of the Lie
transform we have
\begin{align*}
\norm{(T_{\Chi_0}h_1)^{(m)}} &\leq \sum_{k\geq
  0}\frac1{k!}\norm{(\lie{\Chi_0}^k h_1)^{(m)}} \leq e^{-\sigma_1
  m}\sum_{k\geq 0}\frac1{k!} C_{1,k}\leq \\ &\leq
C_0e^{-\sigma_1 m}\sum_{k\geq 0}\frac1{k!}\quadr{\frac{16 e^{-(\sigma_0-\sigma_1)}
    C_0}{(1-e^{-\sigma_0})(1-e^{-(\sigma_0-\sigma_1)})}}^k =\\
&= C_1 e^{-\sigma_1 m}\ ,
\end{align*}
where, using $\sigma_0=2\sigma_1$, we have defined
\begin{displaymath}
C_1 := C_0 \exp{\tond{\frac{16 e^{-\sigma_0/2}
      C_0}{(1-e^{-\sigma_0})(1-e^{-\sigma_0/2})}}}\ .
\end{displaymath}
The claim $C_1(a)=\mathcal{O}(1)$ comes from $(a\to 0)\Rightarrow
(\sigma_0\to +\infty)$.
}
\qed


\subsection{Technical lemmas on Gibbs averages}
\label{app.gibbs}

\subsubsection{Useful integrals}
\label{sss.integrali}

We recall a couple of useful formulas\footnote{We also use the
  notation $k!!:=k[(k-2)!!]$ if $k>1$, and $k!!:=1$ if $k=1$ or
  $k=0$.} for integrals.  The first is
\begin{equation}
\label{e.Ik}
 I_k:=\int_{\RR} x^k e^{-x^2} dx
\qquad\Longrightarrow\qquad
 I_{2k} = \frac{2k-1}2 I_{2(k-1)} = \frac{(2k-1)!!}{2^k}\sqrt{\pi} \ .
\end{equation}
The second formula gives an estimate for
\begin{equation*}
 G_k(\gamma):=\int_{\RR} x^k e^{-(x^2+\gamma^2 x^4)} dx\ .
\end{equation*}
We have the following
\begin{lemma}
\label{l.stimaGk}
Let $\gamma>0$ satisfy $\gamma^2(2k+1)(2k+3) < 4$.  Then one has
\begin{equation*}
 1-\gamma^2\frac{(2k+1)(2k+3)}4
 \leq \frac{G_{2k}(\gamma)}{G_{2k}(0)}
 \leq 1 \ .
\end{equation*}
\end{lemma}

\begin{proof}
There exists $\xi\in[0,\gamma]$, such that
\begin{equation*}
 G_k(\gamma)-G_k(0) =
  \gamma \frac{\partial G_k}{\partial \gamma}(\xi) =
  -2\gamma\xi G_{k+4}(\xi) \ .
\end{equation*}
Since $G_k$ is a monotonic function of $\gamma$ we can
estimate $\xi G_{k+4}(\xi)$ both from above and from below as
$0 \leq \xi G_{k+4}(\xi) \leq \gamma G_{k+4}(0)$ for all
$\xi\in[0,\gamma]$.  In view of $G_k(0)=I_k$ and using twice the recursive
relation \eqref{e.Ik} the claim follows.
\end{proof}

\subsubsection{Partition function}
\label{sss.Z}

Throughout the present Section we consider a periodic chain of
length $l$ and drop the $y$ depending part of $H$ because the
corresponding integrals are trivially factorized.  We also split $H_0$
as
\begin{equation}
\label{e.Hd+Ha}
  H_0 = H_d+H_a\ ,\quad
  H_d = \frac{1}{2} \sum_{j} x_j^2\ ,\quad
  H_a =  \frac{a}{2} \sum_{j} (x_{j+1}-x_j)^2
\end{equation}
where $a$ plays the role of the small coupling parameter.  We also
recall that $H_1=\frac{1}{4}\sum_{j}x_j^4$.  Finally we emphasize the
role of the parameter $a$ by denoting the partition function as
$Z(\beta,a)$.

We present three Lemmas investigating the role of boundary conditions and that
of the parameters $N$, $a$ and $\beta$ on the partition function; we start
recalling the following
\begin{lemma}[Lemma 5 and 6 of \cite{CarM12}]
\label{l.part.N}
There exist constants $\beta_0>0,\,a_0>0,\,K_1>2$ such that for any
$\beta>\beta_0$ and $0<a<a_0$ one has
\begin{equation*}
 Z_l(\beta) \geq
 Z_{l-1}(\beta)\frac1{K_1}\sqrt{\frac{2\pi}{\beta}} \ .
\end{equation*}
Denoting by ${\mathcal Z}$ the partition function for the system with free boundary
 conditions\footnote{see~\eqref{e.calZH}}, one has
\begin{equation*}
 {\mathcal Z}_l(\beta) \leq
 Z_l(\beta) \left(2K_1\right)^{1+4a_0} \ .
\end{equation*}
\end{lemma}
\begin{proof}
For the proof we refer to \cite{CarM12}; please note that our variables
$x_j$ correspond to their variables $q_j$ after a rescaling: $q_j=\sqrt\omega
x_j$.
\end{proof}

\begin{corollary}
\label{c.part.N}
Under the same hypothesis of Lemma~\ref{l.part.N}, for the
 scaled\footnote{see the paragraph before definitions~\eqref{e.mis.scal}
 in the proof of Proposition~\ref{p.aver.2}} quantities one has
\begin{equation*}
 Z_l^\scal(\beta) \geq Z_{l-1}^\scal(\beta)\frac{\sqrt{2\pi}}{K_1}
\qquad\qquad
 {\mathcal Z}_l^\scal(\beta) \leq Z_l^\scal(\beta) \left(2K_1\right)^{1+4a_0} \ .
\end{equation*}
\end{corollary}

In the following Lemma we control the complete partition function
$Z(\beta,a)$ with the totally uncoupled one $Z(\beta,0)$, exploiting
the decomposition \eqref{e.Hd+Ha}. We remark that it is independent of
the boundary conditions.
\begin{lemma}
\label{l.m.1bis}
Let $A_1:=\sqrt{1+4a}$. If $a>0$ then
\begin{displaymath}
\frac1{A_1^l} Z_l(\beta,0)\leq Z_l(\beta,a) \leq Z_l(\beta,0) \ .
\end{displaymath}
\end{lemma}

\begin{proof}
Using the trivial inequality $(v-w)^2\leq 2(v^2+w^2)$, when $a>0$ we use
the estimate $H_a\leq 4aH_d$ to get the following inequalities
\begin{equation*}
\label{e.stimeZ}
 H_d+H_1 \leq H \leq (1+4a) H_d + (1+4a)^2 H_1\ . 
\end{equation*}
Inserting the previous formula in the explicit expressions for
$Z(\beta,a)$ and $Z(\beta,0)$ we have
$$
\displaylines{\qquad
 \int_{\RR^l}e^{-\beta (H_d+H_1)} dx 
\geq \int_{\RR^l}e^{-\beta H} dx \geq
\hfill\cr\hfill
\geq \left[\int_{\RR}e^{-\beta \left((1+4a) \frac{x^2}{2} + (1+4a)^2
    \frac{x^4}{4}\right)} dx  \right]^l
    = \frac1{A_1^l} \int_{\RR^l}e^{-\beta (H_d+H_1)} dx
    \ ,\qquad\cr
    }
$$
using, in the last equality, the change of variable $y=\sqrt{1+4a} x$.
\end{proof}

As a corollary of Lemma~\ref{l.stimaGk}, in the uncoupled case it is
possible to control the partition function $Z(\beta,0)$ with the
corresponding harmonic\footnote{The corresponding scaled version is
clearly $Z_l^\scal(\infty,0):= \left(2\pi\right)^{l/2}$.} one, which can
be denoted by
$Z_l(\infty,0):=\int_{\RR^l}e^{-\beta H_d} dx =
\left(\frac{2\pi}{\beta}\right)^{l/2}$:
\begin{lemma}
\label{l.stimaZ}
For $\beta>\sqrt{3/4}$, denoting $B:=1-\frac3{4\beta^2}<1$, one has
\begin{equation*}
\label{e.stimaZ}
 B^l \leq \frac{Z_l(\beta,0)}{Z_l(\infty,0)} \leq 1 \ .
\end{equation*}
\end{lemma}

\begin{proof}Rescale $Z$ by $\sqrt\beta$, then apply Lemma~\ref{l.stimaGk}.
\end{proof}

As a corollary of the previous Lemmas, one has
\begin{corollary}
\label{l.part.Nl}
Let $\beta_0$, $a_0$ and $K_1$ be the constants of Lemma~\ref{l.part.N}, then
for any $\beta>\beta_0$, $0<a<a_0$ and $l'<l$ one has
\begin{equation*}
\label{e.part.Nl}
  \frac{Z_{l-l'}(\beta,a)Z_{l'}(\beta,a)}{Z_l(\beta,a)} \leq K_1^{l'} 
 \ .
\end{equation*}
\end{corollary}
\begin{proof}
Apply Lemma~\ref{l.part.N} $l'$ times, then apply Lemma~\ref{l.m.1bis}
 and  Lemma~\ref{l.stimaZ}.
\end{proof}

\begin{remark}
The previous statements, Lemma~\ref{l.m.1bis}, Lemma~\ref{l.stimaZ} and
Corollary~\ref{l.part.Nl}, holds also for the scaled quantities
$Z^\scal$, with identical proofs.
\end{remark}

\subsubsection{Gibbs averages of monomial}
\label{sss.aver}

We first describe the possibility of cutting a (periodic) chain into two
(periodic) subchains, showing how the respective averages are related. We
consider in general two functions $f$ and $g$ with disjoint supports; we have
the following
\begin{lemma} 
\label{l.cutting}
In a periodic chain of length $l$, let $f$ and $g$ be two functions with
disjoint supports, in particular $\supp(f)\subseteq \{0,\ldots,l'-1\}$ and
$\supp(g)\subseteq \{l',\ldots,l-1\}$. Let $\beta_0$, $a_0$ and $K_1$ be the
constants of Lemma~\ref{l.part.N}, then for any $\beta>\beta_0$ and $0<a<a_0$
one has
\begin{align*}
 \inter{fg}_l \leq K_1^{l'} 
  \inter{fe^{\beta a\left(x_{l'-1}^2+x_0^2\right)}}_{l'}
  \inter{ge^{\beta a\left(x_{l-1}^2+x_{l'}^2\right)}}_{l-l'} \ ,
\end{align*}
and an identical inequality, but without $\beta$ in the exponential factors,
for the scaled case.
\end{lemma}
\begin{proof}
We start writing the averaged quantity splitting the measure in the following
way:
\begin{equation*}
 \inter{fg}_l = \frac1{Z_l} \int_{\RR^l}[0,f,l'-1][l'-1,l'][l',g,l-1][l-1,l]
  dV_0^{(l')} dV_{l'}^{(l-l')} \ .
\end{equation*}
If we could remove the terms $[l'-1,l']$ and $[l-1,l]$ we would end up in two
separated chain with free boundaries; to recover periodicity in both chains we
should multiply by $e^{\beta a x_{l'-1}x_0}$ and $e^{\beta a x_{l-1}x_{l'}}$;
now observe that, using $|bc|<\frac12(b^2+c^2)$, one has
\begin{align*}
[l'-1,l'] &\frac{e^{\beta a x_{l'-1}x_0}}{e^{\beta a x_{l'-1}x_0}}
[l-1,l] \frac{e^{\beta a x_{l-1}x_{l'}}}{e^{\beta a x_{l-1}x_{l'}}} \leq
\cr
\leq \;&e^{\beta a x_{l'-1}x_0} e^{\beta a\left(x_{l'-1}^2+x_0^2\right)}
e^{\beta a x_{l-1}x_{l'}} e^{\beta a\left(x_{l-1}^2+x_{l'}^2\right)} \ ,
\end{align*}
so we have
\begin{equation*}
 \inter{fg}_l \leq \frac{Z_{l'}Z_{l-l'}}{Z_l}
 \inter{f e^{\beta a\left(x_{l'-1}^2+x_0^2\right)}}_{l'}
 \inter{g e^{\beta a\left(x_{l-1}^2+x_{l'}^2\right)}}_{l-l'} \ ;
\end{equation*}
the thesis now follow from Corollary~\ref{l.part.Nl}. For scaled case, simply
add $\scal$ everywhere and remove $\beta$ as needed in the proof above.
\end{proof}

In the next Lemma we are interested in the phase average of a monomial: in
this case the dimension of the space is not a critical aspect of the estimate,
so we can simply perform our estimate passing to the uncoupled quantities
paying the price of the power of a constant to the dimension of the space.
\begin{lemma}
\label{c.aver.2}
\label{c.aver.3}
Let $x^k$ be a monomial of degree $2r$, $A_2:=\sqrt{1-2a}$, $a<1/2$ and
$\beta>\sqrt{3/4}$, then
\begin{align*}
 \inter{x^k}_l &\leq A_1^l \inter{x^k}_{l,0} \ ,
\cr
 \inter{x^ke^{\beta a\left(x_0^2+x_{l-1}^2\right)}}_l  &\leq
   \frac{A_1^l}{A_2^{2+k_0+k_{l-1}}} \inter{x^k}_{l,0}
\cr
  \inter{x^k}_{l,0} &\leq \frac1{B^l} \frac{2^r r!}{\beta^r} \ .
\end{align*}
\end{lemma}

\begin{proof}
For the first inequality one simply use $H_a>0$ and Lemma~\ref{l.m.1bis}
to get
\begin{equation*}
 \inter{x^k}_l = \frac1{Z_l(\beta,a)}
 \int_{\RR^l} \!\!\! x^k e^{-\beta\left(H_d+H_a+H_1\right)} \leq 
 \frac{A_1^l}{Z_l(\beta,0)} 
 \int_{\RR^l} \!\!\! x^k e^{-\beta\left(H_d+H_1\right)}
 \ ,
\end{equation*}
which is actually $A_1^l \inter{x^k}_{l,0}$.
For the second inequality we further use $e^{{z^4}/4}\leq
e^{(z^4/A_2)^4/4}$ in the previous estimate, and then we scale by
$A_2$ to the ``boundary'' variables appearing in the exponential factor.

For the last inequality, use $H_1>0$ and Lemma~\ref{l.stimaZ}, one first
exploit the control of the (decoupled) nonlinear Gibbs measure by the
(decoupled) harmonic one:
\begin{equation*}
 \inter{x^k}_{l,0} \leq
  \frac1{B^l Z_l(\infty,0)} \int_{\RR^l} x^k e^{-\beta H_d} \ .
\end{equation*}
Then, concerning the numerator we have
\begin{align*}
 \int_{\RR^l} x^k e^{-\beta H_d}
  = \prod_{j=1}^l \int_\RR x^{k_j} e^{-\beta \frac{x^2}2} dx
 &= \prod_{j=1}^l \left(\frac2{\beta}\right)^{\frac{k_j+1}2}
                \Gamma\left(\frac{k_j+1}2\right) \leq
\cr
 &\leq \left(\frac{2}{\beta}\right)^r \left(\sqrt\frac{2\pi}{\beta}\right)^l
       \prod_{j=1}^l \left(\frac{k_j}2\right)! \ ,
\end{align*}
where we used, when $k_j$ is even, the relation
$\Gamma\left(n+\frac12\right) = \frac{(2n)!}{4^n n!}\sqrt{\pi} \leq n!
\sqrt{\pi}$; if $k_j$ is odd, and we actually estimate
$\int|x^{k_j}|e^{-\beta H_d}$, then
$\Gamma\left(\frac{k_j+1}2\right)=\left(\frac{k_j-1}2\right)!$ which
bounded by the estimate of the other case.  As a last step we recall
that $\prod_j \left(m_j!\right) \leq \left(\sum_j m_j\right)!$ to get
the thesis.
%
%
\end{proof}

In the next Lemma, despite its similarity with the previous one, we need a
different approach, since it will be applied in a case with the dimension of
the space growing with $N$; thus it is not possible to pay the price of the
constant $A_1^l$.

\begin{lemma}
\label{c.aver.4}
Let $\beta_0$, $a_0$ and $K_1$ be the constants of Lemma~\ref{l.part.N}; for
every $0\leq i<l$ it holds
\begin{equation*}
\label{e.aver.3.l}
 \biginter{e^{\beta a\left(x_i^2+x_{i+1}^2\right)}}_l  \leq 
\left(2K_1\right)^{\left(1+4a_0\right)}
 \left(\frac{D}{A_2}\right)^2 \ ,
\qquad
D:=\frac{K_1A_1}{B} \ .
\end{equation*}
\end{lemma}

\begin{proof}
Let us write first the Hamiltonian isolating the energy of two chains
with free boundaries\footnote{We recall the use of the calligraphic
letters ${\mathcal H}$ and ${\mathcal Z}$ for the free boundary cases
(see~\eqref{e.calZH}).}, the first defined on the two sites $i$ and
$i+1$ and the second on the remaining sites, plus the connecting
springs:
\begin{align*}
 H(x) =
  {\mathcal H}^{(2)}(x_i,x_{i+1})
 &+ {\mathcal H}^{(l-2)}(x_{i+2},x_{i+3},\ldots,x_{i-1})+
\cr 
 &+ \frac{a}2\left(x_{i-1}-x_i\right)^2
  + \frac{a}2\left(x_{i+1}-x_{i+2}\right)^2 \ ;
\end{align*}
then, estimating by above with 1 the contribution of the connecting springs
\begin{equation*}
\begin{aligned}
 \biginter{e^{\beta a\left({x_i^2+x_{i+1}^2}\right)}}_l &\leq
  \frac1{Z_l} \int_{\RR^{l-2}} e^{-\beta {\mathcal H}^{(l-2)}}
    \int_{\RR^2} e^{\beta a\left(x_i^2+x_{i+1}^2\right)}
      e^{-\beta {\mathcal H}^{(2)}} dx_1\cdots dx_l = 
\\
 &= \frac{Z_{l-2}Z_2}{Z_l} \; \frac{{\mathcal Z}_{l-2}}{Z_{l-2}} \;
    \frac{\int_{\RR^2} e^{\beta a\left(x_i^2+x_{i+1}^2\right)}
      e^{-\beta {\mathcal H}^{(2)}} dx_idx_{i+1}}{Z_2}  
\\
 &\leq K_1^2 \; \left(2K_1\right)^{\left(1+4a_0\right)} \; 
    \frac{\int_{\RR^2} e^{-\beta (1-2a)\left(x_i^2/2+x_{i+1}^2/2\right)}dx_idx_{i+1}}
         {A_1^{-2}Z_2(\beta,0)} \ ,
\end{aligned}
\end{equation*}
where we used Corollary~\ref{l.part.Nl}, Lemma~\ref{l.part.N} and
Lemma~\ref{l.m.1bis}; then using also Lemma~\ref{l.stimaZ} and a
rescaling like in Lemma~\ref{c.aver.2} one has the thesis.
\end{proof}

In the next Proposition we consider the average of monomial whose support is
not the entire space. It is actually a trivial consequence of the
previous results.

\begin{proposition}
\label{p.aver.1}
Let $\beta_0$, $a_0$ and $K_1$ be the constants of Lemma~\ref{l.part.N}, and
let $x^k$ be a left aligned even monomial of degree $2r$ and interaction range
$l'<l$. Then for any $\beta>\beta_0$ and $a<a_0$ one has
\begin{equation*}
\label{e.aver.1}
 \biginter{x^k}_l \leq
 \frac{\left(2K_1\right)^{\left(1+4a_0\right)}}{A_2^{4+k_0+k_{l'-1}}}
 D^{2+l'} \; \frac{2^rr!}{\beta^r} \ .
\end{equation*}
\end{proposition}
\begin{proof}
Use Lemma~\ref{l.cutting} to cut the chain at the boundaries of the
support of $x^k$; then apply Lemma~\ref{c.aver.2} and \ref{c.aver.4} to
the resulting terms.
\end{proof}

\begin{remark}
\label{r.scal.2}

In all the previous results, i.e. Lemma~\ref{l.cutting},
\ref{c.aver.2} and \ref{c.aver.4} and Proposition~\ref{p.aver.1}, the
same estimates hold true if one substitutes every average
$\inter{\cdot}$ with the corresponding scaled version
$\inter{\cdot}^\scal$ and remove $\beta$. The proofs are almost, if
not completely, identical.
\end{remark}

\subsubsection{Gibbs averages of polynomials of class $\Dscr(C,\sigma)$}
\label{sss.l.DCs}

In this subsection we consider polynomials instead of monomials, but we
restrict to those with an exponential decay of the interaction range, as
described in Section~\ref{ss.int.range} (we recall in particular the
decomposition~\eqref{e.decomp} and the definition~\eqref{dcdm.5} of
class $\Dscr(C,\sigma)$). A first result is the following

\begin{lemma}
\label{l.phi.aver}
If $\ph\in\Dscr(C_{\ph},\sigma_s)$ is a polynomial of degree $2s+2$
with $s\leq r$, if $De^{-\sigma_s}\leq 1/2$ then
\begin{equation}
\label{e.inter.ph}
\inter{\ph} \leq
\quadr{\frac{2D^2(2K_1)^{1+4a_0}}{A_2^4}}
C_{\ph}
\quadr{\tond{\frac2{A_2^2\beta}}^{s+1}(s+1)!}
\end{equation}
\end{lemma}

\begin{proof}
Using the decomposition~\eqref{e.decomp}, let us write $\inter{\ph} =
\sum_{l}\inter{(\ph)^{(l)}}$. By applying Proposition~\ref{p.aver.1} we can
estimate each addendum as
\begin{displaymath}
\inter{\ph^{(l)}} \leq \frac{(2K_1)^{1+4a_0}}{A_2^4}
D^{2+l} \tond{\frac2{A_2^2\beta}}^{s+1} (s+1)!
\norm{\ph^{(l)}} \ ;
\end{displaymath}
performing the sum over $l$, we have
\begin{displaymath}
\sum_{l}\inter{\ph^{(l)}} \leq \frac{(2K_1)^{1+4a_0}}{A_2^4}
\tond{\frac2{A_2^2\beta}}^{s+1} (s+1)! D^2 C_\ph
\sum_{l=0}^{N-1}\tond{D e^{-\sigma_s}}^l \ ,
\end{displaymath}
and from the estimate $\sum_{l=0}^{N-1}\tond{D e^{-\sigma_s}}^l \leq 2$
(the latter being true due to condition $D\mu_\flat\leq 1/2$) the thesis follow.
\end{proof}

We are interested in a similar estimate for the square of a given
polynomial; to this purpose we first give the following

\begin{lemma}
\label{l.phi2.decay}
If $\ph\in\Dscr(C_{\ph},\sigma_s)$ then
$\ph^2\in\Dscr(2C_{\ph}^2/(1-e^{-\sigma_s}),\sigma_s)$.
\end{lemma}

\begin{proof} We write $\ph^2 = \sum_{l=0}^{N-1}(\ph^2)^{(l)}$, with
\begin{displaymath}
(\ph^2)^{(l)} = \tond{\ph^{(l)}}^2 + 2\ph^{(l)}\sum_{l'=0}^{l-1}\ph^{(l')}\ ,
\end{displaymath}
hence
\begin{displaymath}
\norm{(\ph^2)^{(l)}}\leq
2\norm{\ph^{(l)}}\sum_{l'=0}^l\norm{\ph^{(l')}}\leq
2C_{\ph}^2e^{-\sigma_s l}\sum_{l'=0}^le^{-\sigma_s
  l'}<\frac{2C_{\ph}^2}{1-e^{-\sigma_s}}e^{-\sigma_s l}\ .
\end{displaymath}
\end{proof}

A trivial consequence of the two previous Lemmas is then the following
\begin{corollary}
\label{l.phi2.aver}
If $\ph\in\Dscr(C_{\ph},\sigma_s)$ is a polynomial of degree
$2s+2$, if $De^{-\sigma_s}\leq 1/2$ then
\begin{equation}
\label{e.inter.ph2}
\inter{\ph^2} \leq
\quadr{\frac{4D^2(2K_1)^{1+4a_0}}{A_2^4}}
\frac{C_{\ph}^2}{1-e^{-\sigma_s}}
\quadr{\tond{\frac2{A_2^2\beta}}^{2s+2}(2s+2)!}
\end{equation}
\end{corollary}

\subsection{Cancellations}
\label{app.canc}

In this appendix we provide the formal proof of the cancellations that take
place in the correlation terms of Proposition~\ref{p.aver.2}.

We first recall a few basic facts about binary strings used in the
proof.  Define $\Sigma_m := \left\{{\tt k}={\tt k}_1\cdots{\tt k}_m
\quad {\rm s.t.} \quad {\tt k}_l\in\{{\tt 0},{\tt 1}\} \ \forall
l\right\}$ the space of binary strings of finite length.  We shall
denote by $\ell({\tt k})$ the length of the string {\tt k}.  Recall that
the usual bitwise ``and'' operator $\wedge\>
:\>\Sigma_m\times\Sigma_m\rightarrow\Sigma_m$ is defined as
\begin{equation}
\label{e.andb}
({\tt i},{\tt j}) \mapsto {\tt k} = {\tt i}\wedge {\tt j} \ ,\quad
{\tt k}_l=
             \begin{cases}
	      {\tt 1}\quad {\tt i}_l\!=\!1 \ {\rm and}\  {\tt j}_l\!=\!1
	      \cr
	      {\tt 0}\quad {\rm otherwise}
	     \end{cases}
\end{equation}
\begin{remark}
\label{r.canc}
An elementary property is the following: given two strings {\tt j} and {\tt
k} in $\Sigma_m$, if ${\tt j} \wedge {\tt k}=0$, then the total number of
zeros contained collectively in {\tt j} and {\tt k} is at least $m$.
\end{remark}

A relevant point is that the presence of a {\tt 1} allows us to
decouple the measure as follows.  Expand $[p,r]=\sum_{{\tt j}}
[p,r]_{{\tt j}}$ and select a string ${\tt j}={\tt v1w}$ with a {\tt
  1} in the position corresponding to the site $q$.  Then we have
$[p,r]_{{\tt j}}=[p,q]_{{\tt v}}[q+1,r]_{{\tt w}}$ in place of the
general decomposition $[p,r]=[p,q][q,r]$.  We shall use this property
in dealing with expressions~\eqref{e.canc.int} in cases one has ${\tt
  j} \wedge {\tt j}'\neq0$ and ${\tt k} \wedge {\tt k}'\neq0$.  For,
if ${\tt j}\wedge {\tt j}'\ne {\tt 0}$, then both {\tt j} and {\tt j}'
have a {\tt 1} in the same position, so that a decomposition as above
may take place in the corresponding site of the chain.

Consider the set
\begin{align*}
 \Sigma:=\bigl\{\{{\tt j}, {\tt k}, {\tt j}', {\tt k}'\} \;\colon\;
  &{\tt j} \wedge {\tt j}'\neq0 \,,\; {\tt k} \wedge {\tt k}'\neq0 \,,\;
\cr
  &\ell({\tt j})=\ell({\tt j}')=l \,,\;
   \ell({\tt k})=\ell({\tt k}')=l' \bigr\} \ .
\end{align*}
We introduce now an equivalence relation as follows.  Let ${\tt
  j}\wedge{\tt j}'\ne 0$ and look for the first digit ${\tt 1}$ in
${\tt j}\wedge {\tt j}'$.  Then split ${\tt j}={\tt v1w}$ and ${\tt
  j}'={\tt v}'{\tt 1w}'$ with $\ell({\tt v})=\ell({\tt v}')$ and with
${\tt v}\wedge{\tt v}'=0$.  Similarly, split ${\tt k}={\tt x1y}$ and
${\tt k}'={\tt x}'{\tt 1y}'$ with $\ell({\tt x})=\ell({\tt x}')$ and
${\tt x}\wedge{\tt x}'=0$.  We say that $\{{\tt a}, {\tt b}, {\tt a}',
{\tt b}'\}\in\Sigma$ is equivalent to  $\{{\tt j}, {\tt
k}, {\tt j}', {\tt k}'\}$ in case it can by obtained by an arbitrary
exchange of the pairs $({\tt v},{\tt v}')$, $({\tt w},{\tt w}')$,
$({\tt x},{\tt x}')$ and $({\tt y},{\tt y}')$.  More precisely,  $({\tt a}, {\tt b}, {\tt a}', {\tt
  b}')$ must be equal to one of
$$
\vcenter{\openup1\jot\halign{
$\displaystyle{#}\,$\hfil
&$\displaystyle{#}\,$\hfil
&$\displaystyle{#}\,$\hfil
&$\displaystyle{#}$\ ,\hfil\quad
&$\displaystyle{#}\,$\hfil
&$\displaystyle{#}\,$\hfil
&$\displaystyle{#}\,$\hfil
&$\displaystyle{#}$\ ,\hfil
\cr
  \{{\tt v} {\tt 1w}, &{\tt x} {\tt 1y}, &{\tt v}'{\tt 1w}',&{\tt x}'{\tt 1y}'\}
& \{{\tt v}'{\tt 1w}, &{\tt x} {\tt 1y}, &{\tt v} {\tt 1w}',&{\tt x}'{\tt 1y}'\}
\cr
  \{{\tt v} {\tt 1w}, &{\tt x} {\tt 1y}',&{\tt v}'{\tt 1w}',&{\tt x}'{\tt 1y} \}
& \{{\tt v}'{\tt 1w}, &{\tt x} {\tt 1y}',&{\tt v} {\tt 1w}',&{\tt x}'{\tt 1y} \}
\cr
  \{{\tt v} {\tt 1w}, &{\tt x}'{\tt 1y}, &{\tt v}'{\tt 1w}',&{\tt x} {\tt 1y}'\}
& \{{\tt v}'{\tt 1w}, &{\tt x}'{\tt 1y}, &{\tt v} {\tt 1w}',&{\tt x} {\tt 1y}'\}
\cr
  \{{\tt v} {\tt 1w}, &{\tt x}'{\tt 1y}',&{\tt v}'{\tt 1w}',&{\tt x} {\tt 1y} \}
& \{{\tt v}'{\tt 1w}, &{\tt x}'{\tt 1y}',&{\tt v} {\tt 1w}',&{\tt x} {\tt 1y} \}
\cr
  \{{\tt v} {\tt 1w}',&{\tt x} {\tt 1y}, &{\tt v}'{\tt 1w}, &{\tt x}'{\tt 1y}'\}
& \{{\tt v}'{\tt 1w}',&{\tt x} {\tt 1y}, &{\tt v} {\tt 1w}, &{\tt x}'{\tt 1y}'\}
\cr
  \{{\tt v} {\tt 1w}',&{\tt x} {\tt 1y}',&{\tt v}'{\tt 1w}, &{\tt x}'{\tt 1y} \}
& \{{\tt v}'{\tt 1w}',&{\tt x} {\tt 1y}',&{\tt v} {\tt 1w}, &{\tt x}'{\tt 1y} \}
\cr
  \{{\tt v} {\tt 1w}',&{\tt x}'{\tt 1y}, &{\tt v}'{\tt 1w}, &{\tt x} {\tt 1y}'\}
& \{{\tt v}'{\tt 1w}',&{\tt x}'{\tt 1y}, &{\tt v} {\tt 1w}, &{\tt x} {\tt 1y}'\}
\cr
  \{{\tt v} {\tt 1w}',&{\tt x}'{\tt 1y}',&{\tt v}'{\tt 1w}, &{\tt x} {\tt 1y} \}
& \{{\tt v}'{\tt 1w}',&{\tt x}'{\tt 1y}',&{\tt v} {\tt 1w}, &{\tt x} {\tt 1y} \}
\cr
}}
$$ where some of the above combinations may well coincide, e.g., if
${\tt v}={\tt v}'$.  The above list of sixteen combinations actually
describes the equivalence classes generated by our relation.
\begin{remark}
It is clearly sufficient to show that cancellations occurs within each
equivalence class, and this is what we are going to do.
\end{remark}

Let us now pick an equivalence class, i.e., chose one of its elements denoted
again by $\{{\tt j}, {\tt k}, {\tt j}', {\tt k}'\} = \{{\tt v} {\tt 1w}, {\tt x}
{\tt 1y}, {\tt v}'{\tt 1w}',{\tt x}'{\tt 1y}'\}$, and exploit the decomposition
corresponding to the {\tt 1} digits. For the term $\dinter{\phi\psi}Z$
in~\eqref{e.canc.int}, writing only the integrand function, we get
$$
\vcenter{\openup1\jot\halign{
$\displaystyle{#}$\hfil
\cr
\bigl([0, \phi, m\!-\!1]\cdot [m\!-\!1,t]_{\tt j}\cdot 
     [t, \psi, t+m'\!-\!1]\cdot [t+m'\!-\!1,N]_{\tt k}\bigr) \times
\cr\qquad\qquad\qquad
\bigl([0,m\!-\!1]\cdot [m\!-\!1,t]_{{\tt j}'}\cdot 
     [t,t+m'\!-\!1]\cdot [t+m'\!-\!1,N]_{{\tt k}'}\bigr) =
\cr\qquad
 \bigl([u'\!-\!1,N]_{{\tt y}}\cdot [0, \phi, m\!-\!1]\cdot
 [m\!-\!1,u-1]_{{\tt v}}\bigr)\times
\cr\qquad\qquad
     \bigl([u,t]_{{\tt w}}\cdot [t, \psi, t+m'\!-\!1]\cdot 
      [t+m'\!-\!1,u'-1]_{{\tt x}}\bigr)\times
\cr\qquad\qquad\qquad
 \bigl([u'\!-\!1,N]_{{\tt y}'}\cdot [0,m\!-\!1]\cdot
   [m\!-\!1,u-1]_{{\tt {\tt v}}'}\bigr)\times 
\cr\qquad\qquad\qquad\qquad
     \bigl([u,t]_{{\tt w}'}\cdot [t,t+m'\!-\!1]\cdot
      [t+m'\!-\!1,u'-1]_{{\tt x}'}\bigr)\ ,
\cr
}}
$$
where $u$ and $u'$ denote the sites corresponding to the {\tt 1},
i.e., $[m\!-\!1,t]_{\tt j} = [m\!-\!1,u-1]_{{\tt v}} [u,t]_{{\tt w}}$ and 
$[t+m'\!-\!1,N]_{\tt k} = [t+m'\!-\!1,u'-1]_{{\tt x}} [u'\!-\!1,N]_{{\tt y}}$,
and the same for the primed strings. We do not write explicitly the analogous
equality for the term in $\dinter{\phi}\dinter{\psi}$.

The argument that follows becomes more transparent by 
introducing a compact notation
for each of the lines in the above formula, namely
\begin{align*}
\phantom{\big\vert}
[{\tt y} : \phi : {\tt v}]
&:= [u'\!-\!1,N]_{{\tt y}}\cdot [0, \phi, m\!-\!1]\cdot
    [m\!-\!1,u-1]_{{\tt v}}\ ,
\cr
\phantom{\big\vert}
[{\tt w} : \psi : {\tt x}]
&:= [u,t]_{{\tt w}}\cdot [t, \psi, t+m'\!-\!1]\cdot 
     [t+m'\!-\!1,u'-1]_{{\tt x}}\ ,
\cr
\phantom{\big\vert}
[{\tt y}' :\,: {\tt v}']
&:= [u'\!-\!1,N]_{{\tt y}'}\cdot [0,m\!-\!1]\cdot [m\!-\!1,t]_{{\tt v}}\cdot
   [m\!-\!1,u-1]_{{\tt {\tt w}}'}\ ,
\cr
\phantom{\big\vert}
[{\tt w}' :\,: {\tt x}']
&:= [u,t]_{{\tt w}'}\cdot [t,t+m'\!-\!1]\cdot
      [t+m'\!-\!1,u'-1]_{{\tt x}'}\ .
\end{align*}
Using such a notation, and writing the corresponding terms for the whole
expression $\dinter{\phi\psi}Z-\dinter{\phi}\dinter{\psi}$ one has
\begin{equation*}
[{\tt y}:\phi:{\tt v}] \cdot
[{\tt w}:\psi:{\tt x}] \cdot
[{\tt y}':\,:{\tt v}'] \cdot
[{\tt w}':\,:{\tt x}'] 
\; - \;
[{\tt y}:\phi:{\tt v}] \cdot
[{\tt w}:\,:{\tt x}] \cdot
[{\tt y}':\,:{\tt v}'] \cdot
[{\tt w}':\psi:{\tt x}'] \ ;
\end{equation*}
we remark that each factor in the above expression is has a
support disjoint from the supports of the other factors, so that the integral
can be factorized.

Select now $\{{\tt v} {\tt 1w}, {\tt x} {\tt 1y}, {\tt v}'{\tt 1w}',{\tt
x}'{\tt 1y}'\}$ and $\{{\tt v} {\tt 1w}', {\tt x}'{\tt 1y}, {\tt v}'{\tt
1w},{\tt x}{\tt 1y}'\}$, i.e. two elements which belong to the same
equivalence class being obtained via a permutation of both pairs $({\tt
w},{\tt w}')$ and $({\tt x},{\tt x}')$, and add together the
corresponding contributions, namely
{
\begin{align*}
&[{\tt y}:\phi:{\tt v}] \cdot
[{\tt w}:\psi:{\tt x}] \cdot
[{\tt y}':\,:{\tt v}'] \cdot
[{\tt w}':\,:{\tt x}'] 
\; - \;
[{\tt y}:\phi:{\tt v}] \cdot
[{\tt w}:\,:{\tt x}] \cdot
[{\tt y}':\,:{\tt v}'] \cdot
[{\tt w}':\psi:{\tt x}']
\; +
\cr
+\; 
&[{\tt y}:\phi:{\tt v}] \cdot
[{\tt w}':\psi:{\tt x}'] \cdot
[{\tt y}':\,:{\tt v}'] \cdot
[{\tt w}:\,:{\tt x}] 
\; - \;
[{\tt y}:\phi:{\tt v}] \cdot
[{\tt w}':\,:{\tt x}'] \cdot
[{\tt y}':\,:{\tt v}'] \cdot
[{\tt w}:\psi:{\tt x}] \ ;
\end{align*}
}
they clearly compensate each other and add up to zero.
Performing the permutation only in the pair $({\tt w},{\tt w}')$ and then
only in the pair $({\tt x},{\tt x}')$, we have instead
{
\begin{align*}
&[{\tt y}:\phi:{\tt v}] \cdot
[{\tt w}':\psi:{\tt x}] \cdot
[{\tt y}':\,:{\tt v}'] \cdot
[{\tt w}:\,:{\tt x}'] 
\; - \;
[{\tt y}:\phi:{\tt v}] \cdot
[{\tt w}':\,:{\tt x}] \cdot
[{\tt y}':\,:{\tt v}'] \cdot
[{\tt w}:\psi:{\tt x}']
\; +
\cr
+\; 
&[{\tt y}:\phi:{\tt v}] \cdot
[{\tt w}:\psi:{\tt x}'] \cdot
[{\tt y}':\,:{\tt v}'] \cdot
[{\tt w}':\,:{\tt x}] 
\; - \;
[{\tt y}:\phi:{\tt v}] \cdot
[{\tt w}:\,:{\tt x}'] \cdot
[{\tt y}':\,:{\tt v}'] \cdot
[{\tt w}':\psi:{\tt x}] \ ,
\end{align*}
}
which again add up to zero zero. The previous expressions show the
cancellations among 4 of 16 terms in the equivalence class; it is not
difficult to verify that the other cancellations take place working
also, jointly or separately, with the permutations of the other pairs
$({\tt v},{\tt v}')$ and $({\tt y},{\tt y}')$.  This completes the
proof that all contributions to
$\dinter{\phi\psi}Z-\dinter{\phi}\dinter{\psi}$ coming from elements
of every equivalence class in $\Sigma$ cancel out in pairs.

\subsection{Structure of the remainder}
\label{app.boh}

We present here a couple of Lemmas concerning the structure of the remainder
of our perturbative construction.  We recall that $\Phi=\varphi^\oplus$, and
$R = \rho^\oplus$, so that in particular $R = \sum_{j=0}^{N-1} \rho_j$ with
$\rho_j:= \rho\circ\tau^{j}$.

\begin{lemma}
\label{l.remainder}
In the original variables $x,\,y$ the seed $\varphi(x,y)$ is even in
the momenta $y$ and consequently the seed of the remainder $\rho(x,y)$
is odd both in the momenta $y$ and in the coordinates $x$.
\end{lemma}
\begin{proof}
We recall that the original Hamiltonian $H(x,y) = T(y) + V(x)$ is
obviously of even degree in the momenta and that also the iterative
scheme gives polynomials $\Phi_r(x,y)$ which are of even degree in the
momenta. Hence the Poisson bracket defining $\rho$ produces a polynomial
of odd degree both in the momenta and, being the whole degree even, in
the configuration variables.
\end{proof}

\begin{lemma}
\label{l.sum.corr.1}
It holds
\begin{equation*}
 \sum_{j=1}^{N-1}(N-j)\inter{\rho\rho_j} = \left\{
 \begin{aligned}
  &\frac{N}2\inter{\rho\rho_{N/2}} + N\sum_{j=1}^{N/2-1}\inter{\rho\rho_j} \ ,
  &\quad N {\rm \ even}
\cr
  &N\sum_{j=1}^{[N/2]}\inter{\rho\rho_j}\ ,
  &\quad N {\rm \ odd}
\end{aligned}
\right.
\end{equation*}
\end{lemma}
\begin{proof}
We provide the proof in the odd case, the other one being very similar.  We
have the following equalities:
\begin{displaymath}
\sum_{j=1}^{N-1}\inter{\rho\rho_j} =
\sum_{j=1}^{[N/2]}\inter{\rho\rho_j} +
\sum_{j=[N/2]+1}^{N-1}\inter{\rho\rho_{N-j}} =
2\sum_{j=1}^{[N/2]}\inter{\rho\rho_j}\ ;
\end{displaymath}
\begin{displaymath}
\sum_{j=1}^{N-1}j\inter{\rho\rho_j} =
\sum_{j=1}^{[N/2]}[j +  (N-j)]\inter{\rho\rho_{j}} =
N\sum_{j=1}^{[N/2]}\inter{\rho\rho_j}\ .
\end{displaymath}
Thus, subtracting the second line to $N$ times the first one, gives the thesis.
\end{proof}

\noindent
{\bf Acknowledgement}

We thank Dario Bambusi, Giancarlo Benettin, Andrea Carati, Luigi
Galgani, Vieri Matropiertro and Antonio Ponno for useful discussions
and comments. This research is partially supported by MIUR-PRIN
program under project 2007B3RBEY (``Teoria delle pertur\-bazioni ed
appli\-cazioni alla Mec\-canica Sta\-tistica ed
all'Elet\-trodina\-mica'') and project 2010JJ4KPA (``Te\-orie
geome\-triche e anali\-tiche dei sistemi Hamilto\-niani in dimensioni
finite e infi\-nite'').


\def\cprime{$'$} \def\i{\ii}\def\cprime{$'$} \def\cprime{$'$}

\end{document}